\documentclass[11pt]{amsart}

\usepackage{epigamath}

\usepackage[english]{babel}

\numberwithin{equation}{section}

\usepackage{enumerate}
\usepackage{amssymb}
\usepackage{array}
\usepackage{enumerate}
\usepackage{mathrsfs}
\usepackage{mathtools}
\usepackage{amsmath}

\usepackage[curve,all]{xy}
\xyoption{matrix}
 \xyoption{curve}
 \xyoption{color}
 \xyoption{line}
\xyoption{arc}

\usepackage[shortlabels]{enumitem}
\setlist[enumerate]{label=\rm{(\arabic*)}}
\setlist[enumerate,2]{label=\rm({\it\roman*}\,)}
\setlist[itemize]{label=\raisebox{0.25ex}{\tiny$\bullet$}}

\theoremstyle{plain}
\newtheorem{theoremA}{Theorem}

\newtheorem{theorem}{Theorem}[subsection]
\newtheorem{proposition}[theorem]{Proposition}
\newtheorem{lemma}[theorem]{Lemma}
\newtheorem{corollary}[theorem]{Corollary}

\theoremstyle{definition}
\newtheorem{definition}[theorem]{Definition}
\newtheorem{notation}[theorem]{Notation}

\theoremstyle{remark}
\newtheorem{remark}[theorem]{Remark}
\newcommand{\incl}[1][r]{\ar@<-0.2pc>@{^(-}[#1] \ar@<+0.2pc>@{-}[#1]}

\renewcommand{\P}{\mathbb{P}}
\newcommand{\E}{\mathcal{E}}

\newcommand{\Spec}{\mathrm{Spec}}

\newcommand{\V}{\mathcal{V}}
\renewcommand{\SS}{\mathcal{S}}
\newcommand{\PP}{\mathcal{P}}
\newcommand{\FF}{\mathcal{F}}

\newcommand{\TT}{\hat{\mathcal{S}}}
\newcommand{\p}{\mathbb{P}}
\newcommand{\s}[1]{s_{#1}}

\renewcommand{\k}{\mathrm{k}}

\newcommand{\A}{\mathbb{A}}

\newcommand{\C}{\mathbb{C}}
\newcommand{\F}{\mathbb{F}}
\newcommand{\N}{\mathbb{N}}

\newcommand{\M}{\mathcal{M}}
\newcommand{\U}{\mathcal{U}}
\renewcommand{\V}{\mathcal{V}}
\newcommand{\Z}{\mathbb{Z}}

\newcommand{\G}{\mathbb{G}}

\renewcommand{\O}{\mathcal{O}}
\newcommand{\PQ}{\mathbb{P}^1 \times \mathbb{P}^1}
\newcommand{\OP}{\mathcal{O}_{\mathbb{P}^1}}
\newcommand{\OPP}{\mathcal{O}_{\mathbb{P}^2}}

\newcommand{\car}{\mathrm{char}}

\newcommand{\OFa}{\mathcal{O}_{\mathbb{F}_a}}
\newcommand{\OFn}{\mathcal{O}_{\mathbb{F}_n}}

\DeclareMathOperator{\SL}{SL}
\DeclareMathOperator{\Aut}{Aut}
\DeclareMathOperator{\Hom}{Hom}

\DeclareMathOperator{\Autz}{Aut^{\circ}}
\DeclareMathOperator{\PGL}{PGL}

\DeclareMathOperator{\Pic}{Pic}

\DeclareMathOperator{\GL}{GL}
\DeclareMathOperator{\Bir}{Bir}

\def\isolow{\vbox to 0pt{\vss\hbox{$\scriptstyle\simeq$}\vskip-2pt}}
\newcommand{\iso}{\xrightarrow{\isolow}}

\newcommand{\mybinom}[2]{\bigl(\genfrac{}{}{0pt}{}{#1}{#2}\bigr)}

\EpigaVolumeYear{6}{2022} \EpigaArticleNr{23} \ReceivedOn{June 18,
2021}
\InFinalFormOn{July 22, 2022}
\AcceptedOn{August 13, 2022}

\title{Automorphisms of $\boldsymbol{\p^1}\!$-bundles over rational surfaces}
\titlemark{Automorphisms of $\p^1\!$-bundles over rational surfaces}

\author{J\'er\'emy Blanc}
\address{Universit{\"a}t Basel,
Departement Mathematik und Informatik,
Spiegelgasse 1, 4051 Basel,
Switzerland}
\email{Jeremy.Blanc@unibas.ch}

\author{Andrea Fanelli}
\address{Institut de Math\'ematiques de Bordeaux, 351 Cours de la Lib\'eration
33400 Talence, France}
\email{andrea.fanelli.1@u-bordeaux.fr }

\author{Ronan Terpereau}
\address{Universit\'{e} de Bourgogne Franche-Comt\'{e}, Institut de Math\'ematiques de Bourgogne, UMR 5584 CNRS, 9 avenue Alain Savary, B.P.~47870 - 21078 Dijon Cedex, France}
\email{ronan.terpereau@u-bourgogne.fr}

\authormark{J.~Blanc, A.~Fanelli and R.~Terpereau}

\AbstractInEnglish{In this paper, we provide the complete classification of $\P^1\!$-bundles over smooth projective rational surfaces whose neutral component of the automorphism group is maximal. Our results hold over any algebraically closed field of characteristic zero.} 



\acknowledgement{The authors acknowledge support by the Swiss National Science
Foundation Grant ``Algebraic subgroups of the Cremona groups'' 200021\_159921.
The second-named author is currently supported by the DFG grant ``Gromov-Witten Theorie, Geometrie und Darstellungen'' (PE 2165/1-2).
The third-named author is supported by the ANR Project FIBALGA ANR-18-CE40-0003-01 and is grateful to the Max-Planck-Institut f\"{u}r Mathematik of Bonn for the warm hospitality and support provided during part of the preparation of this work. The IMB receives support from  the EIPHI Graduate School (contract ANR-17-EURE-0002).}

\begin{document}


\maketitle

\begin{prelims}

\DisplayAbstractInEnglish

\bigskip

\DisplayKeyWords

\medskip

\DisplayMSCclass

\end{prelims}


\newpage

\setcounter{tocdepth}{1}

\tableofcontents


\section{Introduction}
\subsection{Aim and scope}
In this article, we work over an algebraically closed field $\k$ of characteristic zero (even if most of the partial results we obtain along the way are valid over any algebraically closed field $\k$, as we explain at the beginning of each section). We study $\p^1\!$-bundles $X\to S$ (always assumed to be locally trivial for the Zariski topology), where $S$ is a smooth projective rational surface (the smoothness condition is actually not a strong restriction; see \S~\ref{section:resolution}), and study the group scheme $\Aut(X)$ of automorphisms of $X$. This group scheme can have infinitely many components, but the connected component of the identity $\Autz(X)$ is a connected algebraic group; see~\cite{MO67}. The aim of this article is to study the pairs $(X,\Autz(X))$ and to classify these, up to birational conjugation. In particular, we describe the geometry of the pairs arising and the restriction of $X \to S$ to natural curves of $S$.

\bigskip

Our motivation for this study comes from the classification of connected algebraic subgroups of the Cremona group $\Bir(\p^3)$, stated by Enriques and Fano in~\cite{EF1898} and achieved by Umemura over the field $\k=\C$ in a series of four papers \cite{Ume80,Ume82a,Ume82b,Ume85} , using analytic methods. As explained in~\cite{MU83,Ume88}, these groups act on some minimal rational threefolds, and it turns out that most of the threefolds obtained are $\p^1\!$-bundles over rational surfaces. Our plan is to give a shorter geometric proof of the classification of Umemura, and this article is the first step in this direction (the second step is in~\cite{second_paper}). Instead of starting with the group action and trying to find threefolds where the group acts, we will directly study the possible varieties and their symmetries, then reduce to some simple varieties, and in the end compute the neutral component of their automorphism groups.

Our approach should be seen as the analogue of the following way to understand the classification of connected algebraic subgroups of the Cremona group $\Bir(\p^2)$. This classification was initiated by Enriques in~\cite{Enr1893} and can nowadays be easily recovered via the classification of smooth projective rational surfaces, as we now explain. One can conjugate any connected algebraic subgroup of $\Bir(\p^2)$ to a group of automorphisms of a smooth projective rational surface $S$. Contracting all $(-1)$-curves of $S$, we can moreover assume that $S$ is a minimal surface, \textit{i.e.}\ that~$S$ is isomorphic to the projective plane $\p^2$ or a Hirzebruch surface $\F_a$, with $a\ge 0$, $a\not=1$; see~\cite{Bea96} for a general survey on surfaces. One then checks that the groups of automorphisms obtained are maximal, as these have no orbit of finite size: this forbids the existence of equivariant birational maps towards other smooth projective rational surfaces. Every connected algebraic subgroup of $\Bir(\p^2)$ is thus contained in a maximal connected algebraic subgroup of $\Bir(\p^2)$, which is conjugate to the group $\Autz(S)$, where $S$ is a minimal smooth rational surface. See also~\cite{Bla09} for the classification of the (not necessarily connected) maximal algebraic subgroups of $\Bir(\p^2)$ with a similar approach.

The aim of our classification is then to proceed as in the case of surfaces and study minimal threefolds. There are many more such varieties in dimension $3$ than in dimension $2$, and not all of them  yield maximal algebraic subgroups of $\Bir(\p^3)$. Moreover, some maximal connected algebraic subgroups of $\Bir(\p^3)$ are realised by infinitely many minimal threefolds. Our aim is then to understand the geometry of the minimal threefolds obtained in this way and the equivariant birational maps between them.
	
\bigskip

Another motivation consists in unifying some well-known results (most of them over the field of complex numbers, using topological arguments) on $\p^1\!$-bundles over minimal rational surfaces, \textit{i.e.}~the projective plane $\p^2$ and the Hirzebruch surfaces $\F_a$, and to obtain these results from the perspective of automorphisms groups. See for instance Corollaries~\ref{Coro:ABV} and~\ref{Cor:SmallIsoSchwarzenberger} and Remark~\ref{Rem:kCold}.

\bigskip

We also provide moduli spaces parametrising the $\p^1\!$-bundles $X\to \F_a$ over Hirzebruch surfaces having no jumping fibre (see Corollary~\ref{Cor:modspace}), \textit{i.e.}~the $\p^1\!$-bundles such that all fibres of the natural morphism $X\to\p^1$ (induced by the structure morphism $\F_a \to \p^1$) are Hirzebruch surfaces $\F_b$ for a fixed $b$ not depending on the fibre. These $\p^1\!$-bundles can also be described using exact sequences of vector bundles (see Corollary~\ref{Coro:ABV}). The action of $\Autz(\F_a)$ on the moduli spaces, depending on some natural numerical invariants (see Definition~\ref{Defi:NumInv}), is described explicitly; see \S~\ref{section:action on the moduli space}, and more precisely Corollary~\ref{cor action on Mabc}.

Our approach uses very basic tools and is aimed to be easy to follow by any interested reader, not necessarily expert. Most of the time we give proofs which do not use any known results, and we make a reference to the literature when we re-prove a known fact.

\subsection{Summary of the classification}
\begin{definition}\label{Defi:SquareEtc}
Let $\pi\colon X\to S$ and $\pi'\colon X'\to S'$ be two $\p^1\!$-bundles over two smooth projective rational surfaces $S$ and $S'$. A birational map $\varphi\colon X\dasharrow X'$ is said to be
\begin{enumerate}
\item
a \emph{square birational map $($resp.\ square~isomorphism/square automorphism$)$} if there exists a birational map $\eta\colon S\dasharrow S'$ such that $\pi'\varphi=\eta\pi$ (and if $\varphi$ is resp.~an isomorphism/automorphism) 
\[\xymatrix@R=4mm@C=2cm{
    X \ar@{-->}[r]^{\varphi} \ar[d]_{\pi}  & X' \ar[d]^{\pi'} \\
    S \ar@{-->}[r]_{\eta} & S'\rlap{;}
  }\] 
   we say in these cases that $\varphi$ is \emph{above} $\eta$;
\item
a \emph{birational map $($resp.~isomorphism/automorphism$)$ of $\,\p^1$-bundles} if $S=S'$ and $\pi'\varphi=\pi$ (and if $\varphi$ is resp.~an isomorphism/automorphism)
\[\xymatrix@R=4mm@C=2cm{
    X \ar@{-->}[rr]^{\varphi} \ar[rd]_{\pi}  && X'\rlap{;} \ar[ld]^{\pi'} \\
    &S & 
  }\]
\item
  \emph{$\Autz(X)$-equivariant} if $\varphi \Autz(X)\varphi^{-1}\subset \Autz(X')$ (or, equivalently, if $\varphi \Autz(X)\varphi^{-1}\subset \Aut(X')$). 
\end{enumerate}
As the definition depends on $\pi$ and $\pi'$, and not only on $X$ and $X'$, we will often write $\varphi\colon (X,\pi)\dasharrow (X',\pi')$ and say that $(X,\pi)$ and $(X',\pi')$ are, respectively,  square birational/square isomorphic/birational $\p^1\!$-bundles/isomorphic $\p^1\!$-bundles if $\varphi$ satisfies the corresponding condition.
\end{definition}

\begin{remark}
In the above definition, every element of $\Autz(X)$ is a square automorphism  (see Lemma~\ref{blanchard}) but not necessarily a birational map of $\p^1\!$-bundles.
\end{remark}

\begin{definition} \label{def:max}
Let $\pi\colon X\to S$ be a $\p^1\!$-bundle over a smooth projective rational surface $S$. We say that $\Autz(X)$ is \emph{maximal} if for each $\Autz(X)$-equivariant square birational map $\varphi\colon (X,\pi)\dasharrow (X',\pi')$, we have $\varphi \Autz(X)\varphi^{-1}=\Autz(X')$. If we moreover have $(X',\pi')\simeq (X,\pi)$ (resp.\ $\varphi$ is an isomorphism of $\p^1\!$-bundles) for each such $\varphi$, we say that the $\p^1\!$-bundle $(X,\pi)$ is \emph{stiff} (resp.\ \emph{superstiff}).
\end{definition}

\begin{remark}
This definition depends on $X$ and $\pi$, and not only on $X$. For instance, taking $X=\p^1\times \F_1$ and two standard $\p^1\!$-bundle structures $\pi\colon X\to \p^1\times \p^1$ and $\pi'\colon X\to \F_1$, $\Autz(X)$ is maximal with respect to $\pi$ but not with respect to $\pi'$. (The pairs $(X,\pi)$ and $(X,\pi')$ correspond, respectively, to $\FF_0^{1,0}$ and $\FF_1^{0,0}$, see Definition~\ref{def Fabc}, so this observation follows from Theorem~\ref{Thm:MainA}). 
\end{remark}

\begin{remark}
To the best of our knowledge, the notion of stiff/superstiff is new. It is analogous to the notion of equivariant birational rigidity/superrigidity for Mori fibre spaces but is not equivalent since here we only consider $\p^1\!$-bundles. Moreover, birational rigidity for Mori fibre spaces is always up to squares, while stiffness also detects these birational maps; see~\cite{cor00} and~\cite{Puk13} to know more about the notions of rigidity and superrigidity for Mori fibre spaces.
\end{remark}

The next statement, which summarises most of our work, is our main result (the notation is explained after Theorem~\ref{Thm:MainB}).

\begin{theoremA}\label{Thm:MainA}
Let $\pi\colon X\to S$ be a $\p^1\!$-bundle over a smooth projective rational surface $S$. Then, there exists an $\Autz(X)$-equivariant square birational map $(X,\pi)\dasharrow (X',\pi')$ such that $\Autz(X')$ is maximal. Moreover, the group $\Autz(X)$ is maximal if and only if $(X,\pi)$ is square isomorphic to one of the following:
\begin{center}\begin{tabular}{llllllll}
$(a)$& a decomposable &$\p^1\!$-bundle & $\FF_a^{b,c}$&\hspace{-0.3cm}$\longrightarrow$& \hspace{-0.2cm}$\F_a$& with $a,b\ge 0$, $a\not=1$, $c\in \Z$, \\
&&&&&&$c\le 0$ if $b=0$,\\
&&&&&&and where $a=0$ or $b=c=0$\\
&&&&&&or $-a<c<ab$;\\
$(b)$& a decomposable &$\p^1\!$-bundle &$\PP_b$&\hspace{-0.3cm}$\longrightarrow$& \hspace{-0.2cm}$\p^2$& for some $b\ge 0$;\\
$(c)$& an Umemura &$\p^1\!$-bundle &$\U_a^{b,c}$&\hspace{-0.3cm}$\longrightarrow$& \hspace{-0.2cm}$\F_a$& for some $a,b\ge 1, c\ge 2$,\\
&&&&&& with $c-ab<2$ if $a \geq 2$,\\
&&&&&& and $c-ab<1$ if $a=1$;\\
$(d)$ &a Schwarzenberger\! &$\p^1\!$-bundle &$\SS_b$&\hspace{-0.3cm}$\longrightarrow$& \hspace{-0.2cm}$\p^2$& for some $b\ge 1$; or\\
$(e)$ &a &$\p^1\!$-bundle &$\V_b$&\hspace{-0.3cm}$\longrightarrow$& \hspace{-0.2cm}$\p^2$& for some $b\ge 2$.\end{tabular}\end{center}
\end{theoremA}

\begin{remark} We keep the notation of Theorem~\ref{Thm:MainA}: 
\begin{itemize}
\item  $\Autz(\FF_a^{b,c})$ is described in Lemma~\ref{Lem:AutDecOnAutFa} and Remark~\ref{rk aut decompo bundles over Fa};
 \item $\Autz(\PP_b)$ is described in Lemma~\ref{lemm:surjectiveDecP2} and Remark \ref{rk: description auto of P1-bundle of dec bundles over P2};
 \item $\Autz(\U_a^{b,c})$ is described in Lemma~\ref{Lem:SurjectiveActionUmemura} and Remark~\ref{rk aut vert Umemura bundles};  
 \item $\Autz(\SS_b)$ is described in Lemma~\ref{Lemm:SchwarzJumpLines}\ref{AutoSchwarzb}; and
 \item $\Autz(\V_b)$ is described in Lemma~\ref{Lem:UMeF1P2}\ref{AutV1inP2}--\ref{AutoV1lift}.
 \end{itemize}
\end{remark}

Contrary to the dimension $2$ case, there are many $\p^1\!$-bundles $X \to S$ with  maximal $\Autz(X)$ which are birationally conjugate.
This means that the $\p^1\!$-bundles of Theorem~\ref{Thm:MainA} are not always stiff. The next result describes all the possible links between such $\p^1\!$-bundles. The construction of such links is detailed in \S~\ref{Sec:Classif}.

\begin{theoremA}\label{Thm:MainB}
The $\p^1\!$-bundles of Theorem~{\rm\ref{Thm:MainA}} are superstiff only in the cases: 
\begin{enumerate}[label=\rm{($\alph*$)}]
\item  $\FF_a^{b,c}$ with $a=0$ or $b=c=0$;
\item $\PP_b$ for $b\ge 0$; and
\item $\SS_1\simeq \P(T_{\P^2})$.
\end{enumerate}
In the other cases, the equivariant square birational maps between the $\p^1\!$-bundles of Theorem~{\rm\ref{Thm:MainA}} are given by compositions of square isomorphisms of $\,\p^1\!$-bundles and of birational maps appearing in the following list: 
\begin{enumerate}
\item
for all integers $a,b\ge 0$, $c\in \Z$ with $a\not=1$, $-a<c<0$, an infinite sequence of equivariant birational maps of $\p^1\!$-bundles
\[\FF_a^{b,c}\dasharrow \FF_a^{b+1,c+a}\dasharrow \cdots \dasharrow\FF_a^{b+n,c+an}\dasharrow \cdots;\]
\item
for all integers $a,b\ge 1$ with $(a,b)\not=(1,1)$, an infinite sequence of equivariant birational maps of $\p^1\!$-bundles
\[\U_a^{b,2}\dasharrow \U_a^{b+1,2+a}\dasharrow \cdots \dasharrow\U_a^{b+n,2+an}{\dasharrow} \cdots; \]
\item
for each $b\ge 2$, a birational involution $\SS_b\dasharrow \SS_b$;
\item
for each $b\ge2$, the equivariant birational morphisms $\U_1^{b,2}\to \V_b$ obtained by contracting the preimage of the $(-1)$-curve of $\,\F_1$ onto the fibre of a point of $\,\p^2$ in $\V_b$.
\end{enumerate}
\end{theoremA}

In Theorems~\ref{Thm:MainA} and~\ref{Thm:MainB}, decomposable $\p^1\!$-bundles are simply the projectivisations of decomposable rank $2$ vector bundles. The $\p^1\!$-bundles over $\F_a$ and $\p^2$ are particularly easy to describe (see \S\S~\ref{SubSec:DecFa} and~\ref{decomposable_over_P2}). The Schwarzenberger $\p^1\!$-bundles over $\p^2$ are projectivisations of the classical rank $2$ vector bundles of the same name (see below). We describe  the notation here.

\begin{definition}[Schwarzenberger $\p^1\!$-bundles]\label{def:Schwarz}Let $b\ge -1$ be an integer, and let $\kappa\colon \p^1 \times \p^1 \to \p^2$ be the $(2:1)$-cover defined by
\[\begin{array}{rccc}
\kappa\colon & \p^1\times \p^1 & \to &\p^2\\
& \big([y_0:y_1],[z_0:z_1]\big) &\mapsto &[y_0 z_0:y_0 z_1+y_1 z_0:y_1z_1],\end{array}\]whose branch locus is the diagonal $\Delta\subset\p^1 \times \p^1$ and whose ramification locus is the smooth conic $\Gamma=\{ [X:Y:Z] \mid Y^2=4XZ\}\subset\p^2.$
The \emph{$b$-th Schwarzenberger $\p^1\!$-bundle} $\SS_b\to \p^2$ is the $\p^1\!$-bundle defined by
\[\SS_b=\P\big(\kappa_* \O_{\p^1 \times \p^1}(-b-1,0)\big)\to \p^2.\]
\end{definition}

Note that $\SS_b$ is the projectivisation of the classical Schwarzenberger vector bundle $\kappa_* \O_{\p^1 \times \p^1}(-b-1,0)$ introduced in~\cite{Sch61}. Moreover, there is a natural family of lines of $\p^2$ whose preimage by $\SS_b\to \p^2$ is isomorphic to $\F_b$  for each $b\ge 0$ (by Lemma~\ref{Lemm:SchwarzJumpLines}\ref{FnLineSchwarz}). This explains the shift in the notation.

The families of Umemura $\p^1\!$-bundles $\U_a^{b,c}$, introduced by Umemura in~\cite[\S~10]{Ume88}, need a bit more notation to be described. This is done in \S~\ref{Sec:Umemurabundles}. The $\p^1\!$-bundles $\V_1^b\to \p^2$ is a family of $\p^1\!$-bundles such that $\Autz(\V_b)$ is maximal and birationally conjugate  to certain $\Autz(\U_1^{b,c})$, but with an action on $\p^2$ fixing a point; see Lemma~\ref{Lem:UMeF1P2} for the precise relation between the two families.

\begin{remark}
Let us mention that the $\p^1\!$-bundles that appear in Theorem~\ref{Thm:MainA} are toric varieties (families (a) and (b)) or minimal smooth compactifications of $\SL_2/H$ with $H \subseteq \SL_2$ a finite subgroup (families (c), (d), and (e) and certain members of families (a) and (b)). Here we call \emph{minimal smooth compactification} of $\SL_2/H$ an $\SL_2$-equivariant embedding $\SL_2/H \hookrightarrow X$ such that $X$ is a smooth projective variety and any $\SL_2$-equivariant birational morphism $X \to X'$, with $X'$ smooth, is an isomorphism; these $\SL_2$-threefolds were studied and classified by Nakano in~\cite{Nak89} (over $\C$) as an application of Mori theory. The finite subgroups of $\SL_2(\C)$ are well known: there are the cyclic groups, the binary dihedral groups, and the binary polyhedral groups. When $H$ is non-cyclic, $\SL_2/H$ admits a unique minimal smooth compactification, namely the Fano threefolds $V_5$, $V_{22}^{MU}$, and the smooth quadric $Q_3$ for the binary polyhedral groups (see \cite[Proposition~2.5]{Nak89}), and the Schwarzenberger $\p^1\!$-bundles or $\P^2$ for the binary dihedral groups (see \cite[Theorem~2.8]{Nak89}). But when $H$ is cyclic,  $\SL_2/H$ admits several non-isomorphic minimal smooth compactifications including the $\p^1\!$-bundles corresponding to families (c) and (e) in Theorem~\ref{Thm:MainA} (see~\cite[Theorem~2.31]{Nak89} for the complete list).
\end{remark}

\subsection{Comparison with Umemura's classification}
Each of the maximal connected algebraic subgroups of $\Bir(\p^3_\C)$ given in~\cite[Theorem~2.1]{Ume85}
acts on some Mori fibre spaces (as explained in~\cite{Ume88}). The cases (P1), (P2), (E1), (E2) correspond to Fano threefolds, the cases (J10), (J12) to Del Pezzo fibrations ($\p^2$-bundles and quadric fibrations), and all other cases (J1)--(J9), (J11) correspond to $\p^1\!$-bundles. We now explain to which of the cases $(a)$--$(d)$ of Theorem~\ref{Thm:MainA} these correspond.

Note that $\FF^{b,c}_a=\P(\OFa \oplus \OFa(-b\s{a}+cf))\to \F_a$ (see Definition~\ref{def Fabc}) and that $\PP_b=\P(\OPP(b) \oplus \OPP)=\P(\OPP \oplus \OPP(-b))\to \p^2$ (see Definition~\ref{def Pb}).

(J1) is $\Autz(\p^2\times \p^1)$ and is thus given by $\PP_0\to \p^2$.

(J2) is $\Autz(\p^1\times \p^1\times \p^1)$ and is thus given by $\FF_0^{0,0}\to \F_0$.

(J3) is $\Autz(\F_m\times \p^1)$, with $m\ge 2$, which is  given either by $\FF_m^{0,0}\to \F_m$ or by $\FF_0^{0,m}\to \F_0$, which is square isomorphic to $\FF_0^{m,0}\to \F_0$ (but not isomorphic as a $\p^1\!$-bundle).

(J4) is $\Autz(\PGL_3/B)\simeq \PGL_3$ and is thus equal to $\Autz(\SS_1)$, as $\SS_1\to \p^2$, $\P(T_{\P^2})\to\p^2$, and $\PGL_3/B \to \PGL_3/P$ are isomorphic, by Remark~\ref{rk tangent bundle}.

(J5) is $\Autz(\PGL_2/D_{2n})$, with $n\ge 4$, and is thus equal to $\Autz(\SS_b)$, with $b=n-1\ge 3$, by Remark~\ref{rk PGLn mod Dn}.

(J6) is $\Autz(L_{m,n})$, where $L_{m,n}\to \F_0$ is a $\p^1\!$-bundle of bidegree $(m,n)$, with $m \geq 2$, $n\leq -2$ (see~\cite[Theorem~2.1(J6)]{Ume85} and~\cite[\S~5]{Ume88}). It thus corresponds to $\FF_0^{m,-n}\to \F_0$, given in $(a)$.

(J7) is equal to $\Autz(\mathrm{J}_m)=\Autz(\mathrm{J}'_m)$, with $m\ge 2$, where the $\p^1\!$-bundle $\mathrm{J}_m\to \p^2$ defined in~\cite[\S~6]{Ume88} is a compactification of $\mathrm{J}_m'$ of~\cite[Theorem~2.1(J7)]{Ume85}, isomorphic to $\PP_b\to \p^2$ with $m=b$, given in $(b)$.

(J8) is $\Autz(L_{m,n})$, where $L_{m,n}\to \F_0$ is a $\p^1\!$-bundle of bidegree $(m,n)$, with $m\ge n\ge 1$ (see~\cite[Theorem~2.1(J8)]{Ume85} and~\cite[\S~7]{Ume88}). It thus corresponds to $\FF_0^{m,n}\to \F_0$, which is square isomorphic to $\FF_0^{n,m}\to \F_0$, both given in $(a)$.

(J9) is $\Autz(\mathrm{F}_{m,n}')$, for some integers $m>n\ge 2$ (see~\cite[Theorem~2.1(J9)]{Ume85}), and is also equal to $\Autz(\mathrm{F}_{m,n}^k)$ for each integer $k\ge \lfloor {m}/{n}\rfloor$, where $\mathrm{F}_{m,n}^k=\P(\OFn \oplus \OFn(-k\s{-n}-mf))=\P(\OFn \oplus \OFn(-k\s{n}+(nk-m)f))=\FF_n^{k,nk-m}\to \F_n$, given in $(a)$.

(J11) is $\Autz(\mathrm{E}_m'^{\,l})$, where $m=1$, $l\ge 3$ or $m,l\ge 2$ (see \cite[Theorem~2.1(J11)]{Ume85}). Then $\mathrm{E}_m'^{\,l}$ admits a family of compactifications given in~\cite[\S~10]{Ume88} and depending on a parameter $j \geq l$; these compactifications correspond to the Umemura bundles $\U_a^{b,c}\to \F_a$, with $a=m$, $b=j$, $c=(j-l)m+2$.
 
\medskip

Let us note that the family $(e)$ in Theorem~\ref{Thm:MainA} was overlooked in the work of Umemura. These have to correspond to maximal algebraic subgroups of $\Bir(\p^3)$ and should appear in~\cite[\S~10]{Ume88}.

Some of the elements of our list do not appear in~\cite{Ume85,Ume88}. This is sometimes because these do not correspond to maximal algebraic subgroups of $\Bir(\p^3)$, as we can embed the groups in larger groups of automorphisms of Mori fibre spaces, or it is because they are equivalent to other elements of the list, by a birational map not preserving the $\p^1\!$-bundle structure. This is for instance the case of $\SS_2\to \p^2$, which is $\Aut(\SS_2)$-equivariantly birational to $\p^3$, or of $\FF_0^{0,1}\simeq\F_1\times \p^1\to \p^1\times \p^1$, which is not maximal as it is equivariantly birational to $\p^2\times \p^1$. There are many other cases, which are studied in~\cite{second_paper}.

\subsection{Overview of the proof}
Starting with a $\p^1\!$-bundle $\hat\pi\colon\hat{X}\to \hat{S}$ over a smooth projective rational surface $\hat S$, there is a birational morphism $\eta\colon \hat{S}\to S$, where $S$ is a Hirzebruch surface $\F_a$ or the projective plane $\p^2$. Applying Lemma~\ref{Lem:GoingdownMinimalsurfaces}, we obtain a square birational map $\psi\colon (\hat{X},\hat\pi)\dashrightarrow (X,\pi)$, unique up to isomorphism.  Moreover, $\psi$ is $\Autz(X)$-equivariant. We then need to study $\p^1\!$-bundles over Hirzebruch surfaces or over the projective plane.

Section~\ref{Sec:P1bundlesHirz} concerns the case of $\p^1\!$-bundles over Hirzebruch surfaces $\pi\colon X\to \F_a$, with $a\ge 0$. We denote by $\tau_a\colon \F_a\to \p^1$ a $\p^1\!$-bundle structure on $\F_a$, and we study the surfaces $S_p=(\tau_a\pi)^{-1}(p)$ with $p\in \p^1$. The situation is described by Proposition~\ref{Prop:NoMoreJumpingfibre}; we explain it now. There is an integer $b\ge 0$ such that $S_p\simeq \F_b$ for a general $p\in \p^1$. If this holds for all points of $\p^1$, we  in fact have a $\F_b$-bundle $\tau_a\pi\colon X\to \p^1$: we say that there is no \emph{jumping fibre}. If a special point $p\in \p^1$ is such that $S_p\simeq \F_{b'}$ for some $b'\not=b$, then $b'>b$ and we can blow up the exceptional curve of $S_p\simeq \F_{b'}$ and contract the strict transform of $S_p$, in an $\Autz(X)$-equivariant way. After finitely many steps, we reduce to the case where there is no jumping fibre.

We then associate to any $\p^1\!$-bundle $\pi\colon X\to \F_a$ with no jumping fibres two integers $b,c\in \Z$ such that $b\ge 0$, and $c\le 0$ if $b=0$. The integer $b$ is the one such that $\tau_a\pi\colon X\to \p^1$ is a $\F_b$-bundle, and the integer $c$ can be seen either with the transition function of $\pi: X \to \F_a$ or using exact sequences associated to the rank~$2$ vector bundle corresponding to $\pi$, with the following definition (Corollary~\ref{Coro:NumericalInv} shows the equivalence between the two points of view).

\begin{definition}\label{Defi:NumInv}
Let $a,b,c\in \Z$ be such that $a,b\ge 0$ and $c\le0$ if $b=0$.
We say that a $\p^1\!$-bundle $\pi\colon X\to \F_a$ \emph{has numerical invariants $(a,b,c)$} if it is the projectivisation of a rank $2$ vector bundle 
$\mathcal{E}$ which fits in a short exact sequence
\[ 0 \to \O_{\F_a} \to \mathcal{E} \to \O_{\F_a}(-b\s{a}+cf) \to 0,\]
where $f,\s{a}\subset \F_a$ are a fibre and a section of self-intersection $a$ of $\tau_a\colon \F_a\to\p^1$.
\end{definition}

In Proposition~\ref{Prop:EquivABC}, we show that a $\p^1\!$-bundle $\pi\colon X\to \F_a$ with numerical invariants $(a,b,c)$ has no jumping fibre and that every $\p^1\!$-bundle $\pi\colon X\to \F_a$ with no jumping fibre has numerical invariants $(a,b,c)$, for some uniquely determined integers $b,c\in \Z$, with $b\ge 0$, and $c\le0$ if $b=0$. The above numerical invariants are thus really invariant under isomorphisms.

We then prove (see Corollary~\ref{Coro:IsoClass}) that every $\p^1\!$-bundle $X \to \F_a$ with numerical invariants $(a,b,c)$ is decomposable if $b=0$ or $c<2$, and construct a moduli space $\M_{a}^{b,c}$, which is isomorphic to a projective space (see Remark~\ref{rk: dim mod spaces}), parametrising the non-decomposable $\p^1\!$-bundles $X \to \F_a$ with numerical invariants $(a,b,c)$ when $b\ge 1$ and $c\ge 2$ (see Corollary~\ref{Cor:modspace}).

The group $\Autz(\F_a)$ acts naturally on this moduli space, via an algebraic action, which is detailed in \S~\ref{section:action on the moduli space}. Using this action, we are able to describe the geometry of $\p^1\!$-bundles $X\to\F_a$ (see Proposition~\ref{prop decompo and Umemura bundles}). We prove in particular that if no surface $S_p$ (with the notation above) is invariant by $\Autz(X)$, then $\pi$ is isomorphic to a decomposable $\p^1\!$-bundle, to an Umemura $\p^1\!$-bundle, or to a $\p^1\!$-bundle $\TT_b\to \p^1\times \p^1$, which is obtained by pulling back the Schwarzenberger bundle $\SS_{b}\to\p^2$ via the double cover $\kappa \colon\p^1\times \p^1\to \p^2$ defined above (see Lemma~\ref{Lemm:LiftSchwarzIsSchwarzHat}); the last two cases correspond to natural elements of the moduli spaces $\M_a^{b,c}$ fixed by $\Autz(\F_a)$. 
 
In the case $\TT_b\to \p^1\times \p^1$ and in the case where a fibre $S_p$ is invariant, we can reduce to the case of decomposable $\p^1\!$-bundles over $\F_a$ (again by Proposition~\ref{prop decompo and Umemura bundles}). 
 
Section~\ref{section: P1 bundles over P2} concerns $\p^1\!$-bundles over $\p^2$. Despite the fact that the geometry of such bundles is quite rich and complicated (see \textit{e.g.}~\cite{OSS11} for an overview when $\k=\C$), our approach allows us to give a quite simple proof in this case, using the work done in \S~\ref{Sec:P1bundlesHirz} for $\p^1\!$-bundles over Hirzebruch surfaces. We first study the Schwarzenberger $\p^1\!$-bundles $\SS_b\to \p^2$ (see Lemmas~\ref{Lemm:TransSchwarz},~\ref{Lemm:LiftSchwarzIsSchwarzHat}, and~\ref{Lemm:SchwarzJumpLines} and Corollary~\ref{Cor:SmallIsoSchwarzenberger}). We then take a $\p^1\!$-bundle $\pi\colon X\to \p^2$ and denote by $H\subset \Aut(\p^2)$ the image of $\Autz(X)$. If $H$ fixes a point, we blow up the fibre of this point and reduce our study to the case of $\p^1\!$-bundles over $\F_1$. Otherwise, using the structure of $\Aut(\p^2)$, we see that either $H=\Aut(\p^2)=\PGL_3$ or $H\subset\Aut(\p^2,C)=\{g\in \Aut(\p^2)\mid g(C)=C\}$, for some smooth curve $C$ which is a line or a conic (see Lemma~\ref{SubgroupsPGL3}). The case where $C$ is a line cannot happen (see Proposition~\ref{Prop:P1bundlesP2}), and the case where $C$ is a conic corresponds to the Schwarzenberger case; this is proven by using the double cover $\kappa\colon \p^1\times \p^1\to \p^2$ of Definition~\ref{def:Schwarz} and the results on $\p^1\!$-bundles over Hirzebruch surfaces.  The case where $H=\PGL_3$ corresponds to the decomposable $\p^1\!$-bundles over $\p^2$ or to the special Schwarzenberger $\p^1\!$-bundle $\SS_1$, isomorphic to the projectivised tangent bundle $\p(T_{\p^2})$. Again, this is proven by  a reduction to the case of $\p^1\!$-bundles of $\F_1$ (studied in \S~\ref{Sec:P1bundlesHirz}) by blowing up a point of~$\p^2$.

In \S~\ref{Sec:Classif}, we prove Theorems~\ref{Thm:MainA} and~\ref{Thm:MainB}, which achieves our classification. Once we have reduced our study to decomposable $\p^1\!$-bundles over $\F_a$ or $\p^2$ or to Umemura or Schwarzenberger $\p^1\!$-bundles (see Proposition~\ref{Prop:FourCases}), we study birational maps of $\p^1\!$-bundles between elements of these four families, which are in fact obtained by elementary links centred at invariant curves (see Lemma~\ref{Lemm:DecElLinks}), square isomorphisms, and special contractions from $\p^1\!$-bundles over $\F_1$ to $\p^1\!$-bundles over $\p^2$ (see Remark~\ref{Rem:Overview5}). The study of the possible links is made on each family, by describing the possible invariant curves. These are naturally contained in the preimage of invariant curves of the surface $S$ over which we take the $\p^1\!$-bundles and are most of the time obtained by a negative curve in the fibres of smooth rational curves.

\subsection{Acknowledgments.}
We are grateful to the anonymous referee for the careful reading of a previous  version of this paper and for their helpful comments.
We would like to thank Michel Brion, Paolo Cascini, Ciro Ciliberto, Daniele Faenzi, Enrica Floris, Anne-Sophie Kaloghiros, Massimiliano Mella, Lucy Moser-Jauslin, Boris Pasquier, Francesco Russo, and Jean Vall\`es for interesting discussions related to this work. Thanks to Mateusz Michalek for indicating the proof of the combinatoric Lemma~\ref{Lemm:Combinatoric} to us. Thanks also to the MathOverflow community for its help regarding the proofs of Lemmas~\ref{lem:subgroups of PGL2} and~\ref{SubgroupsPGL3}.

\section{Preliminaries} \label{Sec:Preliminaries}
All the results in \S~2 are valid over an algebraically closed field $\k$ of arbitrary characteristic.

\subsection{Blanchard's lemma}
We recall a result due to Blanchard~\cite[\S~I.1]{Bla56} in the setting of complex geometry, whose proof has been adapted to the setting of algebraic geometry by Brion, Samuel, and Uma.

\begin{lemma}[\textit{cf.} \protect{\cite[Proposition~4.2.1]{BSU13}}] \label{blanchard}
Let $f\colon X \to Y$ be a proper morphism between algebraic varieties such that $f_*(\O_X)=\O_Y$. If a connected algebraic group $G$ acts on $X$, then there exists a unique action of $G$ on $Y$ such that $f$ is $G$-equivariant.
\end{lemma}

\subsection{Resolution of indeterminacies} \label{section:resolution}
The next result implies that if $\pi:X \to S$ is a $\p^1\!$-bundle over a singular surface, then there exists an $\Autz(X)$-equivariant square birational map $(X,\pi) \dashrightarrow (X',\pi')$, with $\pi':X' \to S'$ a $\p^1\!$-bundle over a smooth surface. Therefore, it is enough to consider the $\p^1\!$-bundles over smooth surfaces to determine all the maximal automorphism groups $\Autz(X)$ in the sense of Definition~\ref{def:max}.

\begin{lemma} \label{lemma:resolution}
Let $\pi\colon X \to S$ be a $\p^1\!$-bundle over a singular projective surface $S$, and let $G \subset \Autz(X)$ be a connected algebraic subgroup. Then there exist a $\p^1\!$-bundle $\pi'\colon X' \to S'$, equipped with a $G$-action, and $G$-equivariant birational morphisms $\eta\colon S' \to S$  and $\hat\eta\colon  X' \to X$ such that $\eta$ is a resolution of singularities and the  diagram below is cartesian; in particular, ${\hat\eta}^{-1} G \hat\eta$ is a subgroup of $\Autz(X')$. 
\[\xymatrix@R=4mm@C=2cm{
    X' \ar[r]^{\hat\eta} \ar[d]_{\pi'}  & X \ar[d]^{\pi} \\
    S' \ar[r]_{\eta} & S
  }\]  
\end{lemma}

\begin{proof}
By Lemma~\ref{blanchard}, the group $G$ acts biregularly on $S$.
Since $G$ is connected, we can solve the singularities of $S$ in a $G$-equivariant way by repeatedly alternating between normalization and blowing-up of the singular points (see for instance~\cite{Art86}). We denote by $\eta\colon S'\to S$ a $G$-equivariant resolution of singularities of $S$ and define $\pi'\colon X':=X \times_S S' \to S'$ to be the pull-back of $\pi$ along $\eta\colon S' \to S$. The pull-back of a $\p^1\!$-bundle along a morphism of schemes is again a $\p^1\!$-bundle (indeed, all the $\p^1\!$-bundles that we consider are Zariski locally trivial).
Also, since $S' \to S$ is $G$-equivariant, $G$ acts on $X'$, and $\hat\eta: X' \to X$ is $G$-equivariant. The last statement follows from the fact that $\eta$, and thus $\hat\eta$, is birational. 
\end{proof}

In the following, we will only consider the $\p^1\!$-bundles over smooth surfaces; in particular, we will not describe the $\p^1\!$-bundles $X \to S$ over singular surfaces such that $\Autz(X)$ is maximal. There are many such $\p^1\!$-bundles, obtained from those of Theorem~\ref{Thm:MainA} by contracting the negative curve on a Hirzebruch surface onto a singular  point. The obtained threefolds have, however, singularities which are of codimension~$2$, so it is natural to avoid them when working in birational geometry with the classical minimal model program (where varieties have terminal singularities). We also recall that a $\P^n$-bundle over a smooth variety is isomorphic to the projectivisation of a rank $n+1$ vector bundle over this variety (see \textit{e.g.}~\cite[Exercise~II.7.10(c)]{Har77}); this well-known fact will be used implicitly throughout the rest of the article.

\subsection{The descent lemma}
The following simple observation will be often used later. It  explains that $\p^1\!$-bundles $\pi: X \to S$ over smooth projective surfaces are uniquely determined by a description on an open subset with finite complement. This is for instance useful over $\p^2$ or Hirzebruch surfaces, where then only the restriction of $\pi$ to two open subsets isomorphic to $\A^2$ is needed to describe the whole $\p^1\!$-bundle.

\begin{lemma}\label{Lem:Extension}
Let $S$ be a smooth projective surface, let $\Omega\subset S$ be a finite set, and let $g\in \Aut(S)$ be an automorphism that satisfies $g(\Omega)=\Omega$. Let $\pi_1\colon X_1\to S$ and $\pi_2\colon X_2\to S$ be two $\p^1\!$-bundles, and let $\hat{g}\colon X_1\dasharrow X_2$ be a birational map that restricts to an isomorphism $(\pi_1)^{-1}(S\setminus \Omega)\iso (\pi_2)^{-1}(S\setminus \Omega)$ and satisfies $\pi_2\hat{g}=\pi_1 g$. Then $\hat{g}$ is an isomorphism of varieties $X_1\iso X_2$:\[\xymatrix@R=4mm@C=2cm{
    X_1 \ar[r]^{\hat{g}}_{\simeq} \ar[d]_{\pi_1}  & X_2 \ar[d]^{\pi_2} \\
    S \ar[r]^{g}_{\simeq} & S\rlap{.}
  }\] 
In particular, if $g$ is the identity, then $\hat{g}$ is an isomorphism of $\,\p^1\!$-bundles.
\end{lemma}

\begin{proof} It suffices to take a point $p_1\in \Omega_1$ and to show that $\hat{g}$ is a local isomorphism around every point of the curve $(\pi_1)^{-1}(p_1)\subset X_1$. We denote by $U_1\subset S$ an open neighbourhood of $p_1$ and write $U_2=g(U_1)$ for its image, which is an open neighbourhood of $p_2=g(p_1)\in \Omega$.
By shrinking $U_1$, we can assume that $\pi_1$ and $\pi_2$ are trivial $\p^1\!$-bundles over $U_1$ and $U_2$, respectively. The birational map $\hat{g}$ is then described by 
\[\begin{array}{ccccc}
 U_1 \times \p^1 & \dashrightarrow & U_2 \times \p^1 \\
 \left(x,\begin{bmatrix}
  u \\ v
 \end{bmatrix} \right) & \mapsto & \left( g(x), M(x) \cdot \begin{bmatrix}
  u \\ v
 \end{bmatrix} \right), \\
\end{array}\]    
where $M\in \GL_2(\k(S))$. Since $\k(S)$ is the function field of $\O_{p_1}(S)$, which is a UFD (because $S$ is smooth), we can multiplying $M$ with an element of $\k(S)$ and obtain a matrix $M'$ whose entries are all in $\O_{p_1}(S)$ and do not have a common factor. Denote by $f\in \O_{p_1}(S)$ the determinant of $M'$. If $f$ does not vanish at $p_1$, it is invertible in $\O_{p_1}(S)$, and $\hat{g}$ yields an isomorphism $U'\times \p^1\to g(U')\times \p^1$, where $U'\subset U$ corresponds to the open subset where $f\not=0$. This yields the result since $p_1\in U'$. 

It remains to show that $f$ cannot vanish at $p_1$. Indeed, otherwise the zero set of $f$ yields a curve $C$ of $S$ passing through $p_1$, and there is then a curve $\hat{C}\subset C$ with $\pi_1(\hat{C})=C$  on which the map $\hat{g}$ is not defined (since the matrix $M'$ is unique up to multiplication with an element of $\O_{p_1}(S)^*$ and since at least one of the entries of $M'$ is not divisible by~$f$).
\end{proof}

We  can then prove the following descent lemma, already invoked in the introduction.

\begin{lemma}[Descent lemma]\label{Lem:GoingdownMinimalsurfaces}
Let $\eta\colon\hat{S}\to S$ be a birational morphism between two smooth projective surfaces. Let $U\subset S$ and $\hat{U}\subset \hat{S}$ be two maximal open subsets such that $\eta$ induces an isomorphism $\hat U\iso U$ and $\Omega=S\setminus U$ is finite, and let $\hat\pi\colon \hat{X}\to \hat{S}$ be a $\p^1\!$-bundle.

Then, there exist a $\p^1\!$-bundle $\pi\colon X\to S$ and a birational map $\psi\colon \hat{X}\dasharrow X$ such that $\eta\hat\pi=\pi\psi$ $(\psi$ is a square birational map over $\eta)$ and such that $\psi$ induces an isomorphism $\hat\pi^{-1}(\hat{U})\iso \pi^{-1}(U)$. Moreover, $\psi$ is unique, up to composition by an isomorphism of $\,\p^1$-bundles at the target, and $\psi$ is $\Autz(\hat{X})$-equivariant, which means that $\psi \Autz(\hat{X}) \psi^{-1}$ is a subgroup of $\Autz(X)$.
\end{lemma}

\begin{proof}Writing $G=\Autz(\hat{X})$, Lemma~\ref{blanchard} implies that $\hat\pi$ and $\eta$ are $G$-equivariant for some unique biregular action of $G$ on $\hat{S}$ and $S$.

If $\eta$ is an isomorphism, everything is trivial. Otherwise, $\eta$ is the blow-up of finitely many points, so $\eta$ restricts to an isomorphism $\hat{U}\iso U$, where $U\subset S$ is an open subset, $\Omega=S\setminus U$ is a finite set, and $\hat{U}=\eta^{-1}(U)$. We denote by $j$ the inclusion $U \hookrightarrow S$.

Let $\E \to \hat{S}$ be a rank $2$ vector bundle such that $\P(\E) \simeq \hat{X}$, and let $\E_{\hat{U}} \to \hat{U}$ be its restriction over $\hat{U}$. 
Then $\E_{\hat{U}} \to \hat{U}$ identifies with $\eta_* \E_{\hat{U}} \to U$, and we can consider the reflexive hull $\E'=(j_* (\eta_* \E_{\hat{U}}))^{\vee \vee}$ of the coherent sheaf $j_* (\eta_* \E_{\hat{U}})$ on $S$. By~\cite[Corollary~1.4]{Har80}, the reflexive sheaf $\E'$ is locally free, and thus $\E' \to S$ is a rank $2$ vector bundle that extends (uniquely) $\eta_* \E_{\hat{U}}$. Writing $X=\P(\E')$, we obtain a $\p^1\!$-bundle $X \to S$ that extends (uniquely) the $\p^1\!$-bundle $\P (\eta_* \E_{\hat{U}}) \to U$.

Since $\eta: \hat{S} \to S$ is $G$-equivariant and the open subsets $\hat{U}$ and $U$ are $G$-stable, the $\p^1\!$-bundle $\hat{X}\rightarrow \hat{S}$ is $G$-equivariantly birational to ${X}_{U}(=\P (\eta_* \E_{\hat{U}})) \to U$. It remains to apply Lemma~\ref{Lem:Extension} to see that the $G$-action on ${X}_{U}$ extends to $X$.  
\end{proof}

\begin{remark}
If $\eta$ is the blow-up of a point $p\in S$ and $E=\eta^{-1}(p)\subset \hat{S}$ is the exceptional curve, the birational map $\psi$ can be described as follows (according to~\cite[\S~5.7.4, p.~700]{Mel02}), depending on the surface $Z=\hat\pi^{-1}(E)\simeq \F_a$: if $a=0$, then $\psi$ is a birational morphism whose restriction to $Z\simeq \p^1\times \p^1$ is the ``other projection''; if $a>0$, then $\psi$ is given by the anti-flip of the exceptional curve of $Z$ followed by the contraction of the strict transform of $Z$, isomorphic to the weighted projective plane $\P(1,1,a)$, onto a smooth point. 
\end{remark}

\subsection{Hirzebruch surfaces}\label{Subsec:AutFa}
In the following, we will always use the following coordinates for Hirzebruch surfaces, which are the analogue of the standard coordinates for $\p^2=(\A^3\setminus \{0\})/\G_m$. 

\begin{definition} \label{def Hirzebruch surfaces}
Let $a\in \Z$. The \emph{$a$-th Hirzebruch surface} $\F_a$ is defined to be the quotient of $(\A^2\setminus \{0\})^2$ by the action of $(\G_m)^2$ given by
\[\begin{array}{ccc}
(\G_m)^2 \times (\A^2\setminus \{0\})^2 & \to & (\A^2\setminus \{0\})^2\\
\big((\mu,\rho), (y_0,y_1,z_0,z_1)\big)&\mapsto& (\mu\rho^{-a} y_0,\mu y_1,\rho z_0,\rho z_1).\end{array}\]
The class of $(y_0,y_1,z_0,z_1)$ will be written $[y_0:y_1;z_0:z_1]$. The projection 
\[\tau_a\colon\F_a\to\p^1, \ \ [y_0:y_1;z_0:z_1]\mapsto [z_0:z_1]\]
identifies $\F_a$ with $\P(\OP(a) \oplus \OP)$ as a $\p^1\!$-bundle over $\p^1$. 

The disjoint sections $\s{-a},\s{a}\subset \F_a$ given by $y_0=0$ and $y_1=0$ have self-intersection $-a$ and $a$, respectively. The fibres $f\subset \F_a$ given by $z_0=0$ and $z_1=0$ are linearly equivalent and of self-intersection $0$. We moreover get $\Pic(\F_a)=\Z f\bigoplus \Z \s{-a}= \Z f\bigoplus \Z \s{a}$ since $\s{a}\sim\s{-a}+af$.
\end{definition}

\begin{remark}
The surface $\F_a$ is naturally isomorphic to $\F_{-a}$ via $[y_0:y_1;z_0:z_1]\mapsto [y_1:y_0;z_0:z_1]$, so we will most of the time choose $a \ge 0$.
\end{remark}

Let us recall the classical structure of the automorphism groups of Hirzebruch surfaces. The description of Definition~\ref{def Hirzebruch surfaces} allows one to present all automorphisms in a simple way. 

\begin{remark}\label{Rem:AutF0}
If $a=0$, then $\F_a\simeq \p^1\times \p^1$ and the natural action of $(\GL_2)^2$ on $(\A^2\setminus \{0\})^2$ yields a surjective group homomorphism $(\GL_2)^2\to \Autz(\F_0)\simeq (\PGL_2)^2$  (see \textit{e.g.}~\cite[Corollary~4.2.7]{BSU13} for the last isomorphism). All automorphisms are then of the form
\[[y_0:y_1;z_0:z_1]\mapsto [\alpha y_0+\beta y_1:\gamma y_0+\delta y_1;\alpha' z_0+\beta' z_1:\gamma' z_0+\delta' z_1].\]
The action of $\Autz(\F_0)$ on $\F_0$ is then homogeneous (one single orbit). Moreover, $\Aut(\F_0)=\Autz(\F_0)\rtimes \langle \iota\rangle$, with $\iota\colon[y_0:y_1;z_0:z_1]\mapsto [z_0:z_1;y_0:y_1]$. 
\end{remark}

\begin{remark}\label{Rem:AutFa}
If $a\ge 1$, then the curve $\s{-a}$ given by $y_0=0$ is the unique section of negative self-intersection of $\F_a\to\p^1$ and is thus invariant. Denoting by $\k[z_0,z_1]_{a}\subset \k[z_0,z_1]$ the vector space of homogeneous polynomials of degree $a$, one gets an action of $ \k[z_0,z_1]_{a}\rtimes \GL_2$ on $\F_a$ via
\[[y_0:y_1;z_0:z_1]\mapsto [y_0:y_1+y_0p(z_0,z_1);\alpha z_0+\beta z_1:\gamma z_0+\delta z_1],\]
which yields an exact sequence
\[1\rightarrow \mu_a\to \k[z_0,z_1]_{a}\rtimes \GL_2\rightarrow \Aut(\F_a) \rightarrow 1,\]
where $\mu_a\subset \GL_2$ is the cyclic group of homotheties $\alpha$ with $\alpha^a=1$. 
To prove the surjectivity of the morphism $\k[z_0,z_1]_{a}\rtimes \GL_2\rightarrow \Aut(\F_a)$, it suffices to consider the elements acting trivially on the basis of the $\p^1\!$-bundle $\F_a \to \P^1$ (by Blanchard's lemma, since $\GL_2$ itself acts transitively on the basis $\P^1$). One then restricts such an automorphism to the two charts $z_0\not=0$ and $z_1\not=0$ and obtains two $\A^1$-automorphisms of $\p^1\times \A^1$ that fix $y_0=0$, \textit{i.e.}~two triangular elements of $\GL_2(\k[z_0/z_1])$ and $\GL_2(\k[z_1/z_0])$, and that must coincide on the intersection. A simple calculation then shows that it comes from an element of $\k[z_0,z_1]_a\rtimes \GL_2$ as above. In particular, this implies that $\Aut(\F_a)$ is connected. Moreover, $\Aut(\F_a)=\Autz(\F_a)$ acts on $\F_a$ with two orbits, namely $\s{-a}$ and its complement.
\end{remark}

\begin{remark}\label{Rem:AutFaP1transitiveFibres}
It follows from Remarks~\ref{Rem:AutF0} and~\ref{Rem:AutFa} that for each $a\ge 0$, the morphism $\tau_a\colon \F_a\to \p^1$ yields a surjective group homomorphism
\[\Autz(\F_a)\twoheadrightarrow\Aut(\p^1)\simeq \PGL_2.\]
In particular,  $\Aut(\F_a)$ acts transitively on the set of fibres of $\tau_a\colon \F_a\to \p^1$.
\end{remark}

We also recall the following easy observation, that we will need further on.

\begin{lemma}\label{Lem:ExactSeqFa}
Let $a\ge 0$, and let $\mathcal{E}\to \p^1\!$ be a rank $2$ vector bundle that fits into an exact sequence
\[ 0 \to \O_{\p^1} \stackrel{\iota}{\to} \mathcal{E} \to \O_{\p^1}(-a) \to 0.\]
Then we have an isomorphism from $\p(\mathcal{E})\to \p^1\!$ to $\F_a\to\p^1\!$ which sends the section corresponding to $\iota(\O_{\p^1})$ onto~$\s{-a}$.
\end{lemma}

\begin{proof}We trivialise $\mathcal{E}$ on the two open subsets of $\p^1$ given by $\{[1:z]\mid z\in \A^1\}$ and $\{[z:1]\mid z\in \A^1\}$, change coordinates so that $\iota(\O_{\p^1})$ corresponds to $x_0=0$, and get a transition 
\[\begin{array}{rcl}
\A^2\times \A^1&\dasharrow &\A^2\times \A^1\\
\big((x_0,x_1),z\big) &\dasharrow &\big((z^ax_0,x_1+f(z)),\frac{1}{z}\big)\end{array}\]
for some $f\in \k[z,z^{-1}]$. 
We can then compose at the source and target with some automorphisms of the form $((x_0,x_1),z)\mapsto ((x_0,y_1+h_i(z)),z)$ for  polynomials $h_1,h_2\in \k[z]$. This replaces $f$ by $f+h_1(z)+h_2(z^{-1})z^a$. Since $a\ge 0$, we can thus replace $f$ with $0$. The projectivisation of $\mathcal{E}$ is then isomorphic to $\F_a$, by sending $([x_0:x_1],z)\in \p^1\times \A^1$ onto $[x_0:x_1;1:z]$ and $[x_0:x_1;z:1]=[x_0z^a:x_1;1:z^{-1}]$, on both charts.
\end{proof}

\section{\texorpdfstring{$\boldsymbol{\p^1}$}{P\textasciicircum 1}-bundles over Hirzebruch surfaces} \label{Sec:P1bundlesHirz}

Most of the results in \S~\ref{Sec:P1bundlesHirz} are valid over an algebraically closed field $\k$ of arbitrary characteristic. More precisely, \S\S~\ref{SubSec:DecFa}-\ref{subsect: moduli space} are valid in arbitrary characteristic, the results in \S~\ref{section:action on the moduli space} are all valid in arbitrary characteristic, but the proof of Corollary~\ref{cor action on Mabc} that we give is shortened a bit by using the characteristic zero assumption (see Remark~\ref{rk:proof shortened in char zero}); \S\S~\ref{Sec:TTb} and~\ref{Sec:Umemurabundles}  are also valid in arbitrary characteristic different from $2$, except Lemma~\ref{Lemm:AutTTb}\ref{HatS:OnlyInvariantCurve}, which fails if and only if $\car(\k)$ divides $b+1$ (see Remark~\ref{Rem:CharPSchwarz}). In \S~\ref{subsec:invariant fibers} we must assume that the ground field $\k$ has characteristic zero as Lemma~\ref{lem:subgroups of PGL2} and Proposition~\ref{prop decompo and Umemura bundles} are false over any field of positive characteristic (see Remarks~\ref{Rem:PGL2positive} and~\ref{Rem:CharP}).  

\subsection{Decomposable \texorpdfstring{$\boldsymbol{\p^1}$}{P1}-bundles over Hirzebruch surfaces}\label{SubSec:DecFa}

As for Hirzebruch surfaces $\F_a$ (Definition~\ref{def Hirzebruch surfaces}), one can give global coordinates on decomposable $\p^1\!$-bundles over $\F_a$.

\begin{definition} \label{def Fabc}
Let $a,b,c\in \Z$. We define $\FF_a^{b,c}$ to be the quotient of $(\A^2\setminus \{0\})^3$ by the action of $(\G_m)^3$ given by
\[\begin{array}{ccc}
(\G_m)^3 \times (\A^2\setminus \{0\})^3 & \to & (\A^2\setminus \{0\})^3\\
\big((\lambda,\mu,\rho), (x_0,x_1,y_0,y_1,z_0,z_1)\big)&\mapsto& 
(\lambda\mu^{-b} x_0, \lambda\rho^{-c} x_1,\mu\rho^{-a} y_0,\mu y_1,\rho z_0,\rho z_1).
\end{array}\]
The class of $(x_0,x_1,y_0,y_1,z_0,z_1)$ will be written $[x_0:x_1;y_0:y_1;z_0:z_1]$. The projection 
\[\FF_a^{b,c}\to \F_{a}, \quad [x_0:x_1;y_0:y_1;z_0:z_1]\mapsto [y_0:y_1;z_0:z_1]\]
identifies $\FF^{b,c}_a$ with 
\[
\P\big(\OFa(b \s{a}) \oplus \OFa(c f)\big)=\P\big(\OFa \oplus \OFa(-b\s{a}+cf)\big)
\]
as a $\p^1\!$-bundle over $\F_a$, where $\s{a},f\subset \F_a$ are given by $y_1=0$ and $z_1=0$. 

Moreover, every fibre of the composed morphism $\FF^{b,c}_a \to \F_a \to \p^1$ given by the $z$-projection is isomorphic to $\F_b$, and the restrictions of $\FF^{b,c}_a$ on the curves $\s{-a}$ and $\s{a}$ given by $y_0=0$ and $y_1=0$ are isomorphic to $\F_c$ and $\F_{c-ab}$, respectively. 

As for Hirzebruch surfaces, one can reduce to the case $a\ge 0$ without changing the isomorphism class, by exchanging $y_0$ and $y_1$. We then observe that the exchange of $x_0$ and $x_1$ yields an isomorphism $\FF^{b,c}_a\iso \FF^{-b,-c}_a$. We will  assume $a,b\ge 0$ most of the time  from here on. If $b=0$, we can moreover assume $c\le 0$.
\end{definition}

\begin{remark}\label{Rem:FFabcNumabc}
Every decomposable $\p^1\!$-bundle over $\F_a$ is isomorphic to $\FF_a^{b,c}\to \F_a$ for some $b,c\in \Z,b\ge 0$. We can moreover assume  $b\ge 0$, and  $c\le 0$ if $b=0$, since $\FF_a^{b,c}\simeq \FF_a^{-b,-c}$.

The $\p^1\!$-bundle $\FF_a^{b,c}\to \F_a$ has numerical invariants $(a,b,c)$ (see Definition~\ref{Defi:NumInv}). As we will see later (Remark~\ref{Rem:NumInvAreInv}), these numerical invariants are indeed invariant under isomorphism. This will show that the isomorphism classes of decomposable $\p^1\!$-bundles over Hirzebruch surfaces are parametrised by these invariants.
\end{remark}

\begin{remark}\label{FabcToric}
All decomposable $\p^1\!$-bundles $\FF_a^{b,c}$ are toric varieties, with an action given by the map $[x_0:x_1;y_0:y_1;z_0:z_1]\mapsto [x_0:\alpha x_1;y_0:\beta y_1;z_0:\gamma z_1]$. 
\end{remark}

\begin{remark}\label{Rem:Fbbundle}We have two open embeddings
\[\begin{array}{ccc}
\F_b\times \A^1&\hookrightarrow &\FF_a^{b,c}\\
\big([x_0:x_1;y_0:y_1],z\big) &\mapsto & \big([x_0:x_1;y_0:y_1;1:z]\big)\\
\big([x_0:x_1;y_0:y_1],z\big) &\mapsto & \big([x_0:x_1;y_0:y_1;z:1]\big)\end{array}\]
over which the $\F_b$-bundle $\FF_a^{b,c}\to \p^1$ is trivial, with a transition function 
\[\begin{array}{c}
\big([x_0:x_1;y_0:y_1],z\big) \mapsto \left([x_0:x_1z^{c};y_0z^a:y_1],\frac{1}{z}\right).\end{array}\]
\end{remark}

\begin{lemma}\label{Lem:AutDecOnAutFa}
Let $a,b\ge 0$ and $c\in \Z$. The morphism $\pi\colon\FF_a^{b,c}\to \F_a$ yields a surjective group homomorphism
\[\rho\colon\Autz(\FF_a^{b,c})\twoheadrightarrow\Autz(\F_a).\]
\end{lemma}

\begin{proof}
The existence of $\rho$ is given by Lemma~\ref{blanchard}. The fact that it is surjective can be seen by observing that every automorphism $g\in \Autz(\F_a)$ comes from an automorphism of $(\A^2\setminus \{0\})^2$ (see Remarks~\ref{Rem:AutF0} and~\ref{Rem:AutFa}), so we can extend the action to $(\A^3\setminus \{0\})^2$ and then $\FF_a^{b,c}$ by doing nothing on $x_0$ and $x_1$. 
\end{proof}

\begin{remark} \label{rk aut decompo bundles over Fa}
For each $i,j \in \Z$, we denote by $\k[y_0,y_1,z_0,z_1]_{i,j}\subset \k[y_0,y_1,z_0,z_1]$ the space of homogeneous  polynomials of bidegree $(i,j)$, where the variables $y_0,y_1,z_0,z_1$ are of bidegree $(1,-a),(1,0),(0,1),(0,1)$. The group of automorphisms of $\p^1\!$-bundles of $\FF_{a}^{b,c}$ identifies with the (connected) group  
\[ \left \{ \begin{bmatrix}
 p_{1,0,0}& p_{2,-b,c}\\ p_{3,b,-c} & p_{4,0,0}  \end{bmatrix} \in \PGL_2(\k[y_0,y_1,z_0,z_1]) \ \middle| \ \begin{array}{l}
 p_{k,i,j}\in \k[y_0,y_1,z_0,z_1]_{i,j}\\
 \mbox{for }k=1,\ldots,4.\end{array} \right \}\]
whose action on $\FF_{a}^{b,c}$ is as follows:
$$ [x_0:x_1; y_0:y_1:z_0;z_1] \mapsto [x_0 p_1+x_1 p_2:x_0 p_3+x_1 p_4; y_0:y_1;z_0:z_1].$$
This can be seen  directly from the global description of $\FF_{a}^{b,c}$ in Definition~\ref{def Fabc} and by using trivialisations on open subsets isomorphic to $\A^2$.
\end{remark}

\subsection{Removal of jumping fibres}
Removing a fibre into a Hirzebruch surface, we get an open subset isomorphic to $\A^1\times \p^1$. It is then natural to study the $\p^1\!$-bundles over $\A^1\times\p^1$ in order to get a local description of the $\p^1\!$-bundles over Hirzebruch surfaces.

\begin{lemma}\label{Lemm:P1bundleoverA1P1}
Let $\pi\colon X\to \A^1\times \p^1$ be a $\p^1\!$-bundle, and let $\tau\colon \A^1\times \p^1\to \A^1$ be the first projection. Then, there exist an integer $b\ge 0$ and a dense open subset $U\subset \A^1$ $($both uniquely determined by $\pi)$  such that the following hold:
\begin{enumerate}
\item\label{P1bA1P1gen}
The generic fibre of the morphism ${\tau\pi} \colon X\to \A^1$ is isomorphic to the Hirzebruch surface $\F_b$.
\item \label{TrivialFmbundle}
There exists a commutative diagram
\[\xymatrix@R=4mm@C=2cm{
    (\tau\pi)^{-1}(U) \ar[rr]^{\simeq} \ar[dr]_{\pi}  && U\times \F_b\rlap{,} \ar[dl]^{\mathrm{pr}_1\times \tau_b} \\
   & U\times \p^1 
  }\]
  where $\mathrm{pr}_1\times \tau_b$ sends $(u,x)$ onto $(u,\tau_b(x))$ and $\tau_b\colon \F_b\to \p^1$ is the standard $\p^1\!$-bundle.
\item \label{SpecialfibreFmbundle}
For each $p\in \A^1\setminus U$, the fibre $(\tau\pi)^{-1}(p)$ is isomorphic to the Hirzebruch surface $\F_{b+2\epsilon}$ for some positive integer $\epsilon$.
\item\label{SpecialfibreFmbundleblowup}
For each $p\in \A^1\setminus U$, we can blow up the exceptional section of $(\tau\pi)^{-1}(p)$ and contract the strict transform of $(\tau\pi)^{-1}(p)$; this replaces $X$ by another $\p^1\!$-bundle $X'\to \p^1 \times \A^1$ as above, with a new open subset $U'$ which is equal either to $U$ or to $U\cup \{p\}$. After finitely many such steps, we get the case where $U'=\A^1$, corresponding to a trivial $\F_b$-bundle.
\end{enumerate}
\end{lemma}

\begin{proof}
We choose two open subsets of $V_0,V_1\subset \A^1\times \p^1$ isomorphic to $\A^2$ via
\[
\begin{array}{lllll}
\begin{array}{rcccccc}
\iota_0\colon& \A^2&\iso &V_0\subset \A^1\times \p^1, \\
&(x,y) & \to & \big(x,[1:y]\big)\end{array}& 
\begin{array}{rccc}
\iota_1\colon& \A^2&\iso &V_1\subset \A^1\times \p^1.\\
&(x,y) & \to & \big(x,[y:1]\big)\end{array} \end{array}
\]
The restriction of $\pi$ to the open subsets $V_0$ and $V_1$ yield $\p^1\!$-bundles over $\A^2$, which are then trivial, by the Quillen--Suslin theorem (see~\cite[Theorem~XXI.3.7]{Lang}) combined with the fact that a $\p^1\!$-bundle over $\A^2$ is isomorphic to the projectivisation of a rank $2$ vector bundle over $\A^2$ since $\A^2$ is a smooth variety. This gives the existence of  isomorphisms $\varphi_i\colon\pi^{-1}(V_i)\to \A^2 \times \p^1$, for $i=0,1$, such that $\iota_i\mathrm{pr}_1 \varphi_i=\pi$, where $\mathrm{pr}_1\colon \A^2\times \p^1\to \A^2$ is the first projection.
 The isomorphisms $\varphi_i$ are uniquely determined, up to composing at the target and the source with elements of $\PGL_2(\k[\A^2])\subset \Aut(\A^2\times \p^1)$.
We then write the transition function $\varphi_1(\varphi_0)^{-1}$  as 
\[\big((x,y),[u:v]\big)\mapsto \big((x,y^{-1}),[\alpha_{11}(x,y)u+\alpha_{12}(x,y)v:\alpha_{21}(x,y)u +\alpha_{22}(x,y)v]\big),\]
where $A=\begin{bsmallmatrix}
\alpha_{11} & \alpha_{12} \\
\alpha_{21} & \alpha_{22} \end{bsmallmatrix}\in \GL_2(\k[x,y,y^{-1}])$. Note that the $\p^1\!$-bundle $\pi$ is determined by the equivalence class of $A$ modulo $A\sim \lambda MAM'$, where $\lambda \in \k[x,y,y^{-1}]^{*}=\k^*y^{\Z}$, $M\in \GL_2(\k[x,y^{-1}])$, $M'\in\GL_2(\k[x,y])$. In particular, we can multiply $A$ with an element of $\k^*$ and assume  $\det(A)\in y^{\Z}$.

 Working over the field $\k(x)$, we get a $\p^1\!$-bundle over $\p^1_{\k(x)}$, which is therefore isomorphic to a Hirzebruch surface $\F_b$, with $b\ge 0$ (this follows from the fact that a vector bundle on $\p^1$ is decomposable over any field). This yields two matrices $B\in \GL_2(\k(x)[y^{-1}])$ and $ C\in \GL_2(\k(x)[y])$ such that 
\[B^{-1} AC=D=\begin{bmatrix}
y^m & 0 \\
0 & y^n \end{bmatrix}\]
for some integers $m,n\in \Z$ with $m-n=b$, and yields part~\ref{P1bA1P1gen}.  Since $\det(B),\det(C)\in \k(x)^{*}$ and $\det(A),\det(D)\in y^{\Z}$, we get $\det(A)=\det(D)=y^{m+n}$, so $\det(B)=\det(C)\in \k(x)^*$.

Writing the equality $AC=BD$, we can multiply both $B$ and $C$ with the same element of $\k(x)^*$ and assume  $B\in \mathrm{Mat}_{2,2}(\k[x,y^{-1}])$, $C\in \mathrm{Mat}_{2,2}(\k[x,y])$, and  $(B(\lambda),C(\lambda))\not=(0,0)$ for each $\lambda\in \k$, where $B(\lambda),C(\lambda)\in \mathrm{Mat}_{2,2}(\k(y))$ are obtained by replacing $x$ by $\lambda$ in $B,C$.

We denote by $\mathcal{Z}\subset \A^1$ the zero set of $\det(B)=\det(C)\in \k[x]$. On the open set $U=\p^1\setminus \mathcal{Z}$, the matrices $B$ and $C$ correspond to automorphisms of $U\times \A^1$ on the two charts, respectively, so we have a trivial $\F_b$-bundle on $U$. In particular, if $\mathcal{Z}=\emptyset$, then $\tau\pi\colon X\to \A^1$ is a trivial $\F_b$-bundle, and the proof is over. We can thus prove the result by induction on the degree of the polynomial $\det(B)$.

Suppose that $\lambda\in \mathcal{Z}$ is such that the fibre $(\tau\pi)^{-1}(\lambda)$ is a Hirzebruch surface $\F_{\tilde{b}}$. Hence, $A(\lambda)$ corresponds to the transition function of $\F_{\tilde{b}}$, which means that  $\tilde{B}^{-1} A(\lambda) \tilde{C}= \begin{bsmallmatrix} 
y^{\tilde{m}} & 0 \\
0 & y^{\tilde{n}} \end{bsmallmatrix}$ with $\tilde{B}\in\GL_2(\k[y^{-1}])$, $\tilde{C}\in\GL_2(\k[y])$, $\tilde{m},\tilde{n}\in \Z$, and $\tilde{m}-\tilde{n}=\tilde{b}\ge 0$. Computing the determinant yields $m+n=\tilde{m}+\tilde{n}$. Writing $\epsilon=\tilde{m}-m=n-\tilde{n}$, we then get $\tilde{b}-b=(\tilde{m}-\tilde{n})-(m-n)=2\epsilon$. Replacing $A,B,C$ with $\tilde{B}^{-1} A \tilde{C}$, $(\tilde{B})^{-1}B$, $(\tilde{C})^{-1}C$, we keep the equation $AC=BD$, do not change the degree of $\det(B)=\det(C)$ or the set $\mathcal{Z}$, and can then assume  $A(\lambda)=\begin{bsmallmatrix}
y^{\tilde{m}} & 0 \\
0 & y^{\tilde{n}} \end{bsmallmatrix}$. Writing $B(\lambda)=\begin{bsmallmatrix}
\beta_{11} & \beta_{12} \\
\beta_{21} & \beta_{22} \end{bsmallmatrix}$ yields \[C(\lambda)=A(\lambda)^{-1}B(\lambda)D=\begin{bmatrix}
\beta_{11}y^{m-\tilde{m}} & \beta_{12}y^{n-\tilde{m}} \\
\beta_{21}y^{m-\tilde{n}} & \beta_{22}y^{n-\tilde{n}} \end{bmatrix}=\begin{bmatrix}
\beta_{11}y^{-\epsilon} & \beta_{12}y^{-b-\epsilon} \\
\beta_{21}y^{b+\epsilon} & \beta_{22}y^{\epsilon} \end{bmatrix}.\] 

$(a)$ If the first column of $B(\lambda)$ is zero, then so is the first column of $C(\lambda)$. Writing $\Delta=\begin{bsmallmatrix} x-\lambda & 0 \\ 0 & 1\end{bsmallmatrix}$, we get $B=B'\Delta $ and $C=C'\Delta $ for some $B',C'\in \mathrm{Mat}_{2,2}(\k[x,y,y^{-1}])$. Replacing $B$ and $C$ by $B'$ and $C'$ does not change the equation $AC=BD$, since $D$ commutes with $\Delta$, and decreases the degree of $\det(B).$ A similar argument works if the second column of $B$ is zero.

$(b)$ If $\epsilon<0$, then $-b-\epsilon=\epsilon-\tilde{b}<0$, so the second column of $B(\lambda)$ and $C(\lambda)$ is zero since $B(\lambda)\in \mathrm{Mat}_{2,2}(\k[y^{-1}])$ and $C(\lambda)\in\mathrm{Mat}_{2,2}(\k[y])$; we then apply $(a)$.

$(c)$ If $\epsilon=b=0$, then $B(\lambda)=C(\lambda)\in \mathrm{Mat}_{2,2}$. There thus exists an $R\in \GL_2$ such that the first column of $B(\lambda)R$ is zero. We can replace $B,C$ with $BR,CR$, since $R$ commutes with $D=y^m\cdot I=y^n\cdot I$, and reduce to case $(a)$. 

$(d)$ If $\epsilon=0$ and $b>0$, then $\beta_{12}=0$ and $\beta_{22}\in \k$. If $\beta_{22}=0$, we do as above. If $\beta_{22}\not=0$, we get $\beta_{11}=0$ since $\det(B(\lambda))=0$, hence $B(\lambda)=\begin{bsmallmatrix}
0 & 0 \\
\beta_{21} & \beta_{22} \end{bsmallmatrix}$, so the first column of $B(\lambda)\cdot R$ is zero, with $R=\begin{bsmallmatrix}\beta_{22}& 0  \\  -\beta_{21}& 1 \end{bsmallmatrix}\in \GL_2(\k[y^{-1}])$. Writing $R'=D^{-1}RD=\begin{bsmallmatrix} \beta_{22}& 0  \\  -\beta_{21}y^b& 1 \end{bsmallmatrix}\in \GL_2(\k[y])$, we can replace $B$ and $C$ with $BR$ and $CR'$ and get case $(a)$.

$(e)$ The last case is when $\epsilon>0$, which implies that $\tilde{b}=b+2\epsilon\ge b+2\ge 2$ and that $\beta_{11}=\beta_{12}=0$. 

After applying the steps above, we can assume that all elements of $\mathcal{Z}$ give rise to case $(e)$. Writing $U=\A^1 \setminus \mathcal{Z}$, this yields parts~\ref{TrivialFmbundle} and~\ref{SpecialfibreFmbundle}.

It remains to show part~\ref{SpecialfibreFmbundleblowup}, by studying more carefully case $(e)$. Note that the fibre $x=\lambda$ corresponds to the Hirzebruch surface $\F_{\tilde{b}}$, with $\tilde{b}\ge 2$, and a special section corresponds to $u=0$ in the charts $\A^2\times \p^1$. The blow-up of the exceptional section, followed by the contraction of the strict transform of the surface $\F_{\tilde{b}}$, corresponds locally to 
\[\begin{array}{ccc}
\A^2\times \p^1 & \dasharrow & \A^2 \times \p^1\\
((x,y),[u:v])&\mapsto& ((x,y),[u:(x-\lambda)v]).\end{array}\]
This replaces the transition matrix $A$ with $A'=\Delta^{-1} A \Delta$, where $\Delta=\begin{bsmallmatrix} x-\lambda & 0 \\ 0 & 1\end{bsmallmatrix}:$
\[A=\begin{bmatrix}
\alpha_{11} & \alpha_{12} \\
\alpha_{21} & \alpha_{22} \end{bmatrix},\quad A'=\begin{bmatrix}
\alpha_{11} & \frac{\alpha_{12}}{x-\lambda} \\
(x-\lambda)\alpha_{21} & \alpha_{22} \end{bmatrix}.\]
Note that the new transition $A'$ still belongs to $\GL_2(\k[x,y,y^{-1}])$ since $A(\lambda)$ is diagonal, which implies that $\alpha_{12}$ and $\alpha_{21}$ are multiples of $x-\lambda$.

Moreover, the first line of $B(\lambda)$ and $C(\lambda)$ is zero, so we can write $B=\Delta B'$ and $C=\Delta C'$ for some $B',C'\in \mathrm{Mat}_{2,2}(\k[x,y,y^{-1}])$. The blow-up replaces $A$ by $A'$, and we can replace $B,C$ with $B',C'$ since $A'C'=(\Delta^{-1} A \Delta)(\Delta^{-1} C)=\Delta^{-1} BD=B'D$. This process decreases the degree of $\det(B)=\det(C)$; we get a trivial $\F_b$-bundle after finitely many steps.
\end{proof}

As a consequence of Lemma~\ref{Lemm:P1bundleoverA1P1}, we get the following result.

\begin{proposition}[Removal of jumping fibres]\label{Prop:NoMoreJumpingfibre}
Let $a\ge 0$, and let $\pi\colon X\to \F_a$ be a $\p^1\!$-bundle. There exist an integer $b\ge 0$ and a dense open subset of $U\subset\p^1$ such that $(\tau_a\pi)^{-1}(p)$ is a Hirzebruch surface $\F_b$ for each $p\in U$. Moreover, we have:
\begin{enumerate}
\item\label{UP1}
If $U=\p^1$, then $\tau_a\pi\colon X\to \p^1$ is an $\F_b$-bundle which is trivial on every affine open subset of $\,\p^1$. In this case, we say that \emph{$\pi$ has no jumping fibre}.
\item
If one fibre $(\tau_a\pi)^{-1}(\{p\})$ is isomorphic to $\F_c$ for some $c\not=b$, then $c-b$ is a positive even integer $($we say that $\tau_a^{-1}(\{p\})$ is a \emph{jumping fibre}$)$, and the blow-up of the $($unique$)$ exceptional section of $\,\F_c$ followed by the contraction of the strict transform of $\,\F_c$ gives an $\Autz(X)$-equivariant birational map $X\dasharrow X'$ to another $\p^1\!$-bundle over $\F_a$. After finitely many such steps, one gets case~\ref{UP1}.
\end{enumerate}
\end{proposition}

\begin{proof}For each $p\in \p^1$, we have a commutative diagram
\[\xymatrix@R=3mm@C=2cm{
    \A^1\times\p^1 \ar[r]^{\simeq} \ar[d]_{\mathrm{pr}_1}  & \F_a\setminus \tau_a^{-1}(\{p\}) \ar[d]_{\tau_a} \\
    \A^1 \ar[r]_{\simeq} & \p^1\setminus \{p\}\rlap{.}
  }\]
We can thus apply Lemma~\ref{Lemm:P1bundleoverA1P1} on each affine subset $\p^1\setminus \{p\}$ and get the result.
\end{proof}

Another consequence of Lemma~\ref{Lemm:P1bundleoverA1P1} is the following description.

\begin{corollary}\label{Coro:NumericalInv}
Let $\pi\colon X\to S=\A^1\times \p^1$ be a $\p^1\!$-bundle, let $\tau\colon S\to \A^1$ be the first projection, and let $b\ge 0$. The following conditions are equivalent: 
\begin{enumerate}
\item\label{PA1e1}
The fibre $(\tau\pi)^{-1}(p)$ is a Hirzebruch surface $\F_b$ for each $p\in \A^1$.
\item\label{PA1e2}
There exists a commutative diagram
\[\xymatrix@R=4mm@C=1.5cm{
    X \ar[rr]^{\simeq} \ar[dr]_{\pi}  && \A^1\times \F_b\rlap{,} \ar[dl]^{\mathrm{pr}_1\times \tau_b} \\
   & S=\A^1\times \p^1 
  }\]
  where $\mathrm{pr}_1\times \tau_b$ sends $(u,x)$ on $(u,\tau_b(x))$.
  \item\label{PA1e3}
 The $\p^1\!$-bundle $X\to S$ is the projectivisation of a rank $2$ vector bundle $\mathcal{E}$  which fits in a short exact sequence
\[ 0 \to \O_{S} \to \mathcal{E} \to \O_{S}(-bs) \to 0,\]
where $s$ is a fibre of $\,\mathrm{pr}_2\colon S\to \p^1$ $($which satisfies $\Pic(S)=\Z s)$.
\end{enumerate}
\end{corollary}

\begin{proof} The implication \ref{PA1e1} $\Rightarrow$ \ref{PA1e2} is given by 
Proposition~\ref{Prop:NoMoreJumpingfibre}\ref{UP1}. To get $\ref{PA1e2}\Rightarrow\ref{PA1e3}$, we observe that $\ref{PA1e2}$ yields an isomorphism between the $\p^1\!$-bundles $X\to \A^1\times \p^1$ and $\A^1 \times \p\big(\O_{\p^1}\bigoplus \O_{\p^1}(b)\big)\to \A^1\times \p^1$. To get \mbox{\ref{PA1e3} $\Rightarrow$ \ref{PA1e1},} we restrict the exact sequence to each fibre of $\tau\colon S\to \A^1$ and apply Lemma~\ref{Lem:ExactSeqFa}.
\end{proof}

\subsection{Moduli spaces of \texorpdfstring{$\p^1$}{P1}-bundles over \texorpdfstring{$\F_a$}{Fa} with no jumping fibre} \label{subsect: moduli space}
The following proposition associates to every $\p^1\!$-bundle over a Hirzebruch surface with no jumping fibre three unique invariants $(a,b,c)$, called \emph{numerical invariants} in Definition~\ref{Defi:NumInv}. The integer $a$ is given by the Hirzebruch surface over which the $\p^1\!$-bundle is taken. The integer $b$ is given by the generic fibre of the projection to $\p^1$, which is a Hirzebruch surface $\F_b$. The last integer $c$ can be seen using exact sequences or using transition functions, as the following result explains.

\begin{proposition}\label{Prop:EquivABC}
Let $a\ge 0$, and let $\pi\colon X\to \F_a$ be a $\p^1\!$-bundle.
\begin{enumerate}
\item\label{Equivabc}
For all integers $b\ge 0$ and $c\in \Z$, the following are equivalent: 
\begin{enumerate}
\item\label{Equivabc1} The variety $X$ is the gluing of two copies of $\,\F_b\times \A^1$  along $\F_b \times \A^1 \setminus \{0\}$ by the automorphism $\nu_{c,P}\in \Aut(\F_b \times \A^1 \setminus \{0\})$ given by\\
  $\begin{array}{c}
  \nu_{c,P}\colon([x_0:x_1;y_0:y_1],z) \mapsto \left([x_0:x_1z^{c}+x_0P(y_0,y_1,z);y_0z^a:y_1],\frac{1}{z}\right)\end{array}$

for some $P\in\k[y_0,y_1,z,z^{-1}]$, homogeneous of degree $b$ in $y_0,y_1$, such that $\pi\colon X\to \F_a$ sends \linebreak $([x_0:x_1;y_0:y_1],z)\in \F_b\times \A^1$ onto, respectively,
$[y_0:y_1;1:z]\in \F_a$ and $[y_0:y_1;z:1]\in \F_a$ on the two charts. 
 \item\label{Equivabc2}
The $\p^1\!$-bundle $\pi\colon X\to \F_a$ is the projectivisation of a rank $2$ vector bundle 
$\mathcal{E}$ which fits in a short exact sequence
\[ 0 \to \O_{\F_a} \to \mathcal{E} \to \O_{\F_a}(-b\s{a}+cf) \to 0, \]
where $f,\s{a}\subset \F_a$ are given by $y_1=0$ and $z_1=0$.
\end{enumerate}
\item\label{Existabc}
If there exists a $b\ge 0$ such that the preimage of each fibre of the $\p^1\!$-bundle $\tau_a\colon \F_a\to\p^1$ is isomorphic to $\F_b$ $($no jumping fibre$)$, then there is an integer $c\in \Z$ such that the above properties are satisfied. If $b>0$, the integer $c$ is unique. If $b=0$, then $\lvert c\rvert$ is unique and $\pi\colon X\to \F_a$ is isomorphic to the decomposable bundles $\FF_a^{0,c}\to \F_a$ and $\FF_a^{0,-c}\to \F_a$.
\end{enumerate}
\end{proposition}

\begin{proof}
We first prove \ref{Equivabc1} $\Rightarrow$ \ref{Equivabc2}.
The section $x_0=0$ being invariant by the transition function, one can see $X$ as the projectivisation of a rank $2$ vector bundle 
$\mathcal{E}$ which fits in a short exact sequence
\[ 0 \to \O_{\F_a} \to \mathcal{E} \to \O_{\F_a}(\beta \s{a}+\gamma f) \to 0\]
for some integers $\beta,\gamma\in \Z$. To compute these numbers, we take the two open subsets $U_0,U_1\subset\F_a$ isomorphic to $\A^2$ via
\[
\begin{array}{lllll}
\begin{array}{rcccccc}
\theta_0\colon&  \A^2&\iso &U_1\subset \F_a, \\
&(y,z) & \to & [1:y;1:z]\end{array}& 
\begin{array}{rccc}
\theta_1\colon&  \A^2&\iso &U_2\subset \F_a\\
&(y,z) & \to & [y:1;z:1]\end{array} \end{array}
\]
and observe that the vector bundle has a transition function of the form
\[(x_0,x_1,y,z)\mapsto \left(y^bz^{-c}x_0,x_1+x_0P(1,y,z)z^{-c},\frac{z^a}{y},\frac{1}{z}\right),\]
which yields $\beta=-b$ and $\gamma=c$.

We then prove~\ref{Existabc}. It follows from Proposition~\ref{Prop:NoMoreJumpingfibre} that the pull-back on $X$ of the two open subsets $V_0,V_1\subset \F_a$ given by
\[
\begin{array}{lllll}
\begin{array}{rcccccc}
\theta_0\colon&  \p^1\times \A^1&\iso &V_0\subset \F_a,\\
&([y_0:y_1],z) & \to & [y_0:y_1;1:z]\end{array}& 
\begin{array}{rccc}
\theta_1\colon&  \p^1\times \A^1&\iso &V_1\subset \F_a\\
&([y_0:y_1],z) & \to & [y_0:y_1;z:1]\end{array} \end{array}
\]
 are isomorphic to $ \A^1\times \F_b$. The transition function on $\F_a$ being given by $([y_0:y_1],z)\dasharrow ([z^ay_0:y_1],z^{-1})$, we get a transition function $\psi \in \Aut(\F_b \times \A^1\setminus \{0\})$ of the form 
\[
\begin{array}{ccccccc}
\F_b\times \A^1&\dasharrow &\F_b\times \A^1\\
\big([x_0:x_1;y_0:y_1],z\big) & \mapsto & \big([f_0(x_0,x_1,y_0,y_1,z):f_1(x_0,x_1,y_0,y_1,z);z^ay_0:y_1],\frac{1}{z}\big)\end{array}
\]
for some $f_0,f_1\in \k[x_0,x_1,y_0,y_1][z,z^{-1}]$. The isomorphism class of $\pi\colon X\to \F_a$ (as in Definition~\ref{Defi:SquareEtc}) is then determined by $\psi$, up to composition at the source and target with automorphisms of the $\p^1\!$-bundle $\F_b\times \A^1\to \p^1\times \A^1$.

If $b=0$, then $\psi$ is of the form
$([x_0:x_1;y_0:y_1],z)  \mapsto ([\alpha(z)x_0+\beta(z)x_1;\gamma(z)x_0+\delta(z)x_1;z^ay_0:y_1],z^{-1})$, 
where $A=\begin{bsmallmatrix} \alpha(z)& \beta(z) \\ \gamma(z) & \delta(z)\end{bsmallmatrix} \in \GL_2(\k[z,z^{-1}])$. The isomorphism class of $\pi\colon X\to \F_a$ is then given by the matrix $A$, up to replacing $A$ by $\lambda BAC$, where $B\in \GL_2(\k[z^{-1}])$, $C\in \GL_2(\k[z])$, and $\lambda\in \k^*$. The class of $A$ modulo this replacement corresponds to a vector bundle over $\p^1$, which is therefore equivalent to a decomposable one, with a diagonal matrix $A$. We then get an integer $c\in \Z$, unique up to sign, such that $\psi$ is of the form
\[\begin{array}{c}\psi\colon \big([x_0:x_1;y_0:y_1],z\big) \mapsto \left([x_0:x_1z^{c};y_0z^a:y_1],\frac{1}{z}\right).\end{array}\]
This shows that $X\to \F_a$ is isomorphic to $\FF_a^{b,c}=\FF_a^{0,c}\to \F_a$
 (see Remark~\ref{Rem:Fbbundle}) and also to $\FF_a^{0,-c}\to \F_a$.

If $b>0$, then $\psi$ is of the form 
$([x_0:x_1;y_0:y_1],z)  \mapsto  ([x_0:\mu(z)x_1+x_0P(y_0,y_1,z);z^ay_0:y_1],z^{-1})$, 
where $\mu(z)\in \k[z,z^{-1}]^*=\k^*\cdot z^{\Z}$ and $P(y_0,y_1,z)$ is a homogeneous polynomial of degree $b$ in $y_0$, $y_1$, with coefficients in $\k[z,z^{-1}]$. Applying a diagonal automorphism at the target, we can assume  $\mu(z)=z^c$ for some $c\in \Z$. We then observe that $c$ is unique since the transition function is determined up to automorphisms of $\F_b\times \A^1\to \p^1\times \A^1$ at the target and the source, which do not change $c$.

It remains to prove \ref{Equivabc2} $\Rightarrow$ \ref{Equivabc1}.
 The inclusion $\O_{\F_a} \hookrightarrow \mathcal{E}$ corresponds to a section of $X\to \F_a$. We restrict the exact sequence to a fibre $f$ and get $ 0 \to \O_{f} \to \mathcal{E}_f \to \O_{f}(-b) \to 0.$ The corresponding section then needs to be the exceptional section of $\F_b$, unique if and only if $b>0$ (see Lemma~\ref{Lem:ExactSeqFa}).
The preimage of each fibre of the $\p^1\!$-bundle $\tau_a\colon \F_a\to\p^1$ is then isomorphic to $\F_b$ $($no jumping fibre$)$. We then apply~\ref{Existabc} and get $\ref{Equivabc1}$ for some unique integer $c'$. The calculation made in the proof of \ref{Equivabc1} $\Rightarrow$ \ref{Equivabc2} implies that $c=c'$ since the inclusion  $\O_{\F_a} \hookrightarrow \mathcal{E}$ corresponds to the section $x_0=0$.
\end{proof}

\begin{remark}\label{Rem:NumInvAreInv}
Proposition~$\ref{Prop:EquivABC}$ shows that two $\p^1\!$-bundles with different numerical invariants (see Definition~\ref{Defi:NumInv}) are not isomorphic. The numerical invariants defined in Definition~\ref{Defi:NumInv}, which correspond to the integers $(a,b,c)$ of Proposition~$\ref{Prop:EquivABC}$ (where $c\le0$ when $b=0$), are then really invariant under isomorphisms.

In particular, $\FF_a^{b,c}\to \F_a$ is the unique isomorphism class of decomposable $\p^1\!$-bundles with invariants $(a,b,c)$. (This follows from Remark~\ref{Rem:FFabcNumabc}).
\end{remark}

As a direct consequence of Proposition~\ref{Prop:EquivABC}, we obtain the following corollary (which is well known over the field of complex numbers; see \textit{e.g.}~\cite[\S~2.2]{ABV}).

\begin{corollary}\label{Coro:ABV}
Let $a,b\ge 0$, and let $\mathcal{E}$ be a rank $2$ vector bundle over $\F_a$. Then the following are equivalent: 
\begin{enumerate}
\item\label{ExactSeqVB}
There exists an exact sequence
\[ 0 \to \O_{\F_a}(d\s{-a}+rf) \to \mathcal{E} \to \O_{\F_a}(d'\s{-a}+r'f) \to 0\]
for some integers $d,d',r,r'$ such that $b=d-d'$.
\item\label{ExactSeqVB2}
The preimage by $\p(\mathcal{E})\to\F_a$ of each fibre of the $\p^1\!$-bundle $\tau_a\colon \F_a\to\p^1$ is isomorphic to $\F_b$.
\end{enumerate}
Moreover, the extension in~\ref{ExactSeqVB} is unique if $b>0$. 
\end{corollary}

\begin{proof}
\ref{ExactSeqVB} $\Rightarrow$ \ref{ExactSeqVB2} The inclusion $\O_{\F_a}(d\s{-a}+rf) \hookrightarrow \mathcal{E}$ corresponds to a section of $\p(\mathcal{E})\to\F_a$. We restrict the exact sequence to a fibre $f$ and get $ 0 \to \O_{f}(d) \to \mathcal{E}_f \to \O_{f}(d') \to 0$, which yields the same $\p^1\!$-bundle as $ 0 \to \O_{f} \to \mathcal{E}_f \to \O_{f}(-b) \to 0$. The corresponding section then needs to be the exceptional section of $\F_b$ (see Lemma~\ref{Lem:ExactSeqFa}), which is unique if $b>0$.

\ref{ExactSeqVB2} $\Rightarrow$ \ref{ExactSeqVB} This follows from Proposition~\ref{Prop:EquivABC}\ref{Existabc}.
\end{proof}

\begin{notation}\label{Not:ZabcP}
  Let $a,b,c\in \Z$, with $a,b\ge 0$, and $c\le 0$ if $b=0$. For each
  \begin{center}$P\in\k[y_0,y_1,z,z^{-1}]_b=\{f\in \k[y_0,y_1,z,z^{-1}]\mbox{ homogeneous of degree $b$ in }y_0,y_1\}$,\end{center} 
we denote by $Z_{a}^{b,c,P}\to \F_a$ the $\p^1\!$-bundle given by
the gluing of two copies of $\F_b\times \A^1$  along $\F_b \times \A^1 \setminus \{0\}$ by the automorphism $\nu_{c,P}\in \Aut(\F_b \times \A^1 \setminus \{0\})$ given by
$$\nu_{c,P}\colon\big([x_0:x_1;y_0:y_1],z\big) \mapsto \left([x_0:x_1z^{c}+x_0P(y_0,y_1,z);y_0z^a:y_1],\frac{1}{z}\right),
  $$
such that $\pi\colon X\to \F_a$ sends $([x_0:x_1;y_0:y_1],z)\in \F_b\times \A^1$ onto, respectively,  
$[y_0:y_1;1:z]\in \F_a$ and $[y_0:y_1;z:1]\in \F_a$ on the two charts.
\end{notation}

\begin{remark}\label{Rem:ZabcHasNumABC}
Proposition~\ref{Prop:EquivABC} shows that every $\p^1\!$-bundle over $\F_a$ with no jumping fibre is isomorphic to $Z_{a}^{b,c,P}\to \F_a$ for some $b,c\in \Z$ with $b\ge 0$, and $c\le 0$ if $b=0$, and some $P\in\k[y_0,y_1,z,z^{-1}]_b$, and that it has numerical invariants $(a,b,c)$ (see Definition~\ref{Defi:NumInv}). Moreover, $Z_{a}^{b,c,0}\to \F_a$ is isomorphic to $\FF_a^{b,c}$ (see Remark~\ref{Rem:Fbbundle}).
\end{remark}

\begin{lemma}\label{Lemm:EquivalenceClassP}
Let $\pi\colon Z_{a}^{b,c,P}\to \F_a$ and $\pi'\colon Z_{a}^{b',c',P'}\to \F_a$ be two $\p^1\!$-bundles as in Notation~$\ref{Not:ZabcP}$, with $b\ge 1$. Then the following are equivalent: 
\begin{enumerate}
\item\label{Isophichipi}
The $\p^1\!$-bundles $\pi\colon Z_{a}^{b,c,P}\to \F_a$ and $\pi'\colon Z_{a}^{b',c',P'}\to \F_a$ are isomorphic.
\item\label{ccpPprime}
  We have $b'=b$ and $c'=c$, and there exist $\lambda\in \k^*$ and $Q_1,Q_2\in \k[y_0,y_1,z]$ homogeneous of degree $b$ in $y_0,y_1$ such that
  \begin{center}$P'=\lambda P+Q_1(y_0,y_1,z)z^{c}+Q_2(y_0z^a,y_1,z^{-1})$.
  \end{center}
\end{enumerate}
\end{lemma}

\begin{proof}If $b\not=b'$, then the two $\p^1\!$-bundles are not isomorphic since the preimages of the fibres of the $\p^1\!$-bundle $\tau_a\colon \F_a\to\p^1$ are not isomorphic. We can thus assume  $b'=b$.

The two $\p^1\!$-bundles are obtained by gluing two copies of $\F_b\times \A^1$ over $\F_b\times \A^1\setminus \{0\}$ by $\nu_{c,P},\nu_{c',P'}\in \Aut(\F_b \times \A^1 \setminus \{0\})$ (see Notation~$\ref{Not:ZabcP}$). The $\p^1\!$-bundles are thus isomorphic if and only if $\nu_{c',P'}=\alpha\nu_{c,P}\beta$ for some automorphisms $\alpha$, $\beta$ of the $\p^1\!$-bundle $\F_b\times \A^1\to \p^1\times \A^1$. Since $b\ge 1$, such elements are of the form
\[
\begin{array}{rcccccc}
\theta_{\lambda,Q}\colon& \F_b\times \A^1&\iso &\F_b\times \A^1\\
&\big([x_0:x_1;y_0:y_1],z\big) & \mapsto & \big([x_0:\lambda x_1+x_0 Q(y_0,y_1,z);y_0:y_1],z\big),\end{array}
\]
where $\lambda\in \k^*$ and $Q\in \k[y_0,y_1,z]_b$. The composition $\theta_{\lambda_2,Q_2} \nu_{c,P}\theta_{\lambda_1,Q_1}$ yields
\[
\begin{array}{cccccc}
 \F_b\times \A^1\setminus \{0\}&\iso &\F_b\times \A^1\setminus \{0\}\\
\big([x_0:x_1;y_0:y_1],z\big) & \mapsto & \big([x_0:z^{c}\lambda_1\lambda_2x_1+x_0\tilde{P}(y_0,y_1,z);z^ay_0:y_1],\frac{1}{z}\big),\text{ with}\end{array}\]
\[\begin{array}{rcl}
\tilde{P}(y_0,y_1,z)&=&\lambda_2P(y_0,y_1,z)+\lambda_2Q_1(y_0,y_1,z)z^{c}+Q_2\big(y_0z^a,y_1,\frac{1}{z}\big).\end{array}
\]
To get a transition function of the form $\nu_{c',P'}$, we then need $c'=c$ and $\lambda_1\lambda_2=1$.
\end{proof}

From now on we write $\k[z]_{\le r}=\{f\in \k[z]\mid \deg(f)\le r\}=\k\oplus \k z\oplus \cdots \oplus\k z^r$.

\begin{corollary}\label{Coro:IsoClass}
 Let $a,b,c\in \Z$, with $a,b\ge 0$ and with $c\le 0$ if $b=0$.
\begin{enumerate}
\item\label{UniqueIsoDecABC}
There is a unique isomorphism class of decomposable $\p^1\!$-bundles $X\to \F_a$ with numerical invariants $(a,b,c)$, represented by $\FF_a^{b,c}\to \F_a$. 
\item
If $b=0$ or $c\le 1$, every $\p^1\!$-bundle $X\to \F_a$ with numerical invariants $(a,b,c)$ is decomposable, and thus isomorphic to $\FF_a^{b,c}\to \F_a$.
\item\label{Pgiven}
If $b\ge 1$ and $c\ge 2$, every $\p^1\!$-bundle with numerical invariants $(a,b,c)$ is isomorphic to  $Z_{a}^{b,c,P}\to \F_a$, where
\[\begin{array}{c}P(y_0,y_1,z)=\sum_{i=0}^b  {y_0}^i{y_1}^{b-i}P_i(z)z^{ai+1}\end{array}\]and $P_i(z)\in \k[z]_{\le c-2-ai}$ $($hence $P_i=0$ if $c<ai+2)$, for $i=0,\ldots,b$.

The isomorphism class of the $\p^1\!$-bundle is determined by the class of $P$, up to scalar multiplication by an element of $\,\k^*$.  The $\p^1\!$-bundle is decomposable if and only if $P=0$.
\end{enumerate}
\end{corollary}

\begin{proof}
Assertion~\ref{UniqueIsoDecABC} has already been proven (Remark~\ref{Rem:NumInvAreInv}).
If $b=0$, then Proposition~$\ref{Prop:EquivABC}$ shows that every $\p^1\!$-bundle over $\F_a$ with numerical invariants $(a,b,c)$ is isomorphic to $\FF_{a}^{b,c}=\FF_{a}^{0,c}$. We can thus assume  $b\ge 1$.

Proposition~$\ref{Prop:EquivABC}$ shows that every $\p^1\!$-bundle $X\to \F_a$ with invariants $(a,b,c)$ is isomorphic to $Z_{a}^{b,c,P}\to \F_a$ for some $P\in\k[y_0,y_1,z,z^{-1}]_b.$ Lemma~\ref{Lemm:EquivalenceClassP} shows that the isomorphism class is inside the set of $\p^1\!$-bundles with invariants $(a,b,c)$ and corresponds to an equivalence class on $\k[y_0,y_1,z,z^{-1}]_b$, where $P$ and $P'$ are equivalent if and only if
$P'=\lambda P+Q_1(y_0,y_1,z)z^{c}+Q_2(y_0z^a,y_1,z^{-1})$, for $\lambda\in \k^*$ and $Q_1,Q_2\in \k[y_0,y_1,z]_b$.
In particular, each equivalence class is given by an element 
\begin{center}$P(y_0,y_1,z)=\sum_{i=0}^b  {y_0}^i{y_1}^{b-i}P_i(z)z^{ai+1}$,\end{center}
where $P_i(z)\in \k[z]$ is of degree at most $c-ai-2$ for $i=0,\ldots,b$ and the element $P$ is unique up to multiplication by $\lambda\in \k^*$. If $c\le 1$, then $P$ is zero, so every $\p^1\!$-bundle $X\to \F_a$ with numerical invariants $(a,b,c)$ is decomposable. This achieves the proof.
\end{proof}

\begin{corollary} \label{Cor:modspace}
Let $a,b,c\in \Z$, with $a\ge 0$, $b\ge1$, and $c\ge 2$. The isomorphism classes of non-decomposable $\p^1\!$-bundles $X \to \F_a$ with numerical invariants $(a,b,c)$ are parametrised by the projective space 
\[\begin{array}{c}
\M_{a}^{b,c}=\P \left( \bigoplus\limits_{i=0}^{b} y_0^i y_1^{b-i}\cdot\k[z]_{\leq c-2-ai} \right).\end{array}\]
\end{corollary}

\begin{proof}This follows from Corollary~\ref{Coro:IsoClass}\ref{Pgiven}. 
\end{proof}

\begin{remark} \label{rk: dim mod spaces}
We have $\M_{a}^{b,c}\simeq \p^{\frac{1}{2}(d+1)(2(c-1)-ad)-1}$, where $d$ is the biggest integer such that $d\le b$ and $ad\le c-2$ (and is thus equal to $b$ if $ab\le c-2$). Indeed, the dimension of the vector space $\k[z]_{\leq c-2-ai}$ is equal to $c-1-ai$ if $ai\le c-2$ and to $0$ if $ai>c-2$. Hence, the dimension of $\bigoplus_{i=0}^{b} y_0^i y_1^{b-i}\cdot\k[z]_{\leq c-2-ai}$ is equal to $\sum_{i=0}^d (c-1-ai)=\frac{1}{2}(d+1)(2(c-1)-ad)$.
\end{remark}

\subsection{Action of \texorpdfstring{$\boldsymbol{\Autz(\F_a)}$}{Aut(Fa)} on the moduli spaces \texorpdfstring{$\boldsymbol{\M_{a}^{b,c}}$}{Mabc}} \label{section:action on the moduli space}
If $\pi\colon X\to \F_a$ is a $\p^1\!$-bundle with numerical invariants $(a,b,c)$, then so is $\varphi  \circ \pi\colon X\to \F_a$ for each $\varphi \in \Autz(\F_a)$. Indeed, the action of $\Autz(\F_a)$ does not change the exact sequence of Definition~\ref{Defi:NumInv}. This then gives a natural left-action of $\Autz(\F_a)$  on $\M_{a}^{b,c}$, that we describe in this section; it will be very useful later, for the following reason: a point $[p\colon X \to \F_a]\in\M_{a}^{b,c}$ is fixed by $g\in \Autz(\F_a)$ if and only if $g$ lifts to an automorphism of the variety $X$.

 Since we have a group homomorphism $\GL_2\to \Autz(\F_a)$ (see $\S\ref{Subsec:AutFa}$), we get an action of $\GL_2$ on $\M_{a}^{b,c}=\P \big( \bigoplus_{i=0}^{b} y_0^i y_1^{b-i}\cdot\k[z]_{\leq c-2-ai} \big)$ (see Corollary~\ref{Cor:modspace}). We will show that this $\GL_2$-action coincides with the following one.
 
\begin{definition} \label{GL2 action}
Let $r \ge 0$, and let us equip $V=\k^2$ with the standard left-action of $\GL_2$. There is a unique left-action of $\GL_2$ on $\k[z]_{\leq r}$ making the following map $\GL_2$-equivariant: 
\[\begin{array}{ccc}
V & \mapsto & \k[z]_{\leq r}\\
(u,v)& \mapsto & \sum_{i=0}^r u^{i}v^{r-i} \cdot z^i.\end{array}\]
\end{definition}

\begin{remark} \label{Remark rep is Sb(V)}
Equipped with the $\GL_2$-action of Definition~\ref{GL2 action}, the vector space $\k[z]_{\leq r}$ identifies with the $r$-th symmetric power of the standard representation of $\GL_2$. In particular, $\k[z]_{\leq r}$ is an irreducible $\GL_2$-representation (as we assumed $\k$ to be of characteristic zero).
\end{remark}

We first need the following observations on the action of $\GL_2$ on $\k[z]_{\leq r}$ of Definition~\ref{GL2 action}.

\begin{lemma}\label{Lemm:TechnicalAction}
Let $r\ge 0$, and let $P\in \k[z]_{\leq r}$.
\begin{enumerate}
\item\label{DescInv}
If $\sigma=\begin{bsmallmatrix}
0 & 1 \\ 1 & 0
\end{bsmallmatrix}\in \GL_2$, then $\sigma(P)=P(z^{-1})\cdot z^r$.
\item\label{DescTrian}
If $\sigma=\begin{bsmallmatrix}
\alpha  & \beta \\ 0 & \delta
\end{bsmallmatrix}\in \GL_2$, then $\hat{P}:=\sigma(P)\in \k[z]_{\leq r}$ is the unique polynomial that satisfies 
\begin{equation}\label{equationInvertPPprim}
P(z)=\frac{\alpha}{\delta^r(\beta z +\alpha)}\cdot \hat{P}\left(\frac{\delta z}{\beta z +\alpha}\right),\end{equation}
where we identify $\k[z]_{\le r}$ with $\k[z]/(z^{r+1})$ and compute the above equality in this ring.
\end{enumerate}
\end{lemma}

\begin{proof}
\ref{DescInv} The action of $\sigma$ on $V$ being $(u,v)\mapsto (v,u)$, the action on $P\in \k[z]_{\leq r}$ sends $\sum_{i=0}^r a_i z^i$ onto $\sum_{i=0}^r a_{r-i} z^i$, which is exactly $P\mapsto P(z^{-1})z^r$.

\ref{DescTrian}
We first  observe that Equation~\eqref{equationInvertPPprim} uniquely determines $\hat{P}$ in terms of $P$ since $\alpha$, $\delta$, and $\beta z+\alpha$ are invertible in $\k[z,z^{-1}]/(z^{r+1})$ (because $\alpha\delta\not=0$). We then check that if Equation~\eqref{equationInvertPPprim} is true for $P,Q\in \k[z]_{\le r}$, then it is true for all $\mu P +\nu Q$ with $\mu,\nu\in \k$. We thus only need to check it for $P=\sum_{i=0}^r u^{i}v^{r-i}\cdot z^i$, where $(u,v)\in \k^2$. By the definition of the action, we get $\hat{P}=\sigma(P)=\sum_{i=0}^r (\alpha u+\beta v)^{i}(\delta v)^{r-i}\cdot z^i$, which yields
\[\begin{array}{rcl}
(\beta z +\alpha)^{r+1}\cdot P(z)&=&\sum_{i=0}^r u^{i}v^{r-i} z^i(\beta z+\alpha)^{r+1}, 
\\[1ex]
\frac{\alpha(\beta z +\alpha)^r}{\delta^r}\cdot\hat{P}\left(\frac{\delta z}{\beta z +\alpha}\right)&=&\frac{\alpha(\beta z +\alpha)^r}{\delta^r}\cdot \sum_{i=0}^r (\alpha u+\beta v)^{i}(\delta v)^{r-i}\cdot \left(\frac{\delta z}{\beta z +\alpha}\right)^i\\[1ex]
&=&  \sum_{i=0}^r \alpha (\alpha u +\beta v)^{i}(\beta z +\alpha)^{r-i}v^{r-i}z^i. 
\end{array}\]

Comparing the coefficients of $z^p$ for $p=0,\ldots,r-1$, we need to prove that
\[\begin{array}{rcl}
\sum_{i=0}^p u^{i}v^{r-i} \mybinom{r+1}{p-i}\beta^{p-i}\alpha^{r+1-(p-i)}&=&\sum_{i=0}^p 
(\alpha u +\beta v)^iv^{r-i}\mybinom{r-i}{p-i}\beta^{p-i}\alpha^{r+1-p}
\end{array}
\]
for each $p=0,\ldots,r-1$.
Comparing the coefficients of $u^{k}v^{r-k}$, we get
\[\begin{array}{rcl}\mybinom{r+1}{p-k}\beta^{p-k}\alpha^{r+1-(p-k)}&=&\sum_{i=k}^p \binom{i}{k}\beta^{i-k}\alpha^k\mybinom{r-i}{p-i}\beta^{p-i}\alpha^{r+1-p}, \end{array}\]
so the result follows from the next combinatoric lemma.
\end{proof}

\begin{lemma}\label{Lemm:Combinatoric}For all integers  $0\le k\le p\le r$, we have
\[\binom{r+1}{p-k}=\sum_{i=k}^p \binom{i}{i-k}\binom{r-i}{p-i}.\]
 \end{lemma}

\begin{proof}
The sides of the equations count the number of elements of the sets
 \[\begin{array}{rcl}
 \mathcal{S}&=&\vphantom{\Big)}\{C\subset \{0,\ldots ,r \}\mid C\mbox{ contains }p-k\mbox{ elements}\},\\
 \mathcal{S}'&=&\left\{(i,A,B)\left| \begin{array}{ll}
 i\in \{k,\ldots,p\}, A\subset \{0,\ldots,i-1\}, B\subset \{i+1,\ldots,r\},\\
 A\mbox{ contains }i-k \mbox{ elements}, B\mbox{ contains }p-i \mbox{ elements}\end{array} \right\},\right.\end{array}\]
 so it remains to prove that the map $\mathcal{S}'\to \mathcal{S}$, $(i,A,B)  \mapsto  A\cup B$ is bijective. This corresponds to showing that for each $C\in \mathcal{S}$, there exists a unique $i\in \{k,\ldots,p\}\setminus C$ such that $C\cap \{0,\ldots,i-1\}$ contains $i-k$ elements. This is because the map
 \[
 \tau\colon \{k,\ldots,p\}\to \N,\quad i\mapsto (i-k)-\lvert C\cap \{0,\ldots,i-1\}\rvert
 \]
 is non-decreasing, is increasing outside $C$, and satisfies $\tau(k)\le 0\le\tau(p)$.
\end{proof}

\begin{lemma}\label{Lemm:ActionAutFaParSpace}
Let $a \ge 0, b\ge 1, c\ge 2$. Let $\pi\colon Z_{a}^{b,c,P}\to \F_a$ be the $\p^1\!$-bundle induced by $P(y_0,y_1,z)=\sum_{i=0}^b y_0^iy_1^{b-i}P_i(z)z^{ai+1}$ with $ P_i\in \k[z]_{\le c-2-ai}$ $($Notation~$\ref{Not:ZabcP})$.

For each $\varphi\in \Autz(\F_a)$, the $\p^1\!$-bundle $\varphi \pi$ is isomorphic to $Z_{a}^{b,c,\hat{P}}\to \F_a$, where $\hat{P}$ is defined by 
$\hat{P}(y_0,y_1,z)=\sum_{i=0}^b y_0^iy_1^{b-i}\hat{P}_i(z)z^{ai+1}$, with $\hat{P}_i \in \k[z,z^{-1}]$ given as follows: 
\begin{enumerate}
\item\label{ActionGL2}
If $\varphi([y_0:y_1;z_0:z_1])=[y_0:y_1;\alpha z_0+\beta z_1:\gamma z_0+\delta z_1]$ for some $\sigma=\begin{bsmallmatrix}\alpha&\beta\\ \gamma&\delta\end{bsmallmatrix}\in \GL_2$, then $\hat{P}_i= \sigma(P_i)$ for each $i$, where the action of $\,\GL_2$ on the vector space $\k[z]_{\le c-ai}$ is that of Definition~$\ref{GL2 action}$.
\item
\label{ActionGL20}
If $a=0$ and $\varphi([y_0:y_1;z_0:z_1])=[\alpha y_0+\beta y_1;\gamma y_0+\delta y_1:z_0:z_1]$ for some $\alpha,\beta,\gamma,\delta\in \k$ with $\alpha\delta-\beta\gamma\not=0$, then $\hat{P}(y_0,y_1,z)$ satisfies
$P(y_0,y_1,z)=\hat{P}(\alpha y_0+\beta y_1,\gamma y_0+\delta y_1,z)$.
\item
\label{ActionFap}
If $a\ge 1$ and $\varphi([y_0:y_1;z_0:z_1])=[ y_0:y_1+y_0R(z_0,z_1):z_0:z_1]$ for some $R\in \k[z_0,z_1]_a$, then $\hat{P}$ is such that $P(y_0,y_1,z)=\hat{P}(y_0,y_1+y_0R(z,1),z)$.
\end{enumerate}
\end{lemma}

\begin{proof}
Recall that the transition function is given by $\nu_{c,P}\in\Aut(\F_b\times \A^1\setminus \{0\})$, 
\[\begin{array}{c}
\nu_{c,P}\colon \big([x_0:x_1;y_0:y_1],z\big) \mapsto \left([x_0:x_1z^{c}+x_0 P(y_0,y_1,z);y_0z^a:y_1],\frac{1}{z}\right)\end{array}\]
and that the morphism $\pi\colon Z_{a}^{b,c,P}\to \F_a$ is given on the two charts by $\tau_0,\tau_1\colon \F_b\times \A^1  \to \F_a$, which send $([x_0:x_1;y_0:y_1],z)$ onto, respectively,  $[y_0:y_1;1:z]$ and $[y_0:y_1;z:1]$.

We construct a transition function $\nu_{c,P'}$ of the $\p^1\!$-bundle $\varphi \pi\colon X\to \F_a$ for some $\varphi\in \Aut(\F_a)$ and show that it is equivalent to that of $\hat{P}$ (modulo the equivalence described in Lemma~\ref{Lemm:EquivalenceClassP}). To do this, we find corresponding birational maps $\theta_0$ and $\theta_1$ compatible with the following commutative diagram:
\begin{displaymath}
 \xymatrix@R=4mm@C=10mm{
 &&& \F_b \times \A^1\ar@{-->}[llld]_(.4){\theta_0} \ar@{-->}[rr]^{\nu_{c,P}} \ar[dddr]_{\tau_0}&&   \F_b\times \A^1\rlap{.}\ar[dddl]^{\tau_1}\ar@{-->}[llld]^(.4){\theta_1}\\ 
  \F_b \times \A^1 \ar@{-->}[rr]_{\nu_{c,P'}}\ar[dddr]_{\tau_0} &&   \F_b \times \A^1\ar[dddl]^{\tau_1}  \\
\\
  &&&&\F_a\ar[llld]_{\simeq}^{\varphi}\\
  &\F_a
}
\end{displaymath}

\ref{ActionGL2}
When $\varphi([y_0:y_1;z_0:z_1])=[y_0:y_1;\alpha z_0+\beta z_1:\gamma z_0+\delta z_1]$ for some $\sigma=\begin{bsmallmatrix}
\alpha & \beta\\
\gamma& \delta\end{bsmallmatrix}\in \GL_2$, the action of $\varphi$ on $\F_a$ corresponds to 
\[\begin{array}{llll}
\ [y_0:y_1;1:z]&\mapsto &[y_0:y_1;\alpha+\beta z :\gamma+\delta z]=\big[(\alpha+\beta z)^{a}y_0:y_1;1:\frac{\gamma +\delta z}{\alpha +\beta z}\big],\\
\ [y_0:y_1;z:1]&\mapsto &[y_0:y_1; \alpha z +\beta: \gamma z+\delta ]=\big[( \gamma z +\delta)^{a}y_0:y_1;\frac{\alpha z+\beta}{\gamma z +\delta}:1\big]\vphantom{\Big)}\end{array}\] on the two charts. To check that the action we gave on the moduli space is the right one, we only need to check it for generators of $\GL_2$.

$(i)$ We first do the case where $\gamma=0$ (upper-triangular matrices). The second chart of $\F_a$ is preserved, and $\varphi$ corresponds to 
\[\begin{array}{llllll}
\ [y_0:y_1;1:z]&\mapsto &[y_0:y_1;\alpha+\beta z :\delta z]&=&\big[(\alpha+\beta z)^{a}y_0:y_1;1:\frac{\delta z}{\alpha +\beta z}\big],\\
\ [y_0:y_1;z:1]&\mapsto &[y_0:y_1; \alpha z +\beta: \delta ]&=&\big[\delta^{a}y_0:y_1;\frac{\alpha z+\beta}{\delta}:1\big]\vphantom{\Big)}\end{array}\] on the two charts. It then suffices to choose
\[\begin{array}{rccc}
\theta_1\colon & \F_b\times \A^1 & \iso & \F_b\times \A^1\\
&([x_0:x_1;y_0:y_1],z)&\mapsto & \big([x_0:x_1:\delta^{a} y_0:y_1],\frac{\alpha z+\beta}{\delta}\big)\end{array}\]
and to choose a transition function $P'(y_0,y_1,z)\in \k[y_0,y_1,z]$ such that $\theta_0= (\nu_{c,P'})^{-1}\theta_1 \nu_{c,P}$ is a local isomorphism at each point where $z=0$. We compute that $\theta_0([x_0:x_1;y_0:y_1],z)$ is equal to
\[\left(\left[x_0:\left(\frac{\beta z+\alpha}{\delta}\right)^{c}\left(x_1+\frac{x_0}{z^{c}}R(y_0,y_1,z)\right):(\beta z+\alpha)^{a} y_0:y_1\right],\frac{\delta z}{\beta z +\alpha}\right),\]
where $R(y_0,y_1,z)=P(y_0,y_1,z)-P'\big(y_0(\beta z+\alpha)^a,y_1,\frac{\delta z}{\beta z +\alpha}\big)$. We then only need to choose the transition function $P'(y_0,y_1,z)\in \k[y_0,y_1,z]$ so that $\theta_0$ is a local isomorphism at each point where $z=0$, which corresponds to saying that the valuation of $z$ at $R$ is at least $c$.
We will observe that such a $P'$ exists in the equivalence class of $\hat{P}$. Writing $P'(y_0,y_1,z)=\sum_{i_0}^{b} y_0^i y_1^{b-i} P'_i(z) z^{ai+1}$, we get
\[\begin{array}{rcl}
P'(y_0\big(\beta z+\alpha)^a,y_1,\frac{\delta z}{\beta z +\alpha}\big)&=&\sum\limits_{i=0}^b {y_0}^i{y_1}^{b-i}\frac{\delta^{ai+1} P'_i\left(\frac{\delta z}{\beta z +\alpha}\right)}{\beta z +\alpha}z^{ai+1}, \end{array}\]
and then we need to choose the $P'_i$ such that $\frac{\delta^{ai+1} P'_i({\delta z}/({\beta z +\alpha}))}{\beta z +\alpha}-P_i(z)$ has valuation at least $c$ at $z$. Since $\beta z+\alpha$ is invertible in $\k[z]/(z^{c-ai-1})$, we find a unique solution in $\k[z]/(z^{c-ai-1})$, equal to $\frac{\alpha}{\delta^{c-1}}\sigma(P_i)$ by Lemma~\ref{Lemm:TechnicalAction}\ref{DescTrian}.

$(ii)$ It remains to consider the case where $\sigma=\begin{bsmallmatrix}
0 & 1 \\ 1 & 0
\end{bsmallmatrix}$. The two charts are exchanged here, so it suffices to choose $\nu_{c,P'}=(\nu_{c,P})^{-1}$, which yields $P'(y_0,y_1,z)=-P(y_0z^a,y_1,z^{-1})z^{c}$, which is equivalent to 
\[\begin{array}{rcl}
P\big(y_0z^a,y_1,\frac{1}{z}\big)z^{c}&=&\sum_{i=0}^b (y_0z^a)^iy_1^{b-i}P_i\big(\frac{1}{z}\big)\big(\frac{1}{z}\big)^{ai+1}z^{c}\\[1ex]
&=&\sum_{i=0}^b y_0^iy_1^{b-i}\bigl(P_i\big(\frac{1}{z}\big)\,z^{c-ai-2}\bigr)\,z^{ai+1}\\[1ex]
&\stackrel{\text{Lemma}~\ref{Lemm:TechnicalAction}\ref{DescInv}}{=}&\sum_{i=0}^b y_0^iy_1^{b-i}\sigma(P_i)\,z^{ai+1}.
\end{array}\]

\ref{ActionGL20}
Now suppose $a=0$ and that $\varphi([y_0:y_1;z_0:z_1])=[\alpha y_0+\beta y_1;\gamma y_0+\delta y_1:z_0:z_1]$. Choosing $\theta_1=\theta_2\colon ([x_0:x_1;y_0:y_1],z)\mapsto ([x_0:x_1: \alpha y_0+\beta y_1:\gamma y_0+\delta y_1],z)$, we get $\nu_{c,\hat{P}} \theta_1=  \theta_2 \nu_{c,P}$ with $P(y_0,y_1,z)=\hat{P}(\alpha y_0+\beta y_1,\gamma y_0+\delta y_1,z)$.

\ref{ActionFap}
Now suppose  $a>0$ and  $\varphi([y_0:y_1;z_0:z_1])=[ y_0:y_1+y_0R(z_0,z_1):z_0:z_1]$ for some $R\in \k[z_0,z_1]_a$. We then choose $\theta_1\colon ([x_0:x_1;y_0:y_1],z)\mapsto ([x_0:x_1: y_0:y_1+y_0R(z,1)],z)$ and $\theta_2\colon ([x_0:x_1;y_0:y_1],z)\mapsto$ $([x_0:x_1: y_0:y_1+y_0 R(1,z)],z)$ and get $\nu_{c,\hat{P}} \theta_1=  \theta_2 \nu_{c,P}$ with $P(y_0,y_1,z)=\hat{P}(y_0,y_1+y_0R(z,1),z)$.
\end{proof}

\begin{corollary} \label{cor action on Mabc}
Let $a,b,c\in \Z$, with $a\ge 0$, $b\ge1$, and $c\ge 2$. We recall that $\GL_2$ acts on $\k[z]_{\leq r}$ $($see Definition~$\ref{GL2 action})$ and on $\k[y_0,y_1]_b$ when $a=0$ $($see Lemma~$\ref{Lemm:ActionAutFaParSpace}\ref{ActionGL20})$.
We have the following equivariant isomorphisms: 
\begin{enumerate}
\item\label{GL2SV}  If $a\ge 1$, then  $\M_{a}^{b,c} \simeq \P \left( \bigoplus_{i=0}^{b}  y_0^{i}y_1^{b-i} \cdot k[z]_{\leq c-2-ai} \right)$ as $\GL_2$-varieties. 
\item\label{GL2SVVp} If $a=0$, then as $\GL_2 \times \GL_2$-varieties, 
$$\M_{0}^{b,c} \simeq \P\big( k[y_0,y_1]_b \otimes k[z]_{\leq c-2}\big) = \P\big( \Hom((k[y_0,y_1]_b)^*, k[z]_{\leq c-2})\big).$$ 
Moreover, if $b=c-2$, then $\Hom^{\GL_2}((k[y_0,y_1]_b)^*, k[z]_{\leq b})=\{\lambda I; \lambda \in k\}$, and the identity element corresponds to $Z_{a}^{b,b+2,P}$, with 
\begin{center}$P=\sum_{i=0}^b y_0^iy_1^{b-i} P_i(z)z\;$
and $\;P_i(z)=z^i$ for $i=1,\ldots,b$.
\end{center}
\end{enumerate}
\end{corollary}

\begin{proof}
Assertion~\ref{GL2SV} and the first part of~\ref{GL2SVVp} follow from Corollary~\ref{Cor:modspace} and Lemma~\ref{Lemm:ActionAutFaParSpace}.

We now assume  $b=c-2$ and define a non-degenerate bilinear form $\phi$ as follows: 
$$\begin{array}{cccccc}
\phi : & k[y_0,y_1]_b &\times &k[z]_{\leq b} &\to &k \\
       &\big( \sum_{j=0}^{b} c_j y_0^j y_1^{b-j} &, &\sum_{i=0}^{b} d_i z^i \big) &\mapsto &\sum_{l=0}^{b} c_l d_l.\\
      \end{array}$$
The map $\phi$ is $\GL_2$-invariant. Indeed, by bilinearity, it suffices to check that for all $g \in \GL_2$, for all $j=0,\ldots,b$, and for all $(u,v) \in k^2$, we have
$$\phi \left(g \cdot y_0^j y_1^{b-j}, g \cdot \sum_{i=0}^{b} u^i v^{b-i} z^i \right)=\phi \left( y_0^j y_1^{b-j}, \sum_{i=0}^{b} u^i v^{b-i} z^i \right)=u^jv^{b-j}.$$
This can be checked directly for $g=\begin{bsmallmatrix}
0 & 1 \\ 1 & 0
\end{bsmallmatrix}$ and $g=\begin{bsmallmatrix}
\alpha & \beta \\ 0 & \delta
\end{bsmallmatrix}$. Since these elements generate $\GL_2$, the $\GL_2$-invariance of $\phi$ follows.

Therefore, the $\GL_2$-representations $\k[y_0,y_1]_b$ and $\k[z]_{\leq b}$ are dual to each other. Since they are both irreducible (a consequence of Remark~\ref{Remark rep is Sb(V)}), it follows from Schur's lemma that $\Hom^{\GL_2}((k[y_0,y_1]_b)^*, k[z]_{\leq b})=\{\lambda I; \lambda \in k\}$. Identifying $\k[z]_{\leq b}$ with $(\k[y_0,y_1]_b)^*$ via the map $\phi$ above, we see that $\{1,z,\ldots,z^b\}$ is the dual basis of $\{y_1^b, y_0 y_1^{b-1},\ldots,y_0^b\}$. Hence the identity element corresponds to $\sum_{i=0}^b y_0^iy_1^{b-i} z^i$ as an element of $\k[y_0,y_1]_b \otimes k[z]_{\leq b}$, and so it corresponds to the $\p^1\!$-bundle $Z_a^{b,b+2,P}$, with $P$ as in the statement of the corollary.
\end{proof}

\begin{remark} \label{rk:proof shortened in char zero}
In the proof of Corollary~\ref{cor action on Mabc}, we use the fact that the ground field is of characteristic zero to say that the $\GL_2$-representations $\k[y_0,y_1]_b$ and $\k[z]_{\leq b}$ are irreducible (this is used to prove the last part of the statement). The statement of the corollary is true over any algebraically closed field of arbitrary characteristic, but a general proof is slightly more complicated since we can no longer invoke Schur's lemma to prove that $\Hom^{\GL_2}((k[y_0,y_1]_b)^*, k[z]_{\leq b})=\{\lambda I; \lambda \in k\}$, and we have to do the calculation by hand (using Lemma~\ref{Lemm:ActionAutFaParSpace}).
\end{remark}

\subsection{The \texorpdfstring{$\boldsymbol{\p^1}$}{P1}-bundles \texorpdfstring{$\boldsymbol{\TT_b\to \p^1\times \p^1}$}{HatSb -> F0}}\label{Sec:TTb}

We now study two families of non-decomposable $\p^1\!$-bundles $X\to \F_a$, in \S~\ref{Sec:TTb} and \S~\ref{Sec:Umemurabundles}, which will play an important role further on (see Proposition~\ref{prop decompo and Umemura bundles}).

\begin{definition}\label{def hat Schwrzenberger}
For each integer $b\ge 1$, we define $\TT_b\to \F_0=\p^1\times\p^1$ to be the $\p^1\!$-bundle $Z_{0}^{b,b+2,P}\to \F_0$, where $P=\sum_{i=0}^b y_0^iy_1^{b-i} P_i(z)z$
and $P_i(z)=z^{i}$.
\end{definition}

\begin{remark}
  The $\p^1\!$-bundle  $\TT_b\to \F_0$ naturally arises in Corollary~\ref{cor action on Mabc}\ref{GL2SVVp}, which explains why the image of $\Autz(\TT_b)\to \Autz(\F_0)$ contains the diagonal group $H_\Delta=\{(g,g)\mid g\in \PGL_2\}\subset \PGL_2\times \PGL_2=\Autz(\F_0)$. Lemma~\ref{Lemm:AutTTb} below makes this explicit  and shows that $\Autz(\TT_b)\simeq \PGL_2$.
\end{remark}

\begin{remark}
The $\p^1\!$-bundle  $\TT_b\to \F_0$ has numerical invariants $(0,b,b+2)$ since it is equal to $Z_{0}^{b,b+2,P}$ for some polynomial $P$  (see Remark~\ref{Rem:ZabcHasNumABC}).
\end{remark}

\begin{remark}
We will prove in Lemma~\ref{Lemm:LiftSchwarzIsSchwarzHat} that the $\p^1\!$-bundle $\TT_b$ of Definition~\ref{def hat Schwrzenberger} coincides with the lift of the Schwarzenberger bundle $\SS_b\to \p^2$ of Definition~\ref{def:Schwarz}. 
\end{remark}

\begin{lemma}\label{Lemm:AutTTb}
Let $b\ge 1$ be an integer, and let us denote by $\pi,\pi'$ the $\p^1\!$-bundles $\pi\colon \TT_b\to \p^1\times \p^1$ and $\pi'\colon \FF_0^{b+1,b+1}\to \p^1\times \p^1$. Then, the following hold: 
\begin{enumerate}
\item\label{S1S2TTb}
For each $i=1,2$, denoting by $\mathrm{pr}_i\colon \p^1\times \p^1\to \p^1$ the $i$-th projection, the morphism $\mathrm{pr}_i\pi\colon \TT_b\to \p^1$ is a $\F_b$-bundle.
Denoting by $S_i\subset \TT_b$ the union of the $(-b)$-curves of the $\F_b$, the intersection $C=S_1\cap S_2$ is a curve isomorphic to the diagonal $\Delta\subset \p^1\times \p^1$ via~$\pi$. It corresponds to the intersection of $\pi^{-1}(\Delta)$ with the surface $x_0=0$ in both charts.
\item\label{CommDiag}
We have $\Autz(\TT_b)\simeq \PGL_2$ and a commutative diagram
\[\xymatrix@R=3mm@C=2cm{
& X\ar[rd]^{\eta}\ar[ld]_{\epsilon}\\
     \TT_{b}\ar@{-->}[rr]^{\psi} \ar[rd]  && \FF_0^{b+1,b+1}, \ar[ld] \\
    & \p^1\times \p^1
  }\]
  where all maps are $\PGL_2$-equivariant, the action of $\PGL_2$ on $\p^1\times \p^1$ is the diagonal one, the action of $\PGL_2$ on $\FF_0^{b+1,b+1}$ is given by 
  \begin{equation}\label{equationActionPGL2F0mm}
    [x_0:x_1;y_0:y_1;z_0:z_1]\mapsto [x_0:x_1;\alpha y_0+\beta y_1:\gamma y_0+\delta y_1;\alpha z_0+\beta z_1:\gamma z_0+\delta z_1],\end{equation}
the morphism $\epsilon$ is the blow-up of the curve $C$, and the morphism $\eta$ is the blow-up of the curve $C'\subset \FF_0^{b+1,b+1}$ given by $\p^1\hookrightarrow \FF_0^{b+1,b+1}$, $[u:v]\mapsto [1:1:u:v:u:v]$.
  
  Moreover, every automorphism of the $\p^1\!$-bundle $\TT_b\to \p^1\times \p^1$ is trivial.
\item\label{HatS:OnlyInvariantCurve}
If $\car(\k)$ does not divide $b+1$, the curve $C$ is the unique curve invariant by $\Autz(\TT_b)$.
  \end{enumerate}
\end{lemma}

\begin{proof}We write $m=b+1$. The fact that \eqref{equationActionPGL2F0mm} yields an action of $\PGL_2$ on $\FF_0^{m,m}$ follows from the fact that 
$[x_0:x_1;\lambda y_0:\lambda y_1;\lambda z_0:\lambda z_1]=[\lambda^{m}x_0:\lambda^{m}x_1;y_0:y_1;z_0:z_1]=[x_0:x_1;y_0:y_1;z_0:z_1]$ for each $\lambda\in \k^*$. The same argument shows that the map $\p^1\to \FF_0^{m,m}$, $[u:v]\mapsto [1:1:u:v:u:v]$ is a well-defined closed embedding. The image $C'$ is sent to the diagonal $\Delta\subset \p^1\times \p^1$, and the action on $\p^1\times \p^1$ (via the $\p^1\!$-bundle $\FF_0^{m,m}\to \p^1\times \p^1$) is the diagonal action.

We now construct a birational map $\varphi\colon \FF_0^{m,m}\dasharrow \TT_{m-1}=\TT_b$.

The two open subsets $U_0,U_1\subset\FF_0^{m,m}$ where $z_0\not=0$ and $z_1\not=0$, respectively, are isomorphic to $\F_m\times \A^1$ via
\[\begin{array}{rccc}
&\F_m\times \A^1&\hookrightarrow &\FF_0^{m,m}\\
\iota_0\colon &\big([x_0:x_1;y_0:y_1],z\big) & \mapsto & [x_1:x_0;y_0:y_1;1:z],\\
\iota_1\colon &\big([x_0:x_1;y_0:y_1],z\big) & \mapsto & [x_1:x_0;y_0:y_1;z:1].
\end{array}\]
The transition function $\theta'=(\iota_1)^{-1} \iota_0\in \Bir(\F_m\times \A^1)$ is then given by
\[\big([x_0:x_1;y_0:y_1],z\big)\mapsto \big([x_0:x_1z^m;y_0:y_1],\frac{1}{z}\big),\]
and the curve $C'\subset \FF_0^{m,m}$ yields a section of $\F_m\times \A^1\to \A^1$, given by $\A^1\hookrightarrow \F_m\times \A^1$, $z\mapsto ([1:1;1:z],z)$ and $z\mapsto ([1:1;z:1],z)$, respectively, on the two charts. We can then blow up $C$ and contract the strict transform of ${\pi'}^{-1}(\pi'(C'))$; we do it on the two charts via
\[\begin{array}{rccc}
&\F_m\times \A^1&\dasharrow &\F_m\times \A^1\\
\varphi_0\colon &\big([x_0:x_1;y_0:y_1],z\big) & \mapsto & \big([x_0(y_0z-y_1):x_1-x_0y_0^m;y_0:y_1],z\big),\\
\varphi_1\colon &\big([x_0:x_1;y_0:y_1],z\big) & \mapsto & \big([x_0(y_0-y_1z):x_1-x_0y_1^m;y_0:y_1],z\big).
\end{array}\]
Computing the transition function $\theta=\varphi_1 \theta'(\varphi_0)^{-1}\in \Bir(\F_{b}\times \A^1)$, we obtain
\[\big([x_0:x_1;y_0:y_1],z\big)\mapsto \left(\left[x_0:x_1z^{m+1}+x_0z\frac{y_0^mz^m -y_1^m}{y_0z-y_1};y_0:y_1\right],\frac{1}{z}\right),\]
which is the transition function of the $\p^1\!$-bundle $\TT_{b}\to \p^1\times \p^1$ (see Definition~\ref{def hat Schwrzenberger}).

Since $C'$ and $\pi'^{-1}(\pi'(C'))=\pi'^{-1}(\Delta)$ (where $\Delta\subset \p^1\times \p^1$ is the diagonal) are invariant by $\PGL_2$, the birational map $\varphi$ is $\PGL_2$-equivariant, for some biregular action of $\PGL_2$ on $\TT_{b}$, acting diagonally on $\p^1\times \p^1$, and preserving the curve $C\subset \TT_{b}$ being the image of the contracted surface $\pi'^{-1}(\pi'(C'))$. Note that this curve is given by the intersection of $x_0=0$ with $\pi^{-1}(\Delta)$ in both charts, where $\Delta\subset \p^1\times \p^1$ is the diagonal (this follows by replacing $y_0$ and $y_1$ by $1$ and $z$ in $\varphi_0$ and by $z$ and $1$ in $\varphi_1$). This yields the commutative diagram of~\ref{CommDiag}, with $\psi=\varphi^{-1}$. To achieve the proof of~\ref{CommDiag}, we only need to show that the homomorphism $\PGL_2\to \Autz(\TT_{b})$ constructed by this map is surjective.

The second projection $\mathrm{pr}_2\colon \p^1\times \p^1$ satisfies that $\mathrm{pr}_2\pi\colon \TT_b\to \p^1$ is a $\F_b$-bundle, trivial on $\p^1\setminus [0:1]$ and $\p^1\setminus [1:0]$ with transition function $\theta$. The union of the $(-b)$-curves of the $\F_b$ is a surface $S_2\subset \TT_b$, which corresponds to $x_0=0$ on both charts and is then sent by $\psi$ onto the surface $S_2'\subset \FF_0^{m,m}$ given by $x_1=0$.

The involution $\sigma'\in \Aut(\FF_0^{m,m})$ given by $\sigma'\colon [x_0:x_1;y_0:y_1;z_0:z_1]\mapsto [x_1:x_0;z_0:z_1;y_0:y_1]$ commutes with $\PGL_2$ and preserves $C'$ and $\pi'^{-1}(\pi'(C))$; hence $\sigma=\psi^{-1}\sigma' \psi=\varphi\sigma'\varphi^{-1}\in \Aut(\TT_b)$. Since $\sigma$ and $\sigma'$ act on $\p^1\times \p^1$ by the exchange of the two factors, $\mathrm{pr}_1\pi\colon \TT_b\to \p^1$ is also an $\F_b$-bundle. The union of the $(-b)$-curves of the $\F_b$ is a surface $S_1=\sigma(S_2)\subset \TT_b$, which is then sent by $\psi$ onto the surface $S_1'\subset \FF_0^{m,m}$ given by $x_0=0$. The two surface $S_1',S_2'\subset \FF_0^{m,m}$ are disjoint and also disjoint from $C'$.  Their strict transforms on $X$ are then again disjoint, and their images on $\TT_b$ intersect only along $C$. This yields~\ref{S1S2TTb}.

To prove~\ref{CommDiag}, it remains to show that $\Autz(\TT_{b})\simeq \PGL_2$ and that every automorphism of the $\p^1\!$-bundle $\TT_b\to \p^1\times \p^1$ is trivial. Every element of $\Autz(\TT_{b})$ permutes the fibres of the two $\F_b$-bundles $\mathrm{pr}_1\pi,\mathrm{pr}_2\pi\colon\TT_b\to \p^1$, so $S_1$ and $S_2$ are both invariant, and the same holds for $C=S_1\cap S_2$. Since $\pi(C)=\Delta\subset\p^1\times \p^1$, the image of $\Autz(\TT_{b})\to \Autz(\p^1\times\p^1)$ is the diagonal $\PGL_2$. It then suffices  to see that every automorphism of the $\p^1\!$-bundle $\TT_{b}\to \p^1\times\p^1$ is trivial. This amounts to showing that every automorphism of the $\p^1\!$-bundle $\FF_0^{m,m}\to \p^1\times\p^1$ that fixes $C'$ is trivial. Indeed, every automorphism of the $\p^1\!$-bundle $\FF_0^{m,m}\to \p^1\times\p^1$ preserves $S_1'$ and $S_2'$, so is of the form 
\[[x_0:x_1;y_0:y_1;z_0:z_1]\mapsto [\lambda x_0:x_1;y_0:y_1;z_0:z_1]\] for some $\lambda\in \k^*$. It preserves $C'$ if and only if $\lambda=1$.

We finish the proof by proving~\ref{HatS:OnlyInvariantCurve}. It follows from the construction that $C=S_1\cap S_2\subset \TT_b$ is invariant by $\Autz(\TT_b)$. It remains to show that every curve $\ell\subset \TT_b$ invariant by $\Autz(\TT_b)$ is equal to $C$ when $\car(\k)$ does not divide $b+1$. The action of $\Autz(\TT_b)$ on $\p^1\times \p^1$ being the diagonal action of $\PGL_2$, we find that $\pi(\ell)$ is the diagonal $\Delta$ and thus have to see that $C$ is the only curve invariant by the action of $\Autz(\TT_b)\simeq \PGL_2$ on the surface $V=\pi^{-1}(\Delta)$, which is a $\p^1\!$-bundle $V\to \Delta$. To do this, it suffices to find an isomorphism $V\iso \p^1\times \p^1$ that sends $C$ onto the diagonal. Indeed, the action of $\PGL_2$ on $\p^1\times \p^1$ will then have to be the diagonal action, which only preserves the diagonal.

We restrict the transition function of $\TT_b$ to $V$ and get
$$
\big([x_0:x_1;1:z],z\big)\mapsto \left([x_0:x_1z^{b+2}+x_0(b+1) z^{b+1};1:z],\frac{1}{z}\right)=\left([x_0:x_1z^2+x_0(b+1) z;\frac{1}{z}:1],\frac{1}{z}\right).$$
The curve $C$ corresponds to $x_0=0$ on both charts. The isomorphism $V\iso \p^1\times \p^1$ can then be chosen on the two charts as
\[\begin{array}{rcl}
\big([x_0:x_1;1:z],z\big)&\mapsto& \big([x_1:x_0(b+1)+x_1z],[1:z]\big),\\
\big([x_0:x_1;z:1],z\big)&\mapsto &\big([-(b+1)x_0+x_1z:x_1],[z:1]\big),
\end{array}\]
which is an isomorphism since $b+1\not=0$ (as we assumed that $\car(\k)$ does not divide $b+1$).
\end{proof}

\begin{remark}\label{Rem:CharPSchwarz}
When $\car(\k)$ divides $b+1$, there are actually two curves invariant by $\Autz(\hat{S}_b)\simeq \PGL_2$.
\end{remark}

\begin{remark} \label{rk PGLn mod Mun}
Let $m=b+1$. An element $g=\begin{bsmallmatrix}
\alpha  & \beta \\ \gamma & \delta
\end{bsmallmatrix}\in \PGL_2$ sends $p_0=[1:1;1:0;0:1]\in \FF_0^{m,m}$ to $[1:1;\alpha:\gamma;\beta :\delta ]$. Therefore, $g \cdot p_0=p_0$ if and only if $\beta=\gamma=0$ and $\alpha^m=\delta^m=1$, and so $H:=\mathrm{Stab}_{p_0}\PGL_2=\left\{ \begin{bsmallmatrix}
\alpha & 0\\ 0 &\alpha^{-1}
\end{bsmallmatrix}; \ \alpha^m=1\right \}$. As $\dim (\PGL_2/H)=3$, we see that $\PGL_2/H$ is a dense open orbit for the diagonal action of $\PGL_2$ on $\FF_{0}^{m,m}$. As $\TT_b$ is $\PGL_2$-equivariantly birational to $\FF_{0}^{m,m}$, the same holds for~$\TT_b$. 
\end{remark} 

\subsection{Umemura \texorpdfstring{$\P^1$}{P1}-bundles}\label{Sec:Umemurabundles}
In this section we introduce a new class of non-decomposable $\p^1\!$-bundles on Hirzebruch surfaces. To the best of the authors' knowledge, those appeared for the first time in the work of Umemura~\cite[\S~10]{Ume88}, and that is the reason why we chose to call them Umemura bundles.

\begin{definition}[Umemura bundles]  \label{def Umemura bundle}
Let $a,b\ge 1$ and $c\ge 2$ be such that $c=ak+2$ with $0\le k\le b$. We 
call \emph{Umemura $\p^1\!$-bundle} the $\p^1\!$-bundle $\U_{a}^{b,c} \to \F_a$ 
given by $\U_{a}^{b,c}=Z_{a}^{b,c,P}\to \F_a$ with $P=y_0^ky_1^{b-k}z^{c-1}$ .
\end{definition}

\begin{remark}\label{Rem:UmeExplicit}
Recall (see Notation~\ref{Not:ZabcP}) that $\U_{a}^{b,c}=Z_{a}^{b,c,P}$ is
obtained by the gluing of two copies of $\F_b\times \A^1$  along $\F_b \times \A^1 \setminus \{0\}$ by the automorphism $\nu\in \Aut(\F_b \times \A^1 \setminus \{0\})$,
\[\begin{array}{ccl}\nu\colon([x_0:x_1;y_0:y_1],z) &\mapsto &\left([x_0:x_1z^{c}+x_0 y_0^ky_1^{b-k}z^{c-1};y_0z^a:y_1],\frac{1}{z}\right)\\
&=&\left([x_0:x_1z^{c-ab}+x_0 y_0^ky_1^{b-k}z^{c-ab-1};y_0:y_1z^{-a}],\frac{1}{z}\right),\end{array}
\]
and that  $\U_{a}^{b,c} \to \F_a$ sends $([x_0:x_1;y_0:y_1],z)\in \F_b\times \A^1$ onto, respectively,  $[y_0:y_1;1:z]\in \F_a$ and $[y_0:y_1;z:1]\in \F_a$ on the two charts. It then has  numerical invariants $(a,b,c)$.
\end{remark}

\begin{lemma}\label{Lem:SurjectiveActionUmemura}Let $a,b\ge 1$ and $c\ge 2$ be such that $c=ak+2$ with $0\le k\le b$. The morphism $\pi\colon\U_{a}^{b,c}\to \F_a$ yields a surjective group homomorphism
\[\rho\colon\Autz(\U_{a}^{b,c})\twoheadrightarrow\Autz(\F_a).\]
\end{lemma}

\begin{proof}
The statement corresponds to showing that the element of $\M_{a}^{b,c}$ corresponding to $\U_{a}^{b,c}$ is fixed by the whole group $\Autz(\F_a)$. We only need to check it for generators of $\Autz(\F_a)$, using Lemma~\ref{Lemm:ActionAutFaParSpace}.

Recall that the polynomial $P$ is given by $P(y_0,y_1,z)=\sum_{i=0}^b y_0^iy_1^{b-i}P_i(z)z^{ai+1}$, with $P_i\in \k[z]_{\le c-2-ai}$ (see Notation~\ref{Not:ZabcP}), where  all $P_i$ are zero except one, namely $P_k$, which is equal to $1$. Lemma~\ref{Lemm:ActionAutFaParSpace}\ref{ActionGL2} shows that $\GL_2$ fixes the class of $\U_{a}^{b,c}$ in $\M_{a}^{b,c}$ since $P_k$ is sent onto $\sigma(P_k)\in \k[z]_{\le c-2-ak}=\k[z]_{\le 0}=\k$. Lemma~\ref{Lemm:ActionAutFaParSpace}\ref{ActionFap} then shows that an automorphism $[y_0:y_1;z_0:z_1])=[ y_0:y_1+y_0 R(z_0,z_1):z_0:z_1]$ of $\F_a$ sends $P$  onto the polynomial $P'$ equivalent to the polynomial $\hat{P}(y_0,y_1,z)$ that satisfies $y_0^ky_1^{b-k}z^{c-1}=P(y_0,y_1,z)=\hat{P}(y_0,y_1+y_0R(z,1),z)$. The polynomial $\hat{P}_k$ is then equal to $P_k=1$, and all the polynomials $\hat{P}_i$ with $i< k$ are zero and thus equal to $P_i$. The other coefficients do not appear in the equivalence class since $c-ai-2<0$ for these.
\end{proof}

\begin{remark} \label{rk aut vert Umemura bundles}
With the notation above, the group of automorphisms of $\p^1\!$-bundles of $\U_{a}^{b,c}$ identifies with the vector group $\bigoplus_{i=1}^{b} k[z_0,z_1]_{ai-c}$ whose action on $\U_{a}^{b,c}$ can be described on the first chart by the following biregular action:  
\[\begin{array}{ccc}
\F_b\times \A^1& \to & \F_b\times \A^1\\
\big([x_0:x_1; y_0:y_1],z\big)&\mapsto& \left([x_0:x_1+x_0 \sum_{i=1}^{b} y_0^i y_1^{b-i} q_i(z,1); y_0:y_1],z\right),\end{array}\]
where $Q=(q_i)_{i=1,\ldots,b} \in \bigoplus_{i=1}^{b} k[z_0,z_1]_{ai-c}$.
This can be computed for instance by using the transition function of $\U_{a}^{b,c}$.
The action on $\pi^{-1}(s_{-a})$, which corresponds to $y_0=0$ on both charts, is then trivial.
\end{remark}

\begin{remark}\label{Ume:ActionGL2}
The group $\GL_2$ acts on $\U_a^{b,c}$ by acting rationally on both charts via, respectively, 
\[\begin{array}{ccc}
\F_b\times \A^1& \hspace{-0.2cm}\to &\hspace{-0.3cm} \F_b\times \A^1\\
([x_0:x_1;y_0:y_1],z)&\hspace{-0.2cm}\mapsto &\hspace{-0.3cm}\left([x_0: x_1\frac{(\beta z+\alpha)^c}{\alpha \delta-\beta \gamma}\!+\!x_0 \frac{\beta(\beta z+\alpha)^{c-1}y_0^k y_1^{b-k}}{\alpha \delta-\beta \gamma} ; y_0 (\beta z+\alpha)^a; y_1],\frac{\delta z +\gamma}{\beta z +\alpha}\right),\\[1ex]
([x_0:x_1;y_0:y_1],z)&\hspace{-0.2cm}\mapsto &\hspace{-0.3cm}\left([x_0: x_1\frac{(\gamma z+\delta)^c}{\alpha \delta-\beta \gamma}\!-\!x_0 \frac{\gamma(\gamma z+\delta)^{c-1}y_0^k y_1^{b-k}}{\alpha \delta-\beta \gamma} ; y_0 (\gamma z+\delta)^a; y_1],\frac{\alpha z +\beta}{\gamma z +\delta}\right)\end{array}\]
as can directly be checked using the transition function.
\end{remark}

\subsection{Invariant fibres}  \label{subsec:invariant fibers}
The following result shows that one can reduce the study of $\p^1\!$-bundles $X\to \F_a$ to the case where the action of $\Autz(X)$ on $\F_a$ is transitive on the set on fibres of $\tau_a\colon \F_a\to \p^1$. This is in particular the case when the action on $\F_a$ yields a surjective group homomorphism $\Autz(X)\twoheadrightarrow\Autz(\F_a)$ (see Remark~\ref{Rem:AutFaP1transitiveFibres}) and holds for decomposable $\p^1\!$-bundles (see Lemma~\ref{Lem:AutDecOnAutFa}), for the $\p^1\!$-bundles $\TT_b\to \p^1\times \p^1$ (see Lemma~\ref{Lemm:AutTTb}), and for Umemura bundles (see Lemma~\ref{Lem:SurjectiveActionUmemura}).

\begin{lemma}\label{Lemm:fibreInvariant}
Let $\pi\colon X\to \F_a$ be a $\p^1\!$-bundle. If there is a point $p\in \p^1$ such that the surface $\tau\pi^{-1}(p)\subset X$ is invariant by $\Autz(X)$ $($where $\tau_a\colon \F_a\to \p^1$ is the standard $\p^1\!$-bundle$)$, then there exist a decomposable bundle $\FF_a^{b,c}\to \F_a$ and a commutative diagram
\[\xymatrix@R=3mm@C=2cm{
    X \ar@{-->}[r]^{\psi} \ar[rd]_{\pi}  & \FF_a^{b,c} \ar[d] \\
    & \F_a\rlap{\,,}
  }\]
  where $\psi$ is a birational map satisfying $\psi\Autz(X)\psi^{-1}\subsetneq \Autz(\FF_a^{b,c})$.
\end{lemma}

\begin{proof}
The surface $S=\F_a\setminus (\tau_a)^{-1}(p)$ is isomorphic to $\p^1\times \A^1$, and $U=\pi^{-1}(S)$ is invariant by $\Autz(X)$. Applying Proposition~\ref{Prop:NoMoreJumpingfibre}, one can perform finitely many $\Autz(X)$-equivariant birational maps and reduce to the case where $\tau_a  \pi\colon U\to \p^1\setminus \{b\}$ is a trivial $\F_b$-bundle. We then find an integer $c\ge 0$ and an inclusion $U\hookrightarrow \FF_a^{b,c}$ such that the action of $\Autz(X)$ extends to a biregular action on $\FF_a^{b,c}$. This will yield a birational map as above, satisfying $\psi\Autz(X)\psi^{-1}\subset \Autz(\FF_a^{b,c})$. Moreover,  equality does not hold since $\Autz(\FF_a^{b,c})$ acts transitively on the set of fibres of $\tau_a\colon \F_a\to \p^1$ (see Lemma~\ref{Lem:AutDecOnAutFa} and Remark~\ref{Rem:AutFaP1transitiveFibres}).

If $b>0$, the action of every element $g\in\Autz(X)$ on $U$ corresponds  to an automorphism of $\F_b\times \A^1$ of the form
\[\begin{array}{rcl}
\big([x_0:x_1;y_0:y_1],z\big)  &\mapsto & \big([x_0:\alpha x_1+x_0\sum_{i=0}^b y_0^iy_1^{b-i} p_i(z);
f_1(y_0,y_1,z),f_2(y_0,y_1,z)],az+b\big)\end{array}\]
for some $\alpha\in \k^*$ and $p_i\in \k[z]$,  where $f_1,f_2\in \k[y_0,y_1,z]$ correspond to the coordinates of the restriction of an automorphism of $\F_a$. The group $\Autz(X)$ being an algebraic group, the degrees of the polynomials $p_i$ are bounded by an integer which does not depend on $g\in \Autz(X)$. It then suffices to take an integer $c>0$ big enough and to send $([x_0:x_1;y_0:y_1],z)\in \F_b\times \A^1$ to $[x_0:x_1;y_0:y_1;z:1]\in \FF_a^{b,c}$ to be able to extend the action of all elements $g\in \Autz(X)$ to $\FF_a^{b,c}$. (This can be checked for instance using the description of $\FF_{a}^{b,c}$ provided in \S~\ref{Sec:P1bundlesHirz}.)

If $b=0$, the action of every element $g\in\Autz(X)$ on $U$ corresponds similarly to an automorphism of $\F_0\times \A^1=\p^1\times \p^1\times \A^1$ of the form
\[\begin{array}{rcl}
\big([x_0:x_1;y_0:y_1],z\big)&  \mapsto  &\big([\alpha(z) x_0+\beta(z)x_1:\gamma(z) x_0+\delta(z)x_1;f_1(y_0,y_1,z),f_2(y_0,y_1,z)],az+b\big),\end{array}
\]
where $\alpha,\beta,\gamma,\delta\in \k[z]$ are such that $\begin{bsmallmatrix}
\alpha & \beta\\
\gamma& \delta\end{bsmallmatrix}\in \PGL_2(\k[z])$ $($\textit{i.e.}~$\alpha\delta-\beta\gamma\in \k^*)$ and where $f_1,f_2\in \k[y_0,y_1,z]$ correspond to the coordinates of the restriction of an automorphism of $\F_a$.

The morphism $U\to \p^1\times \A^1$, $([x_0:x_1;y_0:y_1],z)\mapsto ([x_0:x_1],z)$ then yields an algebraic group homomorphism $\Autz(X)\to \Aut(\p^1\times \A^1)$. The image $H\subset \Aut(\p^1\times \A^1)$ is then a connected algebraic subgroup of $\Bir(\p^1\times \A^1)$ that preserves the set of fibres $\p^1\times \A^1$. There thus exist an integer $c\ge 0$ and an inclusion $\p^1\times \A^1\to \F_c$ which allows the image to extend. We then find  an inclusion $U\hookrightarrow \FF_a^{b,c}$ which allows the action of $\Autz(X)$ to extend.\end{proof}

\begin{lemma} \label{lem:subgroups of PGL2}
Let $H$ be a maximal proper connected subgroup of $\PGL_2 \times \PGL_2$. Then $H$ is $B \times \PGL_2$ or $\PGL_2 \times B$, where $B \subset \PGL_2$ is a Borel subgroup $($conjugate to the group of upper-triangular matrices$)$, or $H$ is isomorphic to $\PGL_2$. In the latter case, $H$ is conjugate to the diagonal embedding of $\,\PGL_2$ in $\PGL_2 \times \PGL_2$.
\end{lemma}

\begin{proof}
  Let $p_1$ and $p_2$ be the two natural projections $\PGL_2 \times \PGL_2 \to \PGL_2$. If $p_i(H) \subsetneq \PGL_2$, then $p_i(H)$ is contained in a Borel subgroup $B$ of $\PGL_2$ (\textit{e.g.}\ by~\cite[\S~30.4, Theorem~(a)]{Hum75}), and so $H$ is $B \times \PGL_2$ or $\PGL_2 \times B$ since $H$ is maximal.

  We now assume  $p_i(H)=\PGL_2$ for $i=1,2$. Let $K$ be the kernel of ${p_1}_{|H}:\; H \to \PGL_2$. As $K$ is a normal subgroup of $H$ and ${p_2}_{|H}$ is onto, $p_2(K)$ is a normal subgroup of $\PGL_2$. As $H$ is a proper subgroup of $\PGL_2 \times \PGL_2$, we must have $p_2(K)=\{1\}$, and so $K=\{1\}$ and ${p_1}_{|H}$ is a bijective morphism of algebraic groups $H\to \PGL_2$. In particular, $\dim(H)=3$ by~\cite[\S~4.1, Theorem]{Hum75}. 

We now show that $H$ is simple. Let $N$ be a normal subgroup of $H$. As ${p_i}_{|H}$ is onto, $p_i(N)$ is a normal subgroup of $\PGL_2$. If $p_i(N)=\PGL_2$ for some $i$, then $N=H$ (as they have the same dimension). Otherwise, $p_1(N)=p_2(N)=\{1\}$, and so $N=\{1\}$. This achieves to prove that $H$ is a simple group.

As $\dim(H)=3$ and $H$ is of rank at most $2$, the classification of simple root systems yields that $H$ is isomorphic to $\SL_2$ or $\PGL_2$; see~\cite[\S~32 and Appendix]{Hum75}. Since ${p_i}_{|H}: H \to \PGL_2$ is a bijective morphism, $H$ cannot be isomorphic to $\SL_2$ (otherwise $\pm I_2$ would be sent to $I_2$), and so $H \simeq \PGL_2$. In this case, $H$ is conjugate to the diagonal embedding of $\PGL_2$ in $\PGL_2 \times \PGL_2$. Indeed, this follows from the fact that a bijective morphism of algebraic groups is an isomorphism in characteristic zero (\textit{e.g.}\ by Zariski's main theorem), together with the fact that all automorphisms of $\PGL_2$ are inner~\cite[Theorem~1.3 and \S~3.6]{SGA3Ep24}.
\end{proof}

\begin{remark}\label{Rem:PGL2positive}
In characteristic $p>0$, the above result is false: we get infinitely many embeddings of $\PGL_2$ into $\PGL_2\times \PGL_2$ with pairwise distinct images, up to conjugation, given by $\begin{bsmallmatrix}
  a & b \\ c & d
 \end{bsmallmatrix}\mapsto \left(\begin{bsmallmatrix}
  a & b \\ c & d
 \end{bsmallmatrix}, \begin{bsmallmatrix}
  a^{p^n} & b^{p^n} \\ c^{p^n} & d^{p^n}
 \end{bsmallmatrix}\right)$ for $n\in \Z$.
\end{remark}

\begin{proposition}  \label{prop decompo and Umemura bundles}
Let $a\ge 0$, and let $\pi\colon X \to \F_a$ be a $\p^1\!$-bundle.
Then, there exist $b,c\in \Z$ such that one of the following holds: 
\begin{enumerate}
\item \label{OKDec}
$X$ is isomorphic to a decomposable $\p^1\!$-bundle $\FF^{b,c}_a \to \F_a$  $($Definition~$\ref{def Fabc})$;
\item \label{OKUme}
$X$ is isomorphic to an Umemura $\p^1\!$-bundle $\U_{a}^{b,c} \to \F_a$  $($Definition~$\ref{def Umemura bundle})$;
\item \label{OKSch}
we have $a=0$, and $(X,\pi)$ is square isomorphic to the $\p^1\!$-bundle $\TT_{b} \to \F_0$ of Definition~$\ref{def hat Schwrzenberger}$; or
\item\label{FixFibreFibre}
there exist a $\p^1\!$-bundle $\tau\colon\F_a\to \p^1$ and a closed point $p\in \p^1$  such that $(\tau\pi)^{-1}(p)$ is invariant by~$\Autz(X)$.
\end{enumerate}
In cases~\ref{OKSch} and~\ref{FixFibreFibre}, there exist a decomposable bundle $\FF_a^{b,c}\to \F_a$ and a commutative diagram
\[\xymatrix@R=3mm@C=2cm{
    X \ar@{-->}[r]^{\psi} \ar[rd]_{\pi}  & \FF_a^{b,c} \ar[d] \\
    & \F_a\rlap{,}
  }\]
  where $\psi$ is a birational map satisfying $\psi\Autz(X)\psi^{-1}\subsetneq \Autz(\FF_a^{b,c})$.
\end{proposition}

\begin{proof}
Lemma~\ref{Lemm:fibreInvariant} and Lemma~\ref{Lemm:AutTTb}~\ref{CommDiag} give the existence of the birational map $\psi\colon X\dasharrow  \FF_a^{b,c}$ in cases~\ref{OKSch} and~\ref{FixFibreFibre}, with $\psi\Autz(X)\psi^{-1}\subset \Autz(\FF_a^{b,c})$. Moreover, we cannot have equality since the action of $\Autz(\FF_a^{b,c})\to \Aut(\F_a)$ is surjective (see Lemma~\ref{Lem:AutDecOnAutFa}). It remains to show that we can reduce to the  four cases above.

We denote by $H\subset \Autz(\F_a)$ the image of $\Autz(X)$ by the natural homomorphism $\Autz(X) \to \Autz(\F_a)$ (see Lemma~\ref{blanchard}).  If the preimage of one fibre of $\tau_a\colon \F_a\to\p^1$ is invariant by $\Autz(X)$, we get case~\ref{FixFibreFibre}. We can in particular assume that all fibres of $\tau_a\pi\colon X\to \p^1$ are isomorphic to $\F_b$ for the same $b\ge 0$ (no jumping fibre; see Proposition~\ref{Prop:NoMoreJumpingfibre}). Hence, $X\to\p^1$ has numerical invariants $(a,b,c)$ for some $c\in \Z$, which is positive if $b=0$ (see Proposition~\ref{Prop:EquivABC} and Definition~\ref{Defi:NumInv}). We can moreover assume that $X \to \F_a$ is not a decomposable bundle since otherwise we obtain case~\ref{OKDec}. This implies that $b\ge 1$ and $c\ge 2$ and that $X\to \F_a$ is isomorphic to 
$Z_{a}^{b,c,P}\to \F_a$, where $P(y_0,y_1,z)=\sum_{i=0}^b  {y_0}^i{y_1}^{b-i}P_i(z)z^{ai+1}$ and $P_i(z)\in \k[z]_{\le c-2-ai}$  for $i=0,\ldots,b$ (see Corollary~\ref{Coro:IsoClass}).

For each $h \in \Autz(\F_a)$, we have $h \in H$ if and only if there exists an $\hat{h} \in \Autz(X)$ such that the following diagram commutes: 
\begin{equation}\label{eq:diag_commutes}
\xymatrix@R=3mm@C=2cm{
    X \ar[r]^{\hat h} \ar[d]_\pi  & X \ar[d]^\pi \\
    \F_a \ar[r]_h & \F_a\rlap{.}
  }
  \end{equation}
This is thus equivalent to asking that the class of $X\to \F_a$ in $\M_{a}^{b,c}$ (see Corollary~\ref{Cor:modspace}) is fixed by the action of~$h$.  

We now consider two cases, depending on whether $a\ge 1$ or $a=0$.

{\it Case  $a \geq 1$}. We denote by $i_0$ the smallest integer such that $P_{i_0}\not=0$. For each element $\overline\sigma\in\PGL_2$, there exists an element of $H\subseteq \Aut(\F_a)$ whose action on $\p^1$, via $\tau$, corresponds to $\overline\sigma$ in $\PGL_2$. We can write this element as $\varphi_1 \varphi_2$, where $\varphi_1,\varphi_2\in \Aut(\F_a)$ are given by
\[\begin{array}{cccccc}
\varphi_1& \colon& [y_0:y_1;z_0:z_1]&\mapsto & [y_0:y_1;\alpha z_0+\beta z_1:\gamma z_0+\delta z_1],\\[1ex]
\varphi_2& \colon& [y_0:y_1;z_0:z_1]&\mapsto & [ y_0:y_1+y_0R(z_0,z_1):z_0:z_1],\end{array}\]
where $\sigma=\begin{bsmallmatrix} \alpha&\beta\\ \gamma&\delta\end{bsmallmatrix} \in \GL_2$ represents the class $\overline\sigma\in \PGL_2$ and $R\in \k[z_0,z_1]_a$. Lemma~\ref{Lemm:ActionAutFaParSpace} describes the action of $\varphi_1$ and $\varphi_2$ on the class $[\pi]$. The element $\varphi_2$ does not change the polynomial $P_{i_0}$ (see Lemma~\ref{Lemm:ActionAutFaParSpace}\ref{ActionFap}), so $\sigma(P_{i_0})$ has to be equal to a multiple of $P_{i_0}$, for the action of $\GL_2$ on $\k[z]_{\le c-ai_0}$ given in Definition~$\ref{GL2 action}$ (see Lemma~\ref{Lemm:ActionAutFaParSpace}\ref{ActionGL2}). This implies that the class of $P_{i_0}$ in $\P(\k[z]_{\le c-2-ai_0})$ is fixed by the corresponding action of $\PGL_2$. This happens if and only if $(\k[z]_{\le c-2-ai_0})^{\SL_2}$ is non-zero. As $\k[z]_{\le c-2-ai_0}$ is an irreducible $\SL_2$-representation (see Remark~\ref{Remark rep is Sb(V)}), $(\k[z]_{\le c-2-ai_0})^{\SL_2}$ is non-zero if and only if $\k[z]_{\le c-2-ai_0}$ is the trivial representation; that is, $c-2=a i_0$ and $P_{i_0}$ is a constant polynomial. Since $i_0$ is the smallest integer such that $P_{i_0} \neq 0$, we have $P_i=0$ for all $i<i_0$. Moreover, $P_i=0$ for $i>i_0$ since $c-2-ai<0$. This implies that $P(y_0,y_1,z)= \lambda y_0^{i_0} y_1^{b-i_0}  z^{a i_0+1}$, and so that $X\to \F_a$ is an Umemura bundle (see Definition~\ref{def Umemura bundle}). 

{\it Case  $a = 0$}. Let $a=0$ and $\pi: X\to \p^1\times \p^1$ be a $\p^1\!$-bundle. If $H=\PGL_2 \times \PGL_2=\Autz(\F_0)$, then the moduli space $\M_{0}^{b,c}$ must contain a fixed point for the natural $H$-action, described in Corollary~\ref{cor action on Mabc}\ref{GL2SVVp}. This cannot happen since the $\SL_2 \times \SL_2$-representation $\k[y_0,y_1]_b \otimes k[z]_{\leq c-2}$ is irreducible, and non-trivial when $b\geq 1$ (as we assumed above).

Now assume that $H \subsetneq \PGL_2 \times \PGL_2$ is a proper subgroup. If one of the two projections $H\to \PGL_2$ is not onto, then one fibre of a projection $\p^1\times \p^1\to \p^1$ is invariant, and we get case~\ref{FixFibreFibre}. We then assume that $H$ surjects onto $\PGL_2$ via both projections. By Lemma~\ref{lem:subgroups of PGL2}, $H$ is conjugate to $H_\Delta:=\{(h,h) \ |\ h \in \PGL_2\}$ in $\PGL_2 \times \PGL_2$.

If $H$ is conjugate to $H_\Delta$, then the moduli space $\M_{0}^{b,c}$ contains a fixed point for the natural $H_\Delta$-action. By Corollary~\ref{cor action on Mabc}\ref{GL2SVVp}, $\M_{0}^{b,c}\simeq \P(\Hom((k[y_0,y_1]_b)^*,k[z]_{\leq c-2})$ as an $H_\Delta$-variety, where we identify $H_\Delta = \PGL_2$. Hence $\M_{0}^{b,c}$ contains a fixed point if and only if $\Hom^{\SL_2}((k[y_0,y_1]_b)^*,k[z]_{\leq c-2}) \neq \{0\}$. As the $\SL_2$-representations $(k[y_0,y_1]_b)^*$ and $\k[z]_{\leq c-2}$ are irreducible and of dimension $b+1$ and $c-1$, respectively, it follows from Schur's lemma that $\Hom^{\SL_2}((k[y_0,y_1]_b)^*,k[z]_{\leq c-2})=0$ when $b \neq c-2$. On the other hand, if $b=c-2$, then $\Hom^{\SL_2}((k[y_0,y_1]_b)^*,k[z]_{\leq c-2})=\{ \lambda Id; \lambda \in k\}$, and so $\M_{0}^{b,b+2}$ has a unique fixed point corresponding to the identity; the latter is given by  $P(y_0,y_1,z)=\sum_{i=0}^{b} y_0^i y_1^{b-i} z^{i+1}$ (see Corollary~\ref{cor action on Mabc}\ref{GL2SVVp}) and yields case~\ref{OKSch}.
\end{proof}

\begin{remark}\label{Rem:CharP}
In the proof of Proposition~\ref{prop decompo and Umemura bundles}, the assumption that the base field $\k$ is of characteristic zero is required. Indeed, the results from the representation theory of $\SL_2$ that we use in the proof are not valid in positive characteristic. In positive characteristic, there are actually more $\p^1\!$-bundles to consider.
\end{remark}

\section{\texorpdfstring{$\boldsymbol{\p^1}$}{P\textasciicircum 1}-bundles over \texorpdfstring{$\boldsymbol{\p^2}$}{P\textasciicircum 2}} \label{section: P1 bundles over P2}
The results in \S~\ref{decomposable_over_P2} are valid over an algebraically closed field $\k$ of  arbitrary characteristic, but in \S~\ref{schwarzenberger_over_P2} we need to assume $\car(k) \neq 2$ (due to the fact that we work with a quadratic form and need $2$ to be invertible). In \S~\ref{subsec:four families}, Lemma~\ref{Lem:LargeActionP2} and Proposition~\ref{Prop:P1bundlesP2} both rely  on Proposition~\ref{prop decompo and Umemura bundles}, and so are valid only in characteristic zero, while Lemma~\ref{SubgroupsPGL3} holds in characteristic different from $2$.

\subsection{Decomposable bundles over \texorpdfstring{$\boldsymbol{\p^2}$}{P2}}\label{decomposable_over_P2}

In this section, we give an explicit description of the decomposable $\p^1\!$-bundles over $\p^2$, similar to the one provided in \S~\ref{SubSec:DecFa} for the $\p^1\!$-bundles over the Hirzebruch surfaces. We also give global coordinates on decomposable $\p^1\!$-bundles over $\p^2$.

\begin{definition} \label{def Pb}
Let $b\in \Z$. Define $\PP_b$ to be the quotient of $(\A^2\setminus \{0\})\times (\A^3\setminus \{0\})$ by the action of $(\G_m)^2$ given by
\[\begin{array}{ccc}
(\G_m)^2 \times (\A^2\setminus \{0\})\times (\A^3\setminus \{0\}) & \to & (\A^2\setminus \{0\})\times (\A^3\setminus \{0\})\\
\big((\mu,\rho), (y_0,y_1;z_0,z_1,z_2)\big)&\mapsto& (\mu\rho^{-b} y_0,\mu y_1;\rho z_0,\rho z_1,\rho z_2).\end{array}\]
The class of $(y_0,y_1,z_0,z_1,z_2)$ will be written $[y_0:y_1;z_0:z_1:z_2]$. The projection 
\[\PP_b\to \p^2, \quad [y_0:y_1;z_0:z_1:z_2]\mapsto [z_0:z_1:z_2]\]
identifies $\PP_b$ with 
\[\P\big(\OPP(b) \oplus \OPP\big)=\P\big(\OPP \oplus \OPP(-b)\big)\]
as a $\p^1\!$-bundle over $\p^2$. As before, we get an isomorphism of $\p^1\!$-bundles $\PP_{b}\simeq \PP_{-b}$ by exchanging $y_0$ with $y_1$ and will then often assume $b\ge0$ in the following.
\end{definition}

\begin{lemma}\label{lemm:surjectiveDecP2}
For each $b\in \Z$, the morphism $\pi\colon\PP_b\to \p^2$ yields a surjective group homomorphism
\[\rho\colon\Autz(\PP_b)\twoheadrightarrow\Aut(\p^2)=\PGL_3.\]
\end{lemma}

\begin{proof}
The existence of $\rho$ is given by Lemma~\ref{blanchard}. The fact that it is surjective can be seen by making $\GL_3$ act naturally on $\PP_b$, do nothing on $y_0,y_1$, and  act naturally on $z_0,z_1,z_2$. 
\end{proof}

\begin{remark} \label{rk: description auto of P1-bundle of dec bundles over P2}
Assume  $b \geq 1$. The group of automorphisms of $\p^1\!$-bundle of $\PP_b$ identifies with the (connected) group 
 \[
\left \{ \begin{bmatrix} 1 &  0 \\ 
            p & \lambda \end{bmatrix} \in \GL_2(\k[z_0,z_1,z_2])\ \middle| \ \lambda \in k^* \text{ and } p \in \k[z_0,z_1,z_2]_b \right \}\]
whose action on $\PP_b$ is as follows:
$$ (\lambda,p) \cdot [y_0:y_1; z_0:z_1:z_2]=[y_0: \lambda y_1+y_0 p(z_0,z_1,z_2); z_0:z_1:z_2].$$
This can be seen  directly from the global description of $\PP_b$ in Definition~\ref{def Pb}, and by using trivialisations on open subsets isomorphic to $\A^2$.
\end{remark}

\subsection{Schwarzenberger \texorpdfstring{$\boldsymbol{\p^1}$}{P1}-bundles over \texorpdfstring{$\boldsymbol{\p^2}$}{P2}}\label{schwarzenberger_over_P2}
In this subsection, we study the Schwarzenberger $\p^1\!$-bundles $\SS_b\to \p^2$, with $b\ge -1$, given by $\P(\kappa_* \O_{\p^1 \times \p^1}(-b-1,0))\to \p^2$ (see Definition~\ref{def:Schwarz}). As we will observe, only the cases $b\ge 1$ are interesting since $\SS_{-1}$ and $\SS_{0}$ are decomposable (see Corollary~\ref{Cor:SmallIsoSchwarzenberger}).

\begin{lemma}\label{Lemm:TransSchwarz}
Denoting by $U_0,U_1\subset \p^2$ the two open subsets 
\[U_0= \{[X:Y:Z] | X \neq 0\} \simeq \A^2,\quad U_1= \{[X:Y:Z] | Z \neq 0\} \simeq \A^2,\] the restriction of $\SS_b$ on $\P^2 \setminus \{ [0:1:0] \}$ is obtained by gluing $\p^1 \times U_0$ and $\p^1 \times U_1$ along $\p^1 \times (U_0 \cap U_1)$ via the isomorphism given by
$$\begin{array}{rcccccc}
\theta \colon & \p^1 \times U_0 &\dasharrow &\p^1 \times U_1 \\
&\left(\begin{bmatrix} x_0 \\ x_1 \end{bmatrix} ,[1:u:v] \right) & \mapsto & \left(\begin{bmatrix}  \alpha_{11}(u,v) & \alpha_{12}(u,v) \\ \alpha_{21}(u,v) & \alpha_{22}(u,v) \end{bmatrix} \begin{bmatrix} x_0 \\ x_1 \end{bmatrix},\left[\dfrac{1}{v}:\dfrac{u}{v} :1\right]\right),\end{array}$$
where $\alpha_{ij}(u,v)\in \k[u,v]$ are the polynomials satisfying
\[ \begin{bmatrix} \alpha_{11}(s+t,st) & \alpha_{12}(s+t,st) \\
 \alpha_{21}(s+t,st)& \alpha_{22}(s+t,st) \end{bmatrix}=\left\{\begin{array}{ll}\begin{bmatrix}1 & 0 \\ 0 & -st\end{bmatrix}  \text{ if }b=-1,\vspace{0.2cm}\\
 \dfrac{1}{s-t}\begin{bmatrix} s^{b}-t^{b} & st(s^{b-1}-t^{b-1}) \\
 s^{b+1}-t^{b+1}& st(s^{b}-t^{b})\end{bmatrix} \text{ if }b\ge0.\end{array}\right.\]
\end{lemma}

\begin{proof}Recall that $\SS_b=\P(\kappa_* \O_{\p^1 \times \p^1}(-m,0))$, where $m=b+1$ and  $\kappa$ is given by
  \[\kappa: \p^1 \times \p^1 \to \P^2,\quad \big([y_0:y_1],[z_0:z_1]\big) \mapsto [y_0 z_0:y_0 z_1+y_1 z_0:y_1z_1]\]
(see Definition~\ref{def:Schwarz}).
The preimages of $U_0,U_1\subset \p^2$ by $\kappa$ are then two open subsets $T_0=\kappa^{-1}(U_0)$ and $T_1=\kappa^{-1}(U_1)$ of $\p^1\times \p^1$ isomorphic to $\A^2$ using the  standard coordinates $s_0=\frac{y_1}{y_0}, t_0=\frac{z_1}{z_0}, s_1=\frac{y_0}{y_1},t_1=\frac{z_0}{z_1}$ in  $\k(\p^1\times \p^1)$: 
\[\begin{array}{l}
T_0= \kappa^{-1}(U_0)= \left\{\big([y_0:y_1],[z_0:z_1]\big)\in \p^1\times \p^1 \mid y_0z_0 \neq 0 \right\} =\Spec\big(\k[s_0,t_0]\big),\\[1ex] 
T_1= \kappa^{-1}(U_1)= \left\{\big([y_0:y_1],[z_0:z_1]\big)\in \p^1\times \p^1 \mid y_1z_1 \neq 0 \right\} =\Spec\big(\k[s_1,t_1]\big).\end{array}
\]
The line bundle $\O_{\p^1 \times \p^1}(-m,0)$ is trivial on $T_0$ and $T_1$,
 so  $\O_{\p^1 \times \p^1}(-m,0)$ is the gluing of two copies of $\A^3$, via the transition function
\[\begin{array}{cccccc}
(\A^1\setminus \{0\}  \times \A^1\setminus \{0\} )\times \A^1 &\iso &(\A^1\setminus \{0\}  \times \A^1\setminus \{0\} )\times \A^1 \\
(s_0,t_0,a_0) & \mapsto & \big(\frac{1}{s_0}, \frac{1}{t_0},a_0 s_0^{m} \big),\end{array}\]
which corresponds to the identifications $a_1=a_0 s_0^{m}$, $s_1=\frac{1}{s_0}$, $t_1=\frac{1}{t_0}$. This transition function implies that a section on $T_0\cap T_1$ correspond on the first chart to an element $f(s_0,t_0)\in \k[{s_0}^{\pm 1},{t_0}^{\pm 1}]$ and on the second chart to $s_1^{-m}f(\frac{1}{s_1},\frac{1}{t_1})\in \k[{s_1}^{\pm 1},{t_1}^{\pm 1}]$.

To compute the transition on $U_0$ and $U_1$, we take standard coordinates $u_0=\frac{Y}{X}$, $v_0=\frac{Z}{X}$, $u_1=\frac{Y}{Z}$, $v_1=\frac{X}{Z}$ on $U_0=\Spec(\k[u_0,v_0])$ and $U_1=\Spec(\k[u_1,v_1])$. We then observe that for $i=1,2$, the morphisms $\kappa_{|_{T_i}}: T_i \to U_i$ corresponds to the injective algebra morphism $\k[U_i]=\k[u_i,v_i] \hookrightarrow \k[T_i]=\k[s_i,t_i]$ that sends $u_i$ and $v_i$ onto $s_i+t_i$ and $s_i t_i$, respectively.

 The space of sections $\kappa_* \O_{\p^1 \times \p^1}(-m,0)(U_i)=\O_{\p^1 \times \p^1}(-m,0)(T_i) \simeq \k[T_i]=\k[s_i,t_i]$ is a free $\k[U_i]$-module of rank $2$ generated by $1$ and $\xi_i=s_i-t_i$. This basis being chosen, it determines a transition function, which is of the form
 \[\begin{array}{ccc}
 U_0\times \A^2 & \dasharrow & U_1\times \A^2\\
 \left(u_0,v_0,\begin{bmatrix} b_0 \\ b_1 \end{bmatrix}\right) & \mapsto & \left(\dfrac{u_0}{v_0},\dfrac{1}{v_0},\begin{bmatrix} \alpha_{11}(u_0,v_0) & \alpha_{12}(u_0,v_0) \\
 \alpha_{21}(u_0,v_0)& \alpha_{22}(u_0,v_0) \end{bmatrix}\cdot\begin{bmatrix} b_0 \\ b_1 \end{bmatrix}\right)\end{array}\]
 for some $\alpha_{ij}\in \k[u_0,{v_0}^{\pm 1}]$. We then take a section on $U_0\cap U_1$, which is given on the first chart by $f_0(u_0,v_0),f_1(u_0,v_0)\in \k[u_0,{v_0}^{\pm 1}]$ and on the second chart by 
 \[\begin{array}{rcl}
 g_0(u_1,v_1)&=&\alpha_{11}\big(\frac{u_1}{v_1},\frac{1}{v_1}\big)f_0\big(\frac{u_1}{v_1},\frac{1}{v_1}\big)+\alpha_{12}\big(\frac{u_1}{v_1},\frac{1}{v_1}\big)f_1\big(\frac{u_1}{v_1},\frac{1}{v_1}\big)\in \k[u_1,{v_1}^{\pm1}],\vspace{0.1cm}\\ 
 g_1(u_1,v_1)&=&\alpha_{21}\big(\frac{u_1}{v_1},\frac{1}{v_1}\big)f_0\big(\frac{u_1}{v_1},\frac{1}{v_1}\big)+\alpha_{22}\big(\frac{u_1}{v_1},\frac{1}{v_1}\big)f_1\big(\frac{u_1}{v_1},\frac{1}{v_1}\big)\in \k[u_1,{v_1}^{\pm1}].\end{array}\] 
The corresponding section on $T_0\cap T_1$ is then given  on the two charts by  
\[\begin{array}{rcll}
f(s_0,t_0)&=&f_0(s_0+t_0,s_0t_0)+(s_0-t_0)f_1(s_0+t_0,s_0t_0)&\in \k[s_0^{\pm1},t_0^{\pm 1}],\\
g(s_1,t_1) &=&s_1^{-m}f\big(\frac{1}{s_1},\frac{1}{t_1}\big)\\
&=& g_0(s_1+t_1,s_1t_1) +(s_1-t_1)g_1(s_1+t_1,s_1t_1)&\in \k[s_1^{\pm1},t_1^{\pm 1}].
\end{array}\]
We then use the equalities 
\[\begin{array}{rcl}
2g_0(s_1+t_1,s_1t_1)&=&g(s_1,t_1)+g(t_1,s_1)\in \k\big[s_1^{\pm1},t_1^{\pm 1}\big],\\[1ex]
2g_1(s_1+t_1,s_1t_1)&=&\frac{g(s_1,t_1)-g(t_1,s_1)}{s_1-t_1}\in \k[s_1^{\pm1},t_1^{\pm 1}],\\[1ex]
 g_0\big(\frac{u_0}{v_1},\frac{1}{v_1}\big)&=&\alpha_{11}(u_0,v_1)f_0(u_0,v_1)+\alpha_{12}(u_0,v_1)f_1(u_0,v_1)\in \k[u_0,{v_1}^{\pm1}],\\[1ex] 
 g_1\big(\frac{u_0}{v_1},\frac{1}{v_1}\big)&=&\alpha_{21}(u_0,v_1)f_0(u_0,v_1)+\alpha_{22}(u_0,v_1)f_1(u_0,v_1)\in \k[u_0,{v_1}^{\pm1}]\end{array}\]
to compute the $\alpha_{ij}$:

\[\begin{array}{rcl}
2g_0\big(\frac{s_0+t_0}{s_0},\frac{1}{s_0t_0}\big)&=&g\big(\frac{1}{s_0},\frac{1}{t_0}\big)+g\big(\frac{1}{t_0},\frac{1}{s_0}\big)=s_0^{m}f(s_0,t_0)+t_0^{m}f(t_0,s_0),\\[1ex]
&=&(s_0^m+t_0^m)f_0(s_0+t_0,s_0t_0)+(s_0^m-t_0^m)(s_0-t_0)f_1(s_0+t_0,s_0t_0),\\[1ex]
2g_1\big(\frac{s_0+t_0}{s_0t_0},\frac{1}{s_0t_0}\big)&=&-\big(g\big(\frac{1}{s_0},\frac{1}{t_0}\big)-g\big(\frac{1}{t_0},\frac{1}{s_0}\big)\big)\frac{s_0t_0}{s_0-t_0}=-\big(s_0^{m}f(s_0,t_0)-t_0^{m}f(t_0,s_0)\big)\frac{s_0t_0}{s_0-t_0},\\[1ex]
&=&-s_0t_0\frac{s_0^m-t_0^m}{s_0-t_0}f_0(s_0+t_0,s_0t_0)-s_0t_0(s_0^m + t_0^m)f_1(s_0+t_0,s_0t_0)
\end{array}\]
 and get
\[ \begin{bmatrix} \alpha_{11}(s_0+t_0,s_0t_0) & \alpha_{12}(s_0+t_0,s_0t_0) \\
 \alpha_{21}(s_0+t_0,s_0t_0)& \alpha_{22}(s_0+t_0,s_0t_0) \end{bmatrix}=\frac{1}{2}\begin{bmatrix} s_0^m+t_0^m & (s_0^m-t_0^m)(s_0-t_0) \\
 -s_0t_0\frac{s_0^m-t_0^m}{s_0-t_0}& -s_0t_0(s_0^m + t_0^m) \end{bmatrix}.\]
If $m=0$, we get simply the matrix defined above. If $m\ge 1$, we change the transition function, by observing that 
\[ \begin{array}{ll}
\begin{bmatrix} 1 &\frac{s+t}{st}\\ 0& 2 \end{bmatrix} \begin{bmatrix} s^m+t^m & (s^m-t^m)(s-t) \\
 -st\frac{s^m-t^m}{s-t}& -st(s^m + t^m) \end{bmatrix}\begin{bmatrix} 2 &s+t\\ 0& -1 \end{bmatrix} \vspace{0.2cm}\\
 =\dfrac{4st}{t-s}\begin{bmatrix} s^{m-1}-t^{m-1} & st(s^{m-2}-t^{m-2}) \\
 s^m-t^m& st(s^{m-1}-t^{m-1})\end{bmatrix}.\end{array}\]
 \end{proof}

We then recover the following result already observed by Schwarzenberger~\cite[Proposition~7]{Sch61}.

\begin{corollary}\label{Cor:SmallIsoSchwarzenberger}
We have the following isomorphisms of $\p^1\!$-bundles: $\SS_{-1} \simeq \PP_1=\P(\OPP(1) \oplus \OPP)$, $\SS_0 \simeq \PP_0=\p^1 \times \P^2$, and $\SS_1 \simeq \P(T_{\P^2})$.
\end{corollary}

\begin{proof}
Let us note that to obtain the transition function of the tangent bundle $T_{\P^2}$, it suffices to differentiate the map corresponding to the change of coordinates between two affine charts of $\P^2$. It follows that the transition function of the projectivised tangent bundle over $\P^2 \setminus [0:1:0]$ is 
$$\begin{array}{rcccccc}
& \p^1 \times U_0\cap U_1 &\iso &\p^1 \times U_0\cap U_1\\
&\big([x_0:x_1],[1:u:v]\big) & \mapsto & \big([x_0:x_0 u+x_1 v],\big[\frac{1}{v}:\frac{u}{v}:1\big]\big).\end{array}$$ 

Applying Proposition~\ref{Lemm:TransSchwarz}, the transition functions for $b=-1,0,1$ correspond, respectively, to the matrices
$\begin{bsmallmatrix}  1 & 0 \\ 0 & -v \end{bsmallmatrix}$, $\begin{bsmallmatrix}  0 & -1 \\ 1 & 0 \end{bsmallmatrix}$, and $\begin{bsmallmatrix}  1 & 0 \\ u & v \end{bsmallmatrix}$, which gives the result.
\end{proof}
\begin{remark}\label{PGL3:TP2}One can also see $\SS_1 \simeq \P(T_{\P^2})$ as
\[Z=\left.\left\{\left([x_0:x_1:x_2],[y_0:y_1:y_2]\right)\in \p^2 \times \p^2 \right| \sum x_iy_i=0\right\},\]
with $\pi \colon \SS_1\to \p^2$ the projection onto the first factor. We again find the same transition function, by trivialising the $\p^1\!$-bundle over $U_0$ and $U_1$ via
\[\begin{array}{ccc}
U_0\times \p^1&\rightarrow &Z\\
\big([1:u:v], [x_0:x_1]\big)& \mapsto& \big([1:u:v], [-x_0u-x_1v:x_0:x_1]\big),\\[1ex]
U_1\times \p^1&\rightarrow &Z\\
\big([v:u:1], [x_0:x_1]\big)& \mapsto& \big([v:u:1], [-x_1:x_0:-x_0u+x_1v]\big).\end{array}\]
Seeing the equation with vectors, as $\vphantom{x}^t\!x\cdot y=0$, the group $\PGL_3$ then acts via $(x,y)\mapsto (Ax,\vphantom{A}^t\!A^{-1}y)$.

We will see in Lemma~\ref{Lemm:SchwarzJumpLines}\ref{AutoSchwarzb} that $\Autz(\SS_1)\simeq \PGL_3$, so the above action yields all elements of $\Autz(\SS_1)$ (but not of $\Aut(\SS_1)$, as we also have $(x,y)\mapsto (y,x)$).
\end{remark}

\begin{lemma}\label{Lemm:LiftSchwarzIsSchwarzHat}
For each $b\ge 1$, the following hold:
\begin{enumerate}
\item\label{LiftSSisTT}
  The $\p^1\!$-bundle $\TT_b\to \p^1\times \p^1$ is isomorphic to
  \[\SS_{b}\times_{\p^2} \big(\p^1\times \p^1\big)\to \p^1\times \p^1,\] obtained by pulling back the Schwarzenberger bundle $\SS_{b}\to\p^2$ via the double cover $\kappa\colon\p^1\times \p^1\to \p^2$. In particular, $\TT_b$ is isomorphic to \[\p\big(\kappa^*\big(\kappa_* \O_{\p^1 \times \p^1}(-b-1,0)\big)\big)\to \p^1\times \p^1.\]
\item\label{OneCurveSchwarz}
The only curve invariant by $\Autz(\TT_b)$ is the curve $C\subset \TT_b$, being the preimage of the curve $D\subset \SS_b$ given on the two charts of Lemma~$\ref{Lemm:TransSchwarz}$ by
\[\begin{array}{l}
\left\{\left([x_0 :x_1] ,([1:2t:t^2]) \right)\in \p^1\times U_0 \mid x_0+tx_1=0\right\},\\[1ex]
\left\{\left([x_0 :x_1] ,([t^2:2t:1]) \right)\in \p^1\times U_1 \mid x_0-tx_1=0\right\}.\end{array}\]
\end{enumerate}
\end{lemma}

\begin{proof}
\ref{LiftSSisTT}
We take as usual coordinates $([y_0:y_1],[z_0:z_1])$ on $\p^1\times \p^1$ and denote by $T_0,T_1\subset \p^1\times \p^1$ the open subsets given by, respectively,  $y_0z_0\not=0$ and $y_1z_1\not=0$.

The restriction of $\pi\colon \SS_{b}\times_{\p^2} (\p^1\times \p^1)\to \p^1\times \p^1$ to $T_0\cup T_1=\p^1\times \p^1 \setminus \{([0:1],[1:0]),([1:0],[0:1])\}$ is given by gluing $\p^1\times T_0$ and $\p^1\times T_1$ along $\p^1\times T_0\cap T_1$ by the isomorphism $\theta\in \Aut(\p^1\times T_0\cap T_1)$ induced  by 
$$\begin{array}{ccccccc}
 \p^1 \times T_0 &\dasharrow &\p^1 \times T_1 \\[1ex]
 \left(
 \begin{bmatrix} x_0 \\ x_1 \end{bmatrix} ,\big([1:s],[1:t]\big) \right) & \mapsto & \left(\dfrac{1}{s-t}\begin{bmatrix} s^{b}-t^{b} & st(s^{b-1}-t^{b-1}) \\
 s^{b+1}-t^{b+1}& st(s^{b}-t^{b})\end{bmatrix}\cdot \begin{bmatrix} x_0 \\ x_1 \end{bmatrix},\left(\left[\dfrac{1}{s}:1\right], \left[\dfrac{1}{t}:1\right]\right)\right)\end{array}$$
 (see Lemma~\ref{Lemm:TransSchwarz}). We then define two open embeddings $\iota_i\colon \p^1\times T_i\hookrightarrow \F_b\times \A^1$, $i=0,1$:
\[\begin{array}{ccc}
\p^1\times T_0 & \stackrel{\iota_0}{\hookrightarrow} & \F_b\times \A^1\\
\big([x_0 :x_1],([1:s],[1:t]) \big) &\mapsto & \big([-x_0-tx_1:x_1;1:s],t\big),\\[1ex]
\p^1\times T_1 & \stackrel{\iota_1}{\hookrightarrow} & \F_b\times \A^1\\
\big([x_0 :x_1],([s:1],[t:1]) \big) &\mapsto & \big([-x_0+tx_1:x_1;s:1],t\big)
\end{array}\]
and compute $\iota_1 \theta (\iota_0)^{-1}\in \Bir(\F_b \times \A^1)$:
\[\begin{array}{ccccccc}
\F_b\times \A^1 & \iota_1 \theta (\iota_0)^{-1} & \F_b\times \A^1\\[1ex]
\big([x_0:x_1;1:y],z\big)&\mapsto& \left(\big[\frac{y^{b}}{z} x_0: \frac{y^{b+1}-z^{b+1}}{y-z}x_0+z^{b+1}x_1;\frac{1}{y}:1\big],\frac{1}{z}\right)\\
 & = & \left([x_0:x_1z^{b+2}+x_0 z\sum_{i=0}^b y^iz^{b-i};1:y],\frac{1}{z}\right)
 \end{array}\]
This then yields  $\nu_{b,P} \iota_0=\iota_1 \theta$, where $\nu_{b,P}\in\Aut(\F_b\times \A^1\setminus \{0\})$ is given by 
\[
\nu_{b,P}\colon \big([x_0:x_1;y_0:y_1],z\big) \mapsto \left([x_0:x_1z^{b+2}+x_0 P(y_0,y_1,z);y_0:y_1],\frac{1}{z}\right),\]
with $P(y_0,y_1,z)=\sum_{i=0}^b y_0^iy_1^{b-i}P_i(z)z$ and $P_i(z)=z^i\in \k[z]_{\le b}$ for $i=0,\ldots,b$. This shows that the $\p^1\!$-bundle is isomorphic to $\TT_b\to \p^1\times \p^1$ over $T_0\cup T_1$ (see Definition~\ref{def hat Schwrzenberger}) and thus over the whole $\p^1\times \p^1$ by Lemma~\ref{Lem:Extension}. This achieves the proof of~\ref{LiftSSisTT}.

\ref{OneCurveSchwarz}  By Lemma~\ref{Lemm:AutTTb}\ref{HatS:OnlyInvariantCurve}, there is a unique curve $C\subset \TT_b$ invariant by $\Autz(\TT_b)$ which corresponds to the intersection of $\pi^{-1}(\Delta)$ with the surface $x_0=0$ in both charts $\p^1\times \A^1$. It then corresponds to the curve given in $\p^1\times T_0$ and $\p^1\times T_1$ by 
\[\begin{array}{l}
\big\{\big([x_0 :x_1] ,([1:t],[1:t]) \big)\in \p^1\times T_0 \mid x_0+tx_1=0\big\},\\[1ex]
\big\{\big([x_0 :x_1] ,([t:1],[t:1]) \big)\in \p^1\times T_1 \mid x_0-tx_1=0\big\}.\end{array}\]
Sending this curve into $\SS_b$ yields a bijection $C\to D$, where $D\subset \SS_b$ is given locally as above.
\end{proof}

\begin{lemma}\label{Lemm:SchwarzJumpLines}
Let $\pi\colon \SS_b\to \p^2$ be the $b$-th Schwarzenberger $\p^1\!$-bundle, with $b\ge 1$, and let $\Gamma\subset \p^2$ be the conic which is the ramification locus of $\kappa\colon \p^1\times \p^1\to \p^2$ $($see Definition~$\ref{def:Schwarz})$.
\begin{enumerate}
\item\label{FnLineSchwarz}
If $L\subset \p^2$ is a line, the $\p^1\!$-bundle $\pi^{-1}(L)\stackrel{\pi}{\to} L\simeq \p^1$ is isomorphic to 
\[\begin{array}{lll}
\F_b\to \p^1 & \text{ if $L$ is a tangent line to $\Gamma$;}\\
\F_0\to \p^1 & \text{ if $L$ is a not a tangent line to $\Gamma$ and $b$ is even;\; or}\\
\F_1\to \p^1 & \text{ if $L$ is a not a tangent line to $\Gamma$ and $b$ is odd.}\end{array}
\]
\item\label{AutoSchwarzb}
  The action of $\Autz(\SS_b)$ on $\p^2$ yields a group isomorphism
  \[
\Autz(\SS_b)\iso \left\{\begin{array}{llll}
\Aut(\p^2,\Gamma)&=&\{g\in \Aut(\p^2)\mid g(\Gamma)=\Gamma\}\simeq \PGL_2&\text{ if }b\ge 2;\; or\\
\Aut(\p^2)&\simeq & \PGL_3&\text{ if }b= 1.\end{array}\right.\]
Moreover, every automorphism of the $\p^1\!$-bundle $\SS_b\to \p^2$ is trivial.
\end{enumerate}
\end{lemma}

\begin{proof}Lemma~\ref{Lemm:LiftSchwarzIsSchwarzHat} implies that $\hat\pi\colon\TT_b\to \p^1\times \p^1$ is isomorphic to $ \SS_{b}\times_{\p^2} (\p^1\times \p^1)\to \p^1\times \p^1.$ Lemma~\ref{Lemm:AutTTb}\ref{CommDiag} then yields $\Autz(\TT_b)\simeq \PGL_2$, with an action on $\p^1\times \p^1$ which is the diagonal one, and thus corresponds, via $\kappa\colon \p^1\times \p^1\to \p^2$, to $\Aut(\p^2,\Gamma)\simeq \PGL_2$. This implies, in particular, that the image of the group homomorphism $\Autz(\SS_b)\to \Aut(\p^2)$ (given by Lemma~\ref{blanchard}) contains $\Aut(\p^2,\Gamma)\simeq \PGL_2$. 

Taking a line $L\subset \p^2$ tangent to $\Gamma$, the preimage $\kappa^{-1}(L)\subset \p^1\times \p^1$ consists of two fibres $f_1,f_2$ of the two projections, exchanged by the involution associated to the double cover $\kappa$. In particular, $\pi^{-1}(L)$ is isomorphic to $\hat\pi^{-1}(f_1)\simeq \hat\pi^{-1}(f_2)\simeq \F_b$ since $\TT_b$ has numerical invariants $(0,b,b+2)$.

To finish the proof of~\ref{FnLineSchwarz}, we take a line $L\subset \p^2$ not tangent to $\Gamma$ and show that $\pi^{-1}(L)$ is isomorphic to $\F_0$ or $\F_1$, depending on whether $b$ is even or odd. Using the action of $\Autz(\SS_b)$, we can choose the line of equation $Y=0$. Using the notation of Proposition~\ref{Lemm:TransSchwarz}, the restriction of $\SS_b\to \p^2$ on $\P^2 \setminus \{ [0:1:0] \}$ is given by gluing $\p^1 \times U_0$ and $\p^1 \times U_1$ with
$$\begin{array}{rcccccc}
\theta \colon & \p^1 \times U_0 &\dasharrow &\p^1 \times U_1 \\
&\left(\begin{bmatrix} x_0 \\ x_1 \end{bmatrix} ,[1:u:v] \right) & \mapsto & \left(\begin{bmatrix}  \alpha_{11}(u,v) & \alpha_{12}(u,v) \\ \alpha_{21}(u,v) & \alpha_{22}(u,v) \end{bmatrix} \begin{bmatrix} x_0 \\ x_1 \end{bmatrix},\big[\frac{1}{v}:\frac{u}{v} :1\big]\right),\end{array}$$
where $a_{ij}(u,v)\in \k[u,v]$ are the polynomials satisfying
\[ \begin{bmatrix} \alpha_{11}(s+t,st) & \alpha_{12}(s+t,st) \\
 \alpha_{21}(s+t,st)& \alpha_{22}(s+t,st) \end{bmatrix}=
 \frac{1}{s-t}\begin{bmatrix} s^{b}-t^{b} & st(s^{b-1}-t^{b-1}) \\
 s^{b+1}-t^{b+1}& st(s^{b}-t^{b})\end{bmatrix}.\]

We then observe that $s+t$ divides $\frac{s^n-t^n}{s-t}$ if $n$ is even (by replacing $t$ with $-s$), and that $\frac{s^n-t^n}{s-t}\equiv s^{n-1}\equiv \pm (st)^{\frac{n-1}{2}}\pmod{s+t}$ if $n$ is odd. This yields

\[ \begin{bmatrix} \alpha_{11}(0,v) & \alpha_{12}(0,v) \\
 \alpha_{21}(0,v)& \alpha_{22}(0,v) \end{bmatrix}=
 \left\{\begin{array}{ll}
\begin{bmatrix} 0 & \pm v^{\frac{b}{2}} \\
\pm v^{\frac{b}{2}}& 0\end{bmatrix} & \text{ if $b$ is even,}\vspace{0.1cm}\\
\begin{bmatrix} \pm v^{\frac{b-1}{2}} & 0 \\
 0& \pm v^{\frac{b+1}{2}}\end{bmatrix} & \text{ if $b$ is odd} \end{array}\right.
\]
and achieves the proof of~\ref{FnLineSchwarz}.

To prove~\ref{AutoSchwarzb}, we study the group homomorphism $\rho\colon \Aut(\SS_b)\to \Aut(\p^2)$. As we observed, the image contains $\Aut(\p^2,\Gamma)\simeq \PGL_2$. It is then equal to $\Aut(\p^2,\Gamma)$ if $b\ge 2$ (this follows from~\ref{FnLineSchwarz}). If $b=1$, then $\rho$ is surjective since $\SS_1 \simeq \P(T_{\P^2})$ (see Corollary~\ref{Cor:SmallIsoSchwarzenberger} and Remark~\ref{PGL3:TP2}). To finish the proof of~\ref{AutoSchwarzb}, it remains to show that every automorphism $\alpha$ of the $\p^1\!$-bundle $\SS_b\to \p^2$ is trivial. The lift of $\alpha$ yields an element $\hat\alpha\in \Aut(\TT_b)$ which acts trivially on $\p^1\times \p^1$ and thus is trivial by Lemma~\ref{Lemm:AutTTb}\ref{CommDiag}.
\end{proof}

\begin{remark} \label{rk: embedding PGL2 into PGL3}
Let $\Gamma \subset \p^2$ be the smooth conic of Definition~\ref{def:Schwarz}. Then the natural embedding of $\Aut(\p^2,\Gamma) \simeq \PGL_2$ in $\Aut(\p^2)=\PGL_3$ is the one induced from the injective group homomorphism   
\begin{equation}
  \label{embedding PGL2 in PGL3}
\GL_2(k) \to \GL_3(k),\quad \begin{bmatrix}
a & b \\ c & d
\end{bmatrix} \mapsto \frac{1}{ad-bc}\begin{bmatrix}
a^2 &ab & b^2\\
2ac &ad+bc & 2bd\\
c^2& cd& d^2
\end{bmatrix}.
\end{equation}
Also, $\PGL_2$ acts on $\p^2$ with two orbits, which are $\Gamma$ and its complement. Indeed, the morphism $\kappa$ of Definition~\ref{def:Schwarz} is $\PGL_2$-equivariant, where $\PGL_2$ acts diagonally on $\p^1 \times \p^1$, and as $\P^1 \times \p^1$ is the union of two orbits (namely, the diagonal and its complement), the result follows.
\end{remark}

\begin{remark} \label{rk PGLn mod Dn}
Let $m=b+1 \geq 2$. By Lemma~\ref{Lemm:SchwarzJumpLines}, the group $\PGL_2$ acts on $\SS_b$, and the $(2:1)$-cover $\hat \SS_b \to \SS_b $ is $\PGL_2$-equivariant. Hence, by Remark~\ref{rk PGLn mod Mun}, $\SS_b$ has an open $\PGL_2$-orbit, say $\PGL_2/F$, where $F$ is a finite subgroup scheme that fits into an exact sequence $0 \to H \to F \to \Z/2\Z \to 0$ with $H=\left\{ \begin{bsmallmatrix}
\alpha & 0\\ 0 &\alpha^{-1}
\end{bsmallmatrix}; \ \alpha^m=1\right \}$. Actually, using \eqref{embedding PGL2 in PGL3}, an explicit computation of the stabiliser of a general point in the $\p^1$-fibre over $[0:1:0] \in \p^2$ yields that $F$ is conjugate to $\left \langle H,\ \begin{bsmallmatrix}
0 & 1 \\ -1 & 0
\end{bsmallmatrix} \right \rangle$; see~\cite[Lemma~4.5]{Ume88} for details.  
\end{remark} 

\begin{remark}  \label{rk tangent bundle} Writing $X=\P(T_{\P^2})$, Lemma~\ref{Lemm:SchwarzJumpLines} shows that the action of $\Autz(X)$ on $\p^2$ yields an isomorphism $\Autz(X)\iso\PGL_3$. This classical observation can also be seen as follows. Let $B$ be a Borel subgroup of $G=\PGL_3$, and let $P$ be a maximal parabolic subgroup of $G$ containing $B$. Then the structure morphism $X \to \P^2$ identifies with the projection $G/B \to G/P$ (see \textit{e.g.}~\cite[Example~2.1.8]{IP99}).
  Therefore, by~\cite[Theorem~1]{Dem77}, we have $\Autz(X)=\Autz(G/B)=G$.
\end{remark}

\subsection{Reduction to \texorpdfstring{$\boldsymbol{\P^1}$}{P1}-bundles of the four families} \label{subsec:four families}
We first show that the Schwarzenberger and decomposable $\p^1\!$-bundles over $\p^2$ are those which have large automorphism groups.

\begin{lemma}\label{Lem:LargeActionP2}
Let $\pi\colon X \to \P^2$ be a $\p^1\!$-bundle, let $\Gamma\subset \p^2$ be the conic of equation $Y^2=4XZ$, and let $H\subset \Aut(\p^2)$ be the image of $\Autz(X)$. Then, the following hold:
\begin{enumerate}
\item\label{SurjectiveP2}
$H=\Aut(\p^2)$ if and only if $\pi\colon X \to \P^2$ is a decomposable bundle or is isomorphic to the projectivised tangent bundle.
\item\label{noSurjectiveP2}
$\PGL_2\simeq \Aut(\p^2,\Gamma)\subset H$ if and only if $\pi\colon X \to \P^2$ is  a decomposable bundle or is isomorphic to a Schwarzenberger bundle $\SS_b\to \p^2$ for some $b\ge 1$.
\end{enumerate}
\end{lemma}

\begin{proof}
If $\pi\colon X\to \p^2$ is a decomposable bundle, then $H=\Aut(\p^2)$ (see Lemma~\ref{lemm:surjectiveDecP2}). The same holds if $\pi\colon X\to \p^2$ is isomorphic to $\SS_1 \simeq \P(T_{\P^2})$ (see Corollary~\ref{Cor:SmallIsoSchwarzenberger} and Lemma~\ref{Lemm:SchwarzJumpLines}). If $\pi\colon X\to \p^2$ is isomorphic to $\SS_b$ for some $b\ge 2$, then $H=\Aut(\p^2,\Gamma)$ (see Lemma~\ref{Lemm:SchwarzJumpLines}\ref{AutoSchwarzb}). 

It then remains  to assume  $\Aut(\p^2,\Gamma)\subset H$ and deduce from it that $\pi\colon X\to \p^2$ is decomposable or isomorphic to $\SS_b\to \p^2$ for some $b\ge 1$. We use the double cover $\kappa\colon \p^1\times \p^1\to \p^2$ ramified over $\Gamma$ given in  Definition~\ref{def:Schwarz} and define a $\p^1\!$-bundle $\hat\pi\colon \hat{X}=X\times_{\p^2} (\p^1\times \p^1)\to \p^1\times \p^1$. By construction, the image of $\Autz(\hat{X})\to \Autz(\p^1\times \p^1)$ contains the group $\PGL_2$ embedded diagonally, namely the subgroup of $\Autz(\p^1\times \p^1)$ preserving the diagonal $\Delta$, the branch locus of $\kappa$. This implies that $\hat\pi\colon \hat{X}\to \p^1\times \p^1$ is either a decomposable $\p^1\!$-bundle or square isomorphic to $\hat{S}_b$ for some $b\ge 1$ (see Proposition~\ref{prop decompo and Umemura bundles}). In the latter case, the isomorphism of $\p^1\times \p^1$ that comes in the square has to preserve the diagonal and lifts to $\Autz(\hat{X})$ (see Lemma~\ref{Lemm:AutTTb}\ref{CommDiag}),  so $\hat\pi\colon \hat{X}\to \p^1\times \p^1$ is in fact isomorphic to $\hat{S}_b$. Denoting by $\sigma\in \Aut(\hat{X})$ the involution induced by the automorphism of $\p^1\times\p^1$ exchanging the two factors, we find $X=\hat{X}/\sigma$.

We first assume that $\hat{X} \to \PQ$ is a decomposable $\p^1\!$-bundle, say $\FF_{0}^{m_1,m_2}$. For $i=1,2$, the fibres of $\mathrm{pr}_i\circ \hat\pi\colon \FF_{0}^{m_1,m_2}\to \p^1$ are the Hirzebruch surfaces $\F_{\lvert m_i\rvert}$, where $\mathrm{pr}_i\colon \p^1\times \p^1\to \p^1$ is the $i$-th projection. As $\sigma$ exchanges the fibres, we get $m_1=\pm m_2$. We now prove that up to conjugation by an isomorphism of $\p^1\!$-bundles, we have one of the following situations for the action of $\sigma$ on $\hat X$:
\[\begin{array}{ccccc}
\sigma: &\FF_0^{m,-m} &\to &\FF_0^{m,-m} \\
         & [x_0:x_1;y_0:y_1;z_0:z_1] &\mapsto &[x_0:x_1; z_0:z_1; y_0: y_1],& m\ge 0;\\[1ex]
         \sigma: &\FF_0^{m,-m} &\to &\FF_0^{m,-m} \\
         & [x_0:x_1;y_0:y_1;z_0:z_1] &\mapsto &[x_0:-x_1 ; z_0:z_1; y_0: y_1],& m\ge 0; \text{ or }\\[1ex]
\sigma: &\FF_0^{m,m} &\to &\FF_0^{m,m} \\
         & [x_0:x_1;y_0:y_1;z_0:z_1] &\mapsto &[x_1:x_0 ; z_0:z_1; y_0: y_1], & m\ge 0.
\end{array}\]

If $m_1=m_2=0$, then $\FF_0^{0,0}$ is isomorphic to $(\p^1)^3$, so $\sigma$ is of the form $[x_0:x_1;y_0:y_1;z_0:z_1]\mapsto$ \mbox{$[\iota(x_0:x_1);z_0:z_1; y_0: y_1]$} for some element $\iota \in \PGL_2$, which is either the identity or of order $2$. Applying conjugation, we get one of the above cases (the last two cases being conjugate).

We then assume  $m_1=m_2=m>0$. The union of $-m$-curves in the fibres of $\mathrm{pr}_i\circ \hat\pi\colon \FF_{0}^{m,m}\to \p^1$ is then given by $x_0=0$ and $x_1=0$. This implies that $\sigma$ is of the form $[x_0:x_1;y_0:y_1;z_0:z_1]\mapsto$ $[x_1:\xi x_0;z_0:z_1; y_0: y_1]$ for some $\xi\in \k^*$. Conjugating by an automorphism  $[x_0:x_1;y_0:y_1;z_0:z_1]\mapsto [\mu x_0:x_1;y_0:y_1;z_0:z_1]$, we can assume $\xi=1$. This gives the third case.

The remaining cases are when $m_1=-m_2=m>0$. The union of $-m$-curves in the fibres of $\mathrm{pr}_i\circ \hat\pi$ is then given by $x_0=0$, for $i=1,2$. This implies that $\sigma$ is of the form $[x_0:x_1;y_0:y_1;z_0:z_1]\mapsto$ $[x_0:\mu x_1+x_0P(y_0,y_1,z_0,z_1);z_0:z_1; y_0: y_1]$, where $\mu\in \k^*$ and $P$ is of bidegree $m,m$ in $y_0,y_1,z_0,z_1$. As $\sigma$ is an involution, we get $\mu=\pm 1$ and $\mu P(y_0,y_1,z_0,z_1)+P(z_0,z_1,y_0,y_1)=0$. In both cases, conjugating by the automorphism $[x_0:x_1;y_0:y_1;z_0:z_1]\mapsto [x_0:x_1 \pm \frac{1}{2}x_0 P(y_0,y_1,z_0,z_1);y_0:y_1;z_0:z_1]$, we can assume that $P=0$ and obtain one of the first two cases.

We now study the three cases and show that only the first case can appear, in which case $X$ is a decomposable $\p^1\!$-bundle over $\p^2$.
In the first case, we obtain that $X=\FF_{0}^{m,-m}/\sigma=\P(\O_{\P^2}(-m) \oplus \O_{\P^2})$ is a decomposable $\p^1\!$-bundle over $\P^2$. In the second case, we consider $U:=[1:x;1:y;1:z] \simeq \A^3$, which is a $\sigma$-stable affine open subset of $\FF_{0}^{m, -m}$. The corresponding action on $\A^3$ is $(x,y,z)\mapsto (-x,z,y)$, which is conjugate to $(x,y,z)\mapsto (-x,-y,z)$. In the third case, we consider $V:=[x-1:x+1;1:y+z;1:-y+z] \simeq \A^3$, which is a $\sigma$-stable affine open subset of $\FF_{0}^{m, m}$. The action is again given by $(x,y,z)\mapsto (-x,-y,z)$. In both cases, the quotient is singular since the invariant algebra is $\k[x^2,xy,y^2,z]=\k[X_1,X_2,X_3,X_4]/(X_1X_3-(X_2)^2)$, which is impossible since $X$ is smooth.

It remains to study the case where $\hat{X}\to \p^1\times \p^1$ is isomorphic to $\TT_{b}\to \p^1\times \p^1$. By Lemma~\ref{Lemm:LiftSchwarzIsSchwarzHat},  $\hat{X}$ is equal to the pull-back of $\SS_{b}\to\p^2$, so there is an involution $\sigma'\in \Aut(\hat{X})$, acting on $\p^1\times \p^1$ as the exchange of the two factors $\tau\in \Aut(\p^1\times \p^1)$, such that $\hat{X}/\sigma'=\SS_b$. Since every automorphism of the $\p^1\!$-bundle $\TT_b\to \p^1\times \p^1$ is trivial (see Lemma~\ref{Lemm:AutTTb}\ref{CommDiag}), $\sigma=\sigma'$ is the unique lift of $\tau$, which shows that $\pi\colon X\to \p^2$ is isomorphic to $\SS_b\to \p^2$.
\end{proof}

\begin{remark} \label{Rem:kCold}
If $\k=\C$, then Lemma~\ref{Lem:LargeActionP2}\ref{SurjectiveP2} is a particular case of a classical result due to Van de Ven~\cite[Theorem]{VdV72}, and Lemma~\ref{Lem:LargeActionP2}\ref{noSurjectiveP2} follows from a result due to Vall\`es~\cite[Theorem]{Val00}.
\end{remark}

The following result is well known to specialists. Here we give a very short proof relying on recent results by Laurent.

\begin{lemma}\label{SubgroupsPGL3}
If $G\subsetneq \Aut(\p^2)=\PGL_3$ is a proper connected algebraic subgroup that does not fix any point of $\p^2$, then one of the following holds: 
\begin{enumerate}
\item \label{PGL2case}
$G=\Aut(\p^2,\Gamma)\simeq \PGL_2$, where $\Gamma\subset \p^2$ is a smooth conic; or
\item\label{SL2GL2case}
$G$ is a subgroup of $\Aut(\p^2,L)\simeq \GL_2\ltimes \k^2$, where $L\subset \p^2$ is a line, equal either to $\Aut(\p^2,L)$ or to $\SL_2\ltimes \k^2\subset \Aut(\p^2,L)$.
\end{enumerate}
Moreover, in both cases, $G$ acts on $\P^2$ with two orbits: the smooth conic $\Gamma$ in case~\ref{PGL2case} or the line $L$ in case~\ref{SL2GL2case}, and its open complement.
\end{lemma}

\begin{proof}
If $G$ is solvable, then it acts on $\P^2$ with a fixed point by the Borel fixed-point theorem (see~\cite[\S~21.2, Theorem]{Hum75}). If $G$ is semisimple, then by~\cite[Proposition~3.4.1]{Lau}, the only possibility is case~\ref{PGL2case}.  If $G$ is not solvable nor semisimple, then by~\cite[Proposition~3.4.4]{Lau}, the only possibility is case~\ref{SL2GL2case}.
The last sentence of the statement can be checked with a direct computation.
\end{proof}

\begin{proposition}\label{Prop:P1bundlesP2}
Let $\pi\colon X \to \P^2$ be a $\p^1\!$-bundle. If $\,X$ is not decomposable or isomorphic to a Schwarzenberger bundle $\SS_b\to \p^2$, for some $b\ge 1$, then there exists a $p\in \p^2$ such that the fibre $f=\pi^{-1}(p)$ is $\Autz(X)$-invariant. The blow-up $\eta\colon \hat X\to X$ of $f$ yields a commutative diagram with $\Autz(X)$-equivariant maps 
\[\xymatrix@R=4mm@C=2cm{
    \hat X \ar[r]^{\eta} \ar[d]_{\hat\pi}  & X \ar[d]^\pi \\
    \F_1 \ar[r]_{\tau} & \p^2\rlap{,}
}\]
  where $\hat{X}\to \F_1$ is a $\p^1\!$-bundle over $\F_1$ and $\tau$ is the blow-up of $p$.
\end{proposition}

\begin{proof}
Let $H \subset \PGL_3$ be the image of the natural group homomorphism $\Autz(X) \to \Aut(\P^2)$. The assumption implies that $H$ does not contain any group $\Aut(\p^2,\Gamma)$ for some smooth conic $\Gamma$ (see Lemma~\ref{Lem:LargeActionP2}).

If $H$ fixes a point $p\in \p^2$, we blow up simultaneously $p$ in $\P^2$ and $f=\pi^{-1}(p)$ in $X$. Computing explicitly these blow-ups in a trivializing local chart containing $p_0$, we see that $\hat{X} \to \F_1$ is again a $\p^1\!$-bundle. Moreover, since $p$ and $f$ are $\Autz(X)$-stable in $\P^2$ and $X$, respectively, the group $\Autz(X)$ acts on $\hat{X}$. 

It remains to show that the case where no point of $\p^2$ is fixed by $H$ is impossible. By Lemma~\ref{SubgroupsPGL3}, this case can only happen if $H=\Aut(\p^2,L)\simeq \GL_2\ltimes (\k^2,+)$ or $H \simeq \SL_2\ltimes (\k^2,+) \subsetneq \Aut(\p^2,L)$, where $L\subset \p^2$ is a line. We take a point $p \in \p^2\setminus L$ and denote by $G_0$ and $H_0$ the subgroups of $G=\Autz(X)$ and $H$ stabilizing $p$ (note that $H_0\simeq \GL_2$ or $H_0\simeq \SL_2$). Blowing up the point $p$ and its fibre yields a $\p^1\!$-bundle $\hat{X} \to \F_1$ equipped with a $G_0$-action. As the group $H_0$ acts transitively on $\p^2 \setminus (L\cup \{p\})\simeq \A^2\setminus \{0\}$, it acts transitively on the exceptional divisor of the blow-up in $\F_1$, and thus $H_0$ acts transitively on $\p^1$ via the structure morphism $\F_1 \to \p^1$. Therefore, by Proposition~\ref{prop decompo and Umemura bundles}, the $\p^1\!$-bundle $\hat{X} \to \F_1$ is a decomposable or an Umemura $\p^1\!$-bundle. In both cases, the natural group homomorphism $\hat{G}=\Autz(\hat{X}) \to \Aut(\F_1)$ is onto (see Lemmas~\ref{Lem:AutDecOnAutFa} and~\ref{Lem:SurjectiveActionUmemura}), and so $H_0$ acts transitively on the complement of the exceptional section of $\F_1$ (see Remark~\ref{Rem:AutFa}). Moreover, by Lemma~\ref{Lem:Extension}, the group $\hat{G}$ identifies with a subgroup of $G$; in particular, $G$ must act transitively on $\P^2 \setminus \{p\}$, contradicting the equality $H=\Aut(\p^2,L)$. 
\end{proof}

\section{The classification}\label{Sec:Classif}
In this section, we first reduce to four families of $\p^1\!$-bundles, in Proposition~\ref{Prop:FourCases}, which uses Propositions~\ref{prop decompo and Umemura bundles} and~\ref{Prop:P1bundlesP2} and is thus valid only in characteristic zero. We then study equivariant square birational maps between the four families, in \S\S~\ref{FirstLinks}--\ref{RigidityDecP2};
the content of these subsections is valid over any algebraically closed field of characteristic different from $2$. We then apply the results obtained in \S\S~\ref{SubSecRed}--\ref{RigidityDecP2} to get the main result in \S~\ref{LastStep}, valid only in characteristic zero.

\subsection{Reduction to the four families}\label{SubSecRed}
As a first step towards Theorem~\ref{Thm:MainA}, we prove the following. 

\begin{proposition}\label{Prop:FourCases}
Let $\pi\colon X\to S$ be a $\p^1\!$-bundle over a smooth projective rational surface $S$. 
There exists an $\Autz(X)$-equivariant square birational map $(X,\pi)\dasharrow (X',\pi')$ such that $(X',\pi')$ is isomorphic to one of the following:
\begin{center}\begin{tabular}{llllllll}
$(a)$&a decomposable &$\p^1\!$-bundle& $\FF_a^{b,c}$&\hspace{-0.3cm}$\longrightarrow$& \hspace{-0.2cm}$\F_a$& for some $a,b\ge 0$, $c\in \Z$;\\
$(b)$& a decomposable &$\p^1\!$-bundle &$\PP_b$&\hspace{-0.3cm}$\longrightarrow$& \hspace{-0.2cm}$\p^2$& for some $b\ge 0$;\\
$(c)$& an Umemura &$\p^1\!$-bundle &$\U_a^{b,c}$&\hspace{-0.3cm}$\longrightarrow$& \hspace{-0.2cm}$\F_a$& for some $a,b\ge 1, c\ge 2$;\\
    $(d)$ &a Schwarzenberger\! &$\p^1\!$-bundle &$\SS_b$&\hspace{-0.3cm}$\longrightarrow$& \hspace{-0.2cm}$\p^2$& for some $b\ge 1$.
  \end{tabular}
\end{center}
\end{proposition}

\begin{proof}
Using the descent lemma (Lemma~\ref{Lem:GoingdownMinimalsurfaces}), we can assume  $S=\p^2$ or  $S=\F_a$ for some $a\ge 0$.

In the case where $S=\F_a$, we apply Proposition~\ref{prop decompo and Umemura bundles} to reduce to the case of decomposable or Umemura bundles.

In the case where $S=\p^2$, we apply Proposition~\ref{Prop:P1bundlesP2} to reduce to the case of decomposable Schwarzenberger bundles, or to the case of $\F_1$, already studied.
\end{proof}

\subsection{First links}\label{FirstLinks}

\begin{remark}\label{Rem:Overview5}
  We now study equivariant square birational maps from a $\p^1\!$-bundle $X\to S$ to another \mbox{$\p^1\!$-bundle} $X'\to S'$, both in the above families.
  The action being equivariant, we get a birational map $\eta\colon S\dasharrow S'$ which conjugates the image $H\subset \Autz(S)$ of $\Autz(X)$ to a subgroup of $\Autz(S')$. Hence, $\eta$ is a sequence of blow-ups of points fixed by $H$ followed by a sequence of contractions of curves invariant by $H$. The nature of the pair $(H,S)$ given above implies that no point of $S$ is fixed and that the only $(-1)$-curve invariant by $H$ is the exceptional curve of $\F_1$. We then only need to consider this case and study the $\p^1\!$-bundle over $\p^2$ obtained (which is given by the descent lemma (Lemma~\ref{Lem:GoingdownMinimalsurfaces})) and square isomorphisms/birational maps of $\p^1\!$-bundles (doing nothing on $S$).
\end{remark}

We then observe that no decomposable $\p^1\!$-bundle over $\F_1$ has a maximal group of automorphisms.

\begin{lemma}\label{Lem:ExDecF1P2}
For each $b\ge 0$ and each $c\in \Z$, the rational map
\[\begin{array}{rccc}
\varphi\colon &\FF_1^{b,c} & \dasharrow & \PP_{b-c}\\
& [x_0:x_1;y_0:y_1;z_0:z_1] & \mapsto & [x_0y_0^{c}:x_1;y_0z_0:y_1:y_0z_1]
\end{array}\]
is a square birational map above a birational morphism $\eta\colon \F_1\to \p^2$ corresponding to the blow-up of $\,[0:1:0]$. Moreover, $\varphi\Autz(\FF_1^{b,c})\varphi^{-1}\subsetneq\Autz(\PP_{b-c})$.
\end{lemma}

\begin{proof}
The birational morphism $\eta\colon \F_1\to \p^2$, $[y_0:y_1;z_0:z_1]\mapsto [y_0z_0:y_1:y_0z_1]$ is the blow-up of $[0:1:0]$, and by the construction of $\varphi$, we find that this  is a square birational map above $\tau$. We then write $U=\p^2\setminus \{[0:1:0]\}$ and $\hat{U}=\tau^{-1}(U)\subset \F_1$ and observe that $\varphi$ induces an isomorphism $\hat\pi^{-1}(\hat{U})\iso \pi^{-1}(U)$. By Lemma~\ref{Lem:GoingdownMinimalsurfaces}, $\varphi$ is the unique square birational map above $\tau$ having this property (up to composition by an isomorphism of $\p^1\!$-bundles at the target) and is $\Autz(\FF_1^{b,c})$-equivariant, which yields $\varphi\Autz(\FF_1^{b,c})\varphi^{-1}\subset\Autz(\PP_{b-c})$. We moreover have $\varphi\Autz(\FF_1^{b,c})\varphi^{-1}\subsetneq\Autz(\PP_{b-c})$ since $\Autz(\PP_{b-c})$ acts transitively on $\p^2$ (see Lemma~\ref{lemm:surjectiveDecP2}), but every element of $\varphi\Autz(\FF_1^{b,c})\varphi^{-1}$ acts on $\p^2$ by fixing $[0:1:0]$.
\end{proof}

\subsection{Reduction of birational maps to elementary links}
We show here that every birational map of $\p^1\!$-bundles between the four types in Proposition~\ref{Prop:FourCases} is a sequence of elementary links.

\begin{lemma}\label{Lemm:DecElLinks}
Let $G$ be a connected algebraic group acting on a $\p^1\!$-bundle $X\to S$, let $H\subset \Autz(S)$ be the image of $G$ under this action, and assume either that  no curve of $S$ is invariant by $H$ or that $(H,S)$ is one of the following pairs: $(\Autz(\F_a),\F_a)$ with $a\ge 1$, $(\Autz(\p^1\times \p^1,\Delta),\p^1\times \p^1)$,  or $(\Aut(\p^2,\Gamma),\p^2)$, where $\Gamma=\{ [X:Y:Z] \mid Y^2=4XZ\}\subset \p^2$ and $\Delta\subset \p^1\times \p^1$ is the diagonal.

If $\pi'\colon X'\to S$ is a $\p^1\!$-bundle and $\varphi\colon X\dasharrow X'$ is a $G$-equivariant birational map of $\p^1\!$-bundles $($as in Definition~$\ref{Defi:SquareEtc})$ which is not an isomorphism, then we have a sequence of $\p^1\!$-bundles $\pi_i\colon X_i\to S$, $i=0,\ldots,n$, with $\pi_0=\pi$ and $\pi_n=\pi'$, and we can write $\varphi=\varphi_n\circ \cdots \circ \varphi_1$, where $\varphi_i\colon X_i\dasharrow X_{i+1}$ is the blow-up of an irreducible curve $\ell_i\subset X_i$, followed by the contraction of the strict transform of ${\pi_i}^{-1}(\pi_i(\ell))$, and where $\pi_i|_{\ell_i}\colon \ell_i\to \pi_i(\ell_i)$ yields an isomorphism between $\ell_i$ and either $s_{-a}\subset \F_a$ with $a\ge 1$, $\Delta\subset \p^1\times \p^1$, or  $\Gamma\subset \p^2.$

Moreover, taking $n$ minimal, the sequence $\varphi_1,\dots,\varphi_n$ is unique up to isomorphisms of $\,\p^1$-bundles.
\end{lemma}

\begin{proof}
Taking an open subset $U\subset S$ isomorphic to $\A^2$, the $\p^1\!$-bundle $X$ is trivial, so corresponds to $\p^1\times\A^2$. On this chart, the birational map $\varphi$ is of the form
\[\big([x_0:x_1],(u,v))\mapsto ([a(u,v)x_0+b(u,v)x_1:c(u,v)x_0+d(u,v)x_1],(u,v)\big)\]
for some $a,b,c,d\in \k(u,v)$ with $ad-bc\not=0$. Choose  $a,b,c,d\in \k[u,v]$ with no common factor; the polynomials $a,b,c,d$ are unique, up to multiplication by an element of $\k^*$. The zero locus of the determinant $P=ad-bc$ thus corresponds  exactly to the subset of $U$ over which $\varphi$ is not an isomorphism.

Denoting by $K\subset S$ the subset over which $\varphi$ is not an isomorphism, we find that $K$ is a union of closed irreducible curves. The map $\varphi$ being $G$-equivariant, $K$ is invariant by $H$. The assumption made on $(H,S)$ implies that either $K=\emptyset$, in which case $\varphi$ is an isomorphism, or $(K,S)$ is one of the three cases
$(s_{-a},\F_a)$, $(\Delta, \p^1\times \p^1)$, or  $(\Gamma, \p^2).$

In the three cases, we can choose, for each point $p\in K$, an open set $U\subset S$ and an isomorphism $U\iso \A^2$ which sends $K\cap \A^2$ onto the line $u=0$. Writing $\varphi$ with $a,b,c,d\in \k[u,v]$ as above, we find that $ad-bc=\lambda u^n$ for some integer $n\ge 1$ and $\lambda\in \k^*$. As $u$ does not divide all polynomials $a,b,c,d$, the matrix $M(u,v)=\begin{bsmallmatrix} a& b \\ c & d \end{bsmallmatrix}$ is such that $M_0=M(0,v)$ has rank $1$. The ring $\k[v]$ being a principal ideal domain,
we can use the Smith normal form and find $A,B\in \GL_2(\k[v])$ (that are in fact products of elementary matrices since $\k[v]$ is Euclidean) such that $AM_0B=\begin{bsmallmatrix} e& 0 \\ 0 & 0 \end{bsmallmatrix}$ for some $e\in \k[v]\setminus \{0\}$. Replacing $M$ with $AMB$, we then obtain $b=ub'$ and $d=ud'$ for some $b',d'\in \k[u,v]$ and get $M=M'R$ with $M'=\begin{bsmallmatrix} a& b' \\ c & d' \end{bsmallmatrix}$ and $R=\begin{bsmallmatrix} 1 &0 \\ 0 & u \end{bsmallmatrix}$. The base locus of $\varphi$ is then given, in these coordinates, by $u=0$ and $x_0=0$, corresponding to a curve $\ell\subset X$ such that $\pi$ yields an isomorphism $\pi|_{\ell}\colon \ell \to K$. The blow-up of this curve followed by the contraction of the strict transform of the surface $\pi^{-1}(K)$ is locally given by the matrix $R$. As $\det(M')=\lambda u^{n-1}$, we proceed by induction and obtain the result.
\end{proof}

\subsection{Links between decomposable bundles over Hirzebruch surfaces}

\begin{lemma}\label{Lem:InvariantDec}
Let $a,b,c\in \Z$ with $a,b\ge0$. The curves of $\FF_a^{b,c}$ invariant by $\Autz(\FF_a^{b,c})$ are given as follows: 
\begin{enumerate}
\item\label{Invl0}
The curve $l_{00}$ given by $x_0=y_0=0$ is invariant if and only if $ab>0$ or $ac<0$.
\item\label{Invl1}
The curve $l_{10}$ given by $x_1=y_0=0$ is invariant if and only if $ac>0$.
\item\label{Onlyl0l1}
No irreducible curve $\ell \subset \FF_a^{b,c}$ with $\ell\not=l_{00}$ and $\ell\not=l_{10}$ is invariant.
\end{enumerate}
\end{lemma}

\begin{proof}
By Lemma~\ref{Lem:AutDecOnAutFa}, the morphism $\pi\colon \FF_a^{b,c}\to \F_a$ yields a surjective group homomorphism $\Autz(\FF_a^{b,c})\twoheadrightarrow\Autz(\F_a)$. Hence, if $\ell\subset\FF_a^{b,c}$ is a curve invariant by $\Autz(\FF_a^{b,c})$, then $\pi(\ell)=s_{-a}$ and $a>0$ (see Remarks~\ref{Rem:AutF0} and~\ref{Rem:AutFa}). Moreover, $\G_m$ acts on $\FF_a^{b,c}$ via $[x_0:x_1;y_0:y_1;z_0,z_1]\mapsto [x_0:tx_1;y_0:y_1;z_0:z_1]$, so any point of $\ell$ should satisfy $x_0=0$ or $x_1=0$ (otherwise we get a whole fibre of a point of $s_{-a}$ which is contained in $\ell$); hence $\ell$ has to be equal to $l_{00}$ or $l_{10}$. This yields~\ref{Onlyl0l1}.

We can now assume  $a>0$ and show~\ref{Invl0} and~\ref{Invl1} by proving when $l_{00}$ and $l_{10}$ are invariant. The surface $\pi^{-1}(s_{-a})\simeq \F_c$ being invariant, the curve $l_{00}$ is invariant when $c<0$ and $l_{10}$ is invariant when $c>0$. Moreover, the fibres of the $\F_b$-bundle $\FF_a^{b,c}\to \p^1$ (given in Remark~\ref{Rem:Fbbundle}) are exchanged by $\Autz(\FF_a^{b,c})$. If $b>0$, the surface $S_{-b}$ given by $x_0=0$ is the union of the negative sections and is then invariant, so $l_{00}=S_{-b}\cap \pi^{-1}(s_{-a})$ is invariant.

It remains to show that $l_{10}$ and $l_{00}$ are not invariant in the other cases. If $c\le 0$, the group $\G_a$ acts on $\FF_a^{b,c}$ via $(t,[x_0:x_1;y_0:y_1;z_0:z_1])\mapsto ([x_0:x_1+tx_0y_1^bz_1^{-c};y_0:y_1;z_0:z_1])$, 
so $l_{10}$ is not invariant. If $b=0$ and $c\ge 0$, then $\G_a$ acts on $\FF_a^{b,c}$ via $(t,[x_0:x_1;y_0:y_1;z_0:z_1])\mapsto ([x_0+t x_1z_0^c:x_1;y_0:y_1;z_0:z_1])$,
so $l_{00}$ is not invariant.
\end{proof}

\begin{lemma}\label{Lem:InvariantLinksDec}\item
\begin{enumerate}
\item\label{LinkDec}For all $a,b,c\in \Z$ with $a,b\ge 0$, the blow-up of the curve $l_{00}\subset \FF_a^{b,c}$ given by $x_0=y_0=0$ followed by the contraction of the strict transform of the surface $\pi^{-1}(\s{-a})$ onto $l_{10}\subset \FF_a^{b+1,c+a}$ given by $x_1=y_0=0$ yields a birational map
\[\begin{array}{rccc}
\varphi\colon &\FF_a^{b,c}&\dasharrow &\FF_a^{b+1,c+a}\\
&\big([x_0:x_1;y_0:y_1;z_0:z_1]\big) &\mapsto & \big([x_0:x_1y_0;y_0:y_1;z_0:z_1]\big).\end{array}\]
 We then have  $\varphi\Autz(\FF_a^{b,c})\varphi^{-1}\subset \Autz(\FF_a^{b+1,c+a})$ if and only if $ab>0$ or $ac<0$, and we have $\varphi^{-1}\Autz(\FF_a^{b+1,c+a})\varphi\subset \Autz(\FF_a^{b,c})$ if and only if $a(c+a)>0$.
\item\label{OnlyLinksDec}
For all $a,b,c\in \Z$ with $a,b\ge0$, every $\Autz(\FF_a^{b,c})$-equivariant birational map of $\,\p^1\!$-bundles $\FF_a^{b,c}\dasharrow X$ is a composition of birational maps as in~\ref{LinkDec} $($and of their inverses$)$ and of isomorphisms of $\,\p^1\!$-bundles.
\end{enumerate}
\end{lemma}

\begin{proof}
As we can check in local coordinates, the birational map $\F_b\dasharrow \F_{b+1}$ given by $[x_0:x_1;y_0:y_1] \mapsto$ $[x_0:x_1y_0;y_0:y_1]$
is the composition of the blow-up of the point $[0:1;0:1]$, followed by the contraction of the strict transform of $y_0=0$ onto the point $[1:0;0:1]$. Doing this within a family yields $\varphi$. We then have  $\varphi\Autz(\FF_a^{b,c})\varphi^{-1}\subset \Autz(\FF_a^{b+1,c+a})$ if and only if  $\Autz(\FF_a^{b,c})$ preserves $l_{00}$ and $\varphi^{-1}\Autz(\FF_a^{b+1,c+a})\varphi\subset \Autz(\FF_a^{b,c})$ if and only if $\Autz(\FF_a^{b+1,c+a})$ preserves $l_{10}$. Hence, \ref{LinkDec} follows from Lemma~\ref{Lem:InvariantDec} for $b\ge 0$. 

It remains to prove~\ref{OnlyLinksDec}. By Lemma~\ref{Lemm:DecElLinks}, we only need to consider elementary links $\FF_a^{b,c}\dasharrow X$, obtained by blowing up an invariant curve $\ell\subset \FF_a^{b,c}$, followed by contracting $\pi^{-1}(\pi(\ell))$. The only curves in $\FF_a^{b,c}$ that are invariant are $l_{00}$ and $l_{10}$ (see Lemma~\ref{Lem:InvariantDec}), and the links associated to these are given in~\ref{LinkDec}: if we start with $l_{00}$, then it is equal to $\varphi$ as in~\ref{LinkDec}; if we start with $l_{10}$, then it is equal to $\varphi^{-1}$ as in~\ref{LinkDec} when $b\ge 1$ and is the composition of the isomorphism $\FF_a^{0,c}\simeq \FF_a^{0,-c}$ exchanging $x_0$ and $x_1$ with a link $\varphi$ as in~\ref{LinkDec} if $b=0$.
\end{proof}

We recall that the notions of stiff and superstiff $\p^1\!$-bundle, used in the next result and later, were defined in the introduction (Definition~\ref{def:max}).

\begin{corollary}\label{Cor:MaxDecBundlesFa}
Let $a,b\ge 0$ and $c\in \Z$ be such that $c\le 0$ when $b=0$. Then, $\Autz(\FF_a^{b,c})$ is maximal if and only if $a\not=1$ and one of the following holds:
\begin{enumerate}
\item \label{FFazeroSuper}
$a=0$; \textit{i.e.}~$\FF_a^{b,c}$ is a decomposable $\p^1\!$-bundle over $\F_0=\p^1\times \p^1$;
\item \label{FFbczeroSuper}
$b=c=0$; \textit{i.e.}~$\FF_a^{b,c}$ is isomorphic to $\FF_a^{0,0}\simeq \F_a\times \p^1$;
\item \label{FFabcOnlyMax}
$-a<c<ab$.
\end{enumerate}
Moreover, $\FF_a^{b,c}$ is superstiff in cases~\ref{FFazeroSuper} and~\ref{FFbczeroSuper} and is not stiff in case~\ref{FFabcOnlyMax}. 

More precisely, denoting by $r$ the smallest integer such that $c-ra\le 0$, we find $r\le b$ and get an infinite sequence of elementary links
\[\FF_a^{b-r,c-ra}\dasharrow \dots \dasharrow \FF_a^{b,c}\dasharrow \FF_a^{b+1,c+a}\dasharrow \dots \dasharrow\FF_a^{b+n,c+an}\dasharrow \dots\]
which conjugates $\Autz( \FF_a^{b,c})$ to  $\Autz(\FF_a^{b+s,c+as})$ for each integer $s\ge -r$. This gives all $\Autz(\FF_a^{b,c})$-equivariant square birational maps from $\FF_a^{b,c}$ to another $\p^1\!$-bundle. 
\end{corollary}

\begin{proof}
If $a=1$, then $\Autz(\FF_a^{b,c})$ is not maximal, by Lemma~\ref{Lem:ExDecF1P2}. We can thus assume  $a\not=1$. In this case, as we have a surjective group homomorphism $\Autz(\FF_a^{b,c})\twoheadrightarrow \Autz(\F_a)$, every $\Autz(\FF_a^{b,c})$-equivariant square birational map starting from $\FF_a^{b,c}$ is in fact the composition of an element of $\Autz(\FF_a^{b,c})$ with an $\Autz(\FF_a^{b,c})$-equivariant birational map of $\p^1\!$-bundles. These are compositions of links given in Lemma~\ref{Lem:InvariantLinksDec}\ref{LinkDec} and isomorphisms of $\p^1\!$-bundles, as explained in Lemma~\ref{Lem:InvariantLinksDec}\ref{OnlyLinksDec}.

If $a=0$ or $b=c=0$, then every $\Autz(\FF_a^{b,c})$-equivariant birational map of $\p^1\!$-bundles starting from $\FF_a^{b,c}$ is in fact an isomorphism of $\p^1\!$-bundles (see Lemma~\ref{Lem:InvariantLinksDec}). This shows that $\FF_a^{b,c}$ is superstiff and thus that $\Autz(\FF_a^{b,c})$ is maximal in these cases, corresponding to~\ref{FFazeroSuper} and~\ref{FFbczeroSuper}.

We then assume $a\ge 2$ and suppose  $(b,c)\not=(0,0)$. 

Let us show that $\Autz(\FF_a^{b,c})$ is not maximal when $c\ge ab$. Note that $b>0$ in this case, by assumption (we suppose $c<0$ when $b=0$). Lemma~\ref{Lem:InvariantLinksDec} yields an $\Autz(\FF_a^{b,c})$-equivariant birational map $\psi\colon\FF_a^{b,c}\dasharrow \FF_a^{0,c-ab}$ given as the composition of the links $\FF_a^{b,c}\dasharrow \FF_a^{b-1,c-a}{\dasharrow} \dots {\dasharrow}\FF_a^{1,c-a(b-1)}{\dasharrow} \FF_a^{0,c-ab}$. By construction, $\psi \Autz(\FF_a^{b,c})\psi^{-1}\subset \Autz(\FF_a^{0,c-ab})$ preserves the curve $l_{00}$, which then implies  that $\psi \Autz(\FF_a^{b,c})\psi^{-1}\subsetneq \Autz(\FF_a^{0,c-ab})$ by Lemma~\ref{Lem:InvariantDec}\ref{Invl0}.

We now prove that $\Autz(\FF_a^{b,c})$ is not maximal when $c\le -a$. In this case, we use the $\Autz(\FF_a^{b,c})$-equivariant link $\varphi\colon \FF_a^{b,c}\dasharrow \FF_a^{b+1,c+a}$ given by 
Lemma~\ref{Lem:InvariantLinksDec}, which is made such that $\varphi \Autz(\FF_a^{b,c})\varphi^{-1}\subset \Autz(\FF_a^{b+1,c+a})$ preserves the curve $l_{10}$ and thus yields  $\varphi \Autz(\FF_a^{b,c})\varphi^{-1}\subsetneq \Autz(\FF_a^{b+1,c+a})$ by Lemma~\ref{Lem:InvariantDec}\ref{Invl0}.

We can now assume  $-a<c<ab$. Denoting by $r$ smallest integer such that $c-ra\le 0$, we find $r\le b$ and get an infinite sequence of elementary links
\[\FF_a^{b-r,c-ra}\dasharrow \cdots \dasharrow \FF_a^{b,c}\dasharrow \FF_a^{b+1,c+a}\dasharrow \dots \dasharrow\FF_a^{b+n,c+an}\dasharrow \dots\]
which conjugates $\Autz( \FF_a^{b,c})$ to  $\Autz(\FF_a^{b+s,c+as})$ for each integer $s\ge -r$ (see Lemma~\ref{Lem:InvariantLinksDec}). Every $\Autz(\FF_a^{b,c})$-equivariant birational map of $\p^1\!$-bundles $\FF_a^{b,c}\dasharrow X$ is a composition of birational maps such as these and of isomorphisms of $\p^1\!$-bundles since $l_{10}$ is not invariant by $\Autz(\FF_a^{b-r,c-ra})$ (see Lemma~$\ref{Lem:InvariantDec}\ref{Invl1})$ as $c-ra \le 0$ and $b-r\ge 0$.
\end{proof}

\subsection{Links between Umemura bundles over Hirzebruch surfaces}
We can similarly treat the case of Umemura $\p^1\!$-bundles. Recall that such bundles $\U_{a}^{b,c}\to \F_a$ are defined by positive integers  $a,b\ge 1$ and $c\ge 2$ such that $c=ak+2$ with $0\le k\le b$ (see Definition~\ref{def Umemura bundle} and Remark~\ref{Rem:UmeExplicit} for more details).

We start with the case where $a=1$ and study the $\p^1\!$-bundle $\V_b\to \p^2$ obtained from $\U_1^{b,2}\to \F_1$ as follows. 

\begin{lemma}\label{Lem:UMeF1P2}
Let $\eta\colon \F_1\to \p^2$, $[y_0:y_1;z_0:z_1]\mapsto [y_0z_0:y_1:y_0z_1]$ be the blow-up of $\,[0:1:0]$, which induces an isomorphism between $\hat U=\eta^{-1}(U)$ and $U=\p^2\setminus \{[0:1:0]\}$. 
\begin{enumerate}\item\label{ExistUniqueV1}
For each integer $b\ge 1$, there exist a $\p^1\!$-bundle $\pi\colon\V_b\to \p^2$, unique up to isomorphism of $\,\p^1$-bundles, and a birational morphism $\psi\colon \U_1^{b,2}\to \V_b$ such that the following hold:
\begin{enumerate}
\item\label{psisq}
$\psi$ is a square birational map over $\eta$;
\item\label{psiisoU}
$\psi$ induces an isomorphism $\hat\pi^{-1}(\hat{U})\iso \pi^{-1}(U)$;
\item\label{psiBlowUp}
  $\psi$ is the blow-up of the smooth rational curve $\pi^{-1}([0:1:0])\subset \V_b$.
\end{enumerate}
\item\label{JumpingV1}
If $L\subset \p^2$ is a line, the $\p^1\!$-bundle $\pi^{-1}(L)\stackrel{\pi}{\to} L\simeq \p^1$ is isomorphic to 
\[\begin{array}{lll}
\F_b\to \p^1 & \text{ if $L$ is a line through $[0:1:0]$;}\\
\F_{\lvert b-2\rvert}\to \p^1 & \text{ if $L$ is a line not passing through  $[0:1:0]$}.\end{array}
\]
\item\label{AutV1inP2}
The image in $\Aut(\p^2)$ of $\Autz(\V_b)$ is equal to $\Aut(\p^2)$ if $b=1$ and to $\Aut(\p^2,[0:1:0])$ if $b\ge 2$.
\item\label{AutoV1lift}
We have $\psi\Autz(\U_1^{b,2})\psi^{-1}\subset \Autz(\V_b)$, with equality if $b\ge 2$.
\end{enumerate}
\end{lemma}

\begin{proof}
The existence of a unique birational map $\psi\colon \U_1^{b,2}\dasharrow \V_b$ satisfying~\ref{ExistUniqueV1}\ref{psisq}-\ref{psiisoU} follows from the descent lemma (Lemma~\ref{Lem:GoingdownMinimalsurfaces}). We now prove that $\psi$ also satisfies~\ref{psiBlowUp}, which implies that $\psi$ is a birational morphism. To do this, we denote by $W\subset \F_1$ the open subset given by $y_1\not=0$ and show that $\hat\pi \colon \U_1^{b,2}\to \F_1$ is trivial over $W$. As $W$ contains the exceptional curve $s_{-1}\subset \F_1$ of $\eta$ (given by $y_0=0$), this will show that one can contract $\hat\pi^{-1}(s_{-1})\simeq \p^1\times \p^1$ and obtain a birational morphism having the desired properties. To show the triviality of $\hat\pi$ over $W$, we take the transition function of $\U_1^{b,2}$, given by $\nu\in \Aut(\F_b \times \A^1 \setminus \{0\})$ as follows
\[\begin{array}{ccl}\nu\colon\big([x_0:x_1;y_0:y_1],z\big) &\mapsto &\left([x_0:x_1z^{2}+x_0 y_1^{b}z;y_0z:y_1],\frac{1}{z}\right).\end{array}\]
The intersection of $W$ with each chart is isomorphic to $\p^1\times \A^2$, via the inclusion \mbox{$\p^1\times \A^2\hookrightarrow \F_b\times \A^1$} given by  $([x_0:x_1],y,z)\mapsto [x_0:x_1;y:1],z)$, and the transition function becomes 
$([x_0:x_1],y,z) \mapsto\linebreak ([x_0:x_1z^{2}+x_0 z],yz,z^{-1})$, which yields a trivial $\p^1\!$-bundle since $\begin{bsmallmatrix}0 &1 \\ -1 &z^{-1}\end{bsmallmatrix}\cdot\begin{bsmallmatrix}1 &0 \\ z &z^2\end{bsmallmatrix}\cdot\begin{bsmallmatrix}1 &-z \\ 0 &1\end{bsmallmatrix}=\begin{bsmallmatrix}z &0 \\ 0 &z\end{bsmallmatrix}$. This achieves the proof of~\ref{ExistUniqueV1}.

\ref{JumpingV1}-\ref{AutV1inP2}-\ref{AutoV1lift} The existence and
unicity of $\psi$ and $\V_b$ being proven, we then observe that $\psi\Autz(\U_1^{b,2})\psi^{-1}\subset \Autz(\V_b)$ also follows from the descent lemma (see Lemma~\ref{Lem:GoingdownMinimalsurfaces}). This shows in particular that the subgroup $H_b\subset \Aut(\p^2)$, being the image of $\Autz(\V_b)$, contains the group $\Aut(\p^2,[0:1:0])\simeq \GL_2\rtimes \k^2$.

We now take the open subsets $U_0,U_1\subset \p^2$ given by $U_0= \{[X:Y:Z] | X \neq 0\} \simeq \A^2$ and $U_1= $ $\{[X:Y:Z] | Z \neq 0\} \simeq \A^2$ and observe that $\hat{U}_i=\eta^{-1}(U_i)\simeq U_i$ for $i=0,1$. Hence, $\pi^{-1}(U_i)\simeq \hat\pi^{-1}(\hat{U}_i)$, and the transition function is computed as follows: a point $[1:u:v]\in U_0\cap U_1$ corresponds to the point $[1:u;1:v]\in \F_1$, and thus its preimage to  $([x_0:x_1;1:u],v)\in \F_b\times \A^1$ on the first chart is sent onto $([x_0:x_1v^2+x_0u^bv;v:u],v^{-1})=([x_0:x_1v^{2-b}+x_0u^bv^{1-b};1:v^{-1}],v^{-1})$ on the second chart. The transition function is then given by
$$\begin{array}{ccc}
\p^1 \times U_0 &\dasharrow &\p^1 \times U_1 \\
\big([x_0:x_1],[1:u:v]\big)&\mapsto & \big([x_0:x_1v^{2-b}+x_0u^bv^{1-b}],[\frac{1}{v}:\frac{u}{v}:1]\big).\end{array}$$
 For $b=1$, we find the transition function of $\SS_1 \simeq \P(T_{\P^2})$ (see Corollary~\ref{Cor:SmallIsoSchwarzenberger}), which yields $\V_b\simeq \SS_1$ and thus $H_1=\Aut(\p^2)$ (see Remark~\ref{PGL3:TP2} or Lemma~\ref{Lemm:SchwarzJumpLines}\ref{AutoSchwarzb}). In particular, $\psi\Autz(\U_1^{b,2})\psi^{-1}\subsetneq \Autz(\V_b)$ in this case. Assertions~\ref{AutV1inP2} and~\ref{AutoV1lift} are then proven for $b=1$.
 
We now prove~\ref{JumpingV1} (for each $b\ge 1$). We observe that if $L$ passes through $[0:1:0]$, its strict transform on $\F_1$ is a fibre $f$ of the $\p^1\!$-bundle $\F_1\to \p^1$, so $\hat\pi^{-1}(f)\simeq \F_b$. Since $\pi^{-1}(L)\simeq \hat\pi^{-1}(f)$ (because $\psi$ is a blow-up of a curve), we obtain $\pi^{-1}(L)\simeq \F_b$. We now take  a line $L$ not passing through $[0:1:0]$, use the action of $\Autz(\V_b)$ to restrict to the case where $L$ is the line given by $Y=0$, and set $u=0$ in the transition function above to get $\pi^{-1}(L)\simeq \F_{\lvert b-2\rvert}$.

Assertion~\ref{AutV1inP2} for $b\ge 2$ is now given as follows: $[0:1:0]$ has to be fixed by $H_b$ because of~\ref{JumpingV1}, and $\Aut(\p^2,[0:1:0])\subset H_b$ was already proven.

It remains to show~\ref{AutoV1lift}, and then to prove that every element  $g\in\Autz(\V_b)$ belongs to $\psi^{-1}g\psi\in\Autz(\U_1^{b,2})$ when $b\ge 2$. Assertion~\ref{AutV1inP2} implies that $g$ preserves the curve $\pi^{-1}([0:1:0])$, so this follows from~\ref{ExistUniqueV1}\ref{psiBlowUp}.\end{proof}

\begin{lemma}\label{Lem:InvariantUme}
Let $a,b\ge 1$ and $c\ge 2$ be such that $c=ak+2$ with $0\le k\le b$. The curves of $\,\U_a^{b,c}$ invariant by $\Autz(\U_a^{b,c})$ are given as follows: 
\begin{enumerate}
\item\label{Invl0Um}
The curve $l_{00}$ given by $x_0=y_0=0$ on both charts  is invariant.
\item\label{Invl1Um}
The curve $l_{10}$ given by $x_1=y_0=0$ on both charts is invariant if and only if $k>0$ $($\textit{i.e.}~when $c>2)$.
\item\label{OnlyInvUm}
These are the two curves $($respectively, the only curve$)$ of $\,\U_a^{b,c}$ invariant by $\Autz(\U_a^{b,c})$ if $c>2$ $($respectively, when $c=2)$.
\end{enumerate}
\end{lemma}

\begin{proof}
Since $a\ge 1$, the curve $s_{-a}\subset \F_a$ is invariant by $\Autz(\F_a)$; hence the surface $\pi^{-1}(s_{-a})$ is given by $y_0=0$ on both charts. The fibres of the $\F_b$-bundle $\U_a^{b,c}\to \p^1$ are exchanged by $\Autz(\U_a^{b,c})$. Since $b\ge 1$, the surface $S_{-b}$ given by $x_0=0$ is the union of the negative sections and is then invariant. This yields~\ref{Invl0Um}.

Recall that the transition function of $\U_a^{b,c}$ is given by
\[\begin{array}{ccl}\nu\colon\big([x_0:x_1;y_0:y_1],z\big) &\mapsto &\left([x_0:x_1z^{c}+x_0 y_0^ky_1^{b-k}z^{c-1};y_0z^a:y_1],\frac{1}{z}\right)\end{array}\]
(see Remark~\ref{Rem:UmeExplicit}).
If $k>0$, the surface $\pi^{-1}(s_{-a})$, corresponding to $y_0=0$, is  isomorphic to $\F_c$, and the curve $l_{10}$ given by $x_1=y_0=0$ on both charts corresponds to the curve $s_{-c}\subset \F_c$ (with $c>0$) and is thus invariant.

It remains to see that $l_{10}$ is not invariant if $k=0$ and that no curve distinct from $l_{10}$ or $l_{00}$ can be invariant.

\ref{OnlyInvUm} Let $\ell\subset \U_a^{b,c}$ be an invariant curve. As the morphism $\pi\colon \U_a^{b,c}\to \F_a$ yields a surjective  group homomorphism $\Autz(\U_{a}^{b,c})\twoheadrightarrow\Autz(\F_a)$ (see Lemma~\ref{Lem:SurjectiveActionUmemura}), and because $a\ge 1$, we have $\pi(\ell)=s_{-a}$. We then use the fact that $\ell$ has to be invariant by the $\GL_2$-action given explicitly in Remark~\ref{Ume:ActionGL2}. We consider the action of the upper-triangular group by taking $\gamma=0$ and obtain that the image of $([x_0:x_1;0:1],0)$, on the first chart, is equal to $\big([x_0:\frac{x_1\alpha^{c-1}}{\delta};0:1],0\big)$ if $k>0$ and to $\big([x_0:\frac{x_1\alpha^{c-1}}{\delta}-\frac{x_0\alpha^{c-2}\beta}{\delta};0:1],0\big)$ if $k=0$. We then find that either $([0:1;0:1],0)\in \ell$  or $([1:0;0:1],0)\in \ell$ and $k>0$. In the first case, we get $\ell=l_{00}$ since $l_{00}$ is an orbit. In the second case, $l_{10}$ is an orbit and $\ell=l_{10}$. This achieves the proof.
\end{proof}

\begin{lemma}\label{Lem:InvariantLinksUme}\quad
  
\begin{enumerate}
\item\label{LinkUme}
For each Umemura $\p^1\!$-bundle $\U_a^{b,c}\to \F_a$, the blow-up of the curve $l_{00}\subset \U_a^{b,c}$ followed by the contraction of the strict transform of the surface $\pi^{-1}(\s{-a})$ onto the curve $\ell_{10}$ yields a birational map
$\varphi\colon \U_a^{b,c}\dasharrow \U_a^{b+1,c+a}$ satisfying $\varphi \Autz(\U_a^{b,c})\varphi^{-1}=\Autz(\U_a^{b+1,c+a})$.
\item\label{SpecialUmeDec}
For each $a\ge 1$, we have a birational map $\varphi\colon\U_a^{1,a+2}\dasharrow \FF_a^{0,2}$ such that $\varphi(\Autz(\U_a^{1,a+2}))\varphi^{-1}\subsetneq\Autz(\FF_a^{0,2})$.
\end{enumerate}
\end{lemma}

\begin{proof}\ref{LinkUme}
For each $a\ge 1, b\ge 0$, and $0\le k\le b$, we write $c=ak+2$,  denote by $\nu_a^{b,c}\in \Aut(\F_b\times \A^1\setminus \{0\})$ the birational map
\[\begin{array}{rccc}
&\F_b\times \A^1 & \dasharrow & \F_{b}\times \A^1\\
\nu_a^{b,c}\colon&\big([x_0:x_1;y_0:y_1],z\big) &\mapsto &\left([x_0:x_1z^{c}+x_0 y_0^ky_1^{b-k}z^{c-1};y_0z^a:y_1],\frac{1}{z}\right),\end{array}\]
and observe that $\nu_a^{b,c}$ is the transition function of $\U_a^{b,c}$ if $b\ge 1$ (see Remark~\ref{Rem:UmeExplicit}).  We then denote by $\varphi_b$ the birational map 
\[\begin{array}{rccc}
&\F_b\times \A^1 & \dasharrow & \F_{b+1}\times \A^1\\
\varphi_b\colon & ([x_0:x_1;y_0:y_1],z) & \mapsto & ([x_0:x_1y_0;y_0:y_1],z)\end{array}\]
and observe that $\varphi_b\nu_a^{b,c}\varphi_b^{-1}=\nu_a^{b+1,c+a}$. If $b\ge 1$, the blow-up of $l_{00}\subset \U_a^{b,c}$, followed by the contraction of the strict transform of the surface $\pi^{-1}(s_{-a})$, yields a birational map given in the two charts by $\varphi_b$. This then corresponds  to a birational map $\U_a^{b,c}\dasharrow \U_a^{b+1,c+a}$, which is the blow-up of the curve $l_{00}\subset \U_a^{b,c}$ followed by the contraction of the strict transform of the surface $\pi^{-1}(\s{-a})$. Since $l_{00}$ and $\pi^{-1}(\s{-a})$ are invariant by $\Autz(\U_a^{b,c})$ (see Lemma~\ref{Lem:InvariantUme}), we get $\varphi \Autz(\U_a^{b,c})\varphi^{-1}\subset\Autz(\U_a^{b+1,c+a})$. We then observe that $\varphi^{-1}$ is the blow-up of $l_{10}\subset \U_a^{b+1,c+a}$ followed by the contraction of the strict transform of the surface $\pi^{-1}(\s{-a})$. As $l_{10}$ is invariant by $\Autz(\U_a^{b+1,c+a})$, we obtain $\varphi \Autz(\U_a^{b,c})\varphi^{-1}=\Autz(\U_a^{b+1,c+a})$. This  achieves the proof of~\ref{LinkUme}.

We now consider the above construction in the case $b=0$ (which yields $k=0$ and $c=2$). The transition function $\nu_{a}^{b+1,c+a}$ still corresponds to the transition function of $\U_a^{b+1,c+a}=\U_a^{1,a+2}$, but the transition function $\nu_{a}^{b,c}$ corresponds to a transition function on a $\p^1\!$-bundle over $\F_a$ with numerical invariants $(a,b,c)$, which is therefore decomposable and isomorphic to $\FF_a^{b,c}=\FF_a^{0,2}$ (see Proposition~\ref{Prop:EquivABC}\ref{Existabc}). The maps $(\varphi_b)^{-1}$ on both charts then yield  an elementary link $\varphi\colon\U_a^{1,a+2}\dasharrow \FF_a^{0,2}$ centred at $l_{10}$, which is then $\Autz(\U_a^{1,a+2})$-equivariant. We moreover have $\varphi(\Autz(\U_a^{1,a+2}))\varphi^{-1}\subsetneq\Autz(\FF_a^{0,2})$ since $\Autz(\FF_a^{0,2})$ contains a torus of dimension $3$ (see Remark~\ref{FabcToric}), which is not the case for $\Autz(\U_a^{1,a+2})$ (as follows from Remark~\ref{rk aut vert Umemura bundles}). This then achieves  the proof of~\ref{SpecialUmeDec}.
\end{proof}

\begin{corollary}\label{Cor:MaxUmeBundlesFa}
Let $\U_a^{b,c}$ be an Umemura bundle. Then, $\Autz(\U_a^{b,c})$ is maximal if and only if one of the following holds:
\begin{enumerate}
\item
$a\ge 2$ and $c-ab<2$;
\item
$a=1$ and $c-ab<1$.
\end{enumerate}
 In this case, $\U_a^{b,c}$ is not stiff. More precisely, denoting by $k$ the integer such that $c=ak+2$, we find $0\le k<b$ and get a sequence of birational maps
\[\U_a^{b-k,2}\stackrel{\varphi_{-k}}{\dasharrow} \U_a^{b-k+1,2+a}\dots \stackrel{\varphi_{-1}}{\dasharrow}\U_a^{b,c}\stackrel{\varphi_0}{\dasharrow} \cdots \stackrel{\varphi_{n-1}}{\dasharrow} \U_a^{b+n,c+na} \stackrel{\varphi_n}{\dasharrow}\cdots\]such that $\varphi_n \Autz(\U_a^{b+n,2+na})\varphi_n^{-1}=\Autz(\U_a^{b+n+1,2+(n+1)a})$ for each $n\ge -k$. If $a\ge 2$, this gives all $\Autz(\U_a^{b,c})$-equivariant square birational maps from $\U_a^{b,c}$ to another $\p^1\!$-bundle. If $a=1$, we add the birational morphism $\U_a^{b-k,2}=\U_1^{b-k,2}\to \V_{b-k,1}$ of Lemma~$\ref{Lem:UMeF1P2}$.
\end{corollary}

\begin{proof}
We denote as usual by $k$ the integer satisfying $0\le k\le b$ such that $c=ak+2$ and find that $c-ab< 2\Leftrightarrow k<b$. 

If $c-ab\ge 2$, then $k=b$, so $c=ab+2$. We construct an $\Autz(\U_a^{b,c})$-equivariant birational map of $\p^1\!$-bundles $\U_a^{b,c}\dasharrow \U_a^{1,a+2}$ which is the composition of birational maps $\U_a^{b,c}=\U_a^{b,ab+2}\dasharrow \U_a^{b-1,a(b-1)+2}\dasharrow \dots\dasharrow \U_a^{1,a+2}$ (see Lemma~\ref{Lem:InvariantLinksUme}\ref{LinkUme}). Since $\Autz(\U_a^{1,a+2})$ is not maximal (see Lemma~\ref{Lem:InvariantLinksUme}\ref{SpecialUmeDec}), neither is $\Autz(\U_a^{b,c})$.

We now assume $c-ab< 2$, which means $k<b$. Lemma~\ref{Lem:InvariantLinksUme}\ref{LinkUme} yields a sequence of birational maps \[\U_a^{b-k,2}\stackrel{\varphi_{-k}}{\dasharrow} \U_a^{b-k+1,2+a}\dots \stackrel{\varphi_{-1}}{\dasharrow}\U_a^{b,c}\stackrel{\varphi_0}{\dasharrow} \cdots \stackrel{\varphi_{n-1}}{\dasharrow} \U_a^{b+n,c+na} \stackrel{\varphi_n}{\dasharrow}\cdots\]
such that $\varphi_n \Autz(\U_a^{b+n,2+na})\varphi_n^{-1}=\Autz(\U_a^{b+n+1,2+(n+1)a})$ for each $n\ge -k$. 

If $a\ge 2$, every $\Autz(\U_a^{b,c})$-equivariant birational map of $\p^1\!$-bundles $\U_a^{b,c}\dasharrow X$ is a composition of birational maps as these and of isomorphisms of $\p^1\!$-bundles since $l_{10}$ is not invariant by $\Autz(\U_a^{b-k,2})$ (see Lemma~$\ref{Lem:InvariantUme})$. We then get the result in this case ($a\ge 2$). 

It remains to do the case where $a=1$, which yields $c=k+2$ and $c-ab=k-b+2$. If $c-ab=1$, then $b-k=1$, which implies that $\Autz(\U_a^{b-k,2})$ is not maximal (see Lemma~\ref{Lem:UMeF1P2}) and thus neither is $\Autz(\U_a^{b,c})$. If $c-ab<1$, then $b-k>2$. We thus get a birational morphism $\psi\colon\U_a^{b-k,2}\to \V_{b-k}$ which satisfies $\psi\Autz(\U_1^{b-k,2})\psi^{-1}= \Autz(\V_{b-k})$ (see Lemma~\ref{Lem:UMeF1P2}). It remains to show that we cannot get any further link. At the level of surfaces, the only $\Aut(\F_1)$-equivariant birational maps $\F_1\dasharrow S$, where $S$ is a smooth projective surface, are isomorphisms or blow-ups $\F_1\to \p^2$ of a point of $\p^2$. We then only need to show that there is no square birational map $\V_{b-k}\dasharrow X$, to a $\pi$-bundle $X\to \p^2$, which is not a square isomorphism. We can reduce to the case of birational of $\p^1\!$-bundles (doing nothing on $\p^2$) and then use  Lemma~\ref{Lemm:DecElLinks}, and only need to observe that no curve of $\p^2$ is invariant by the action of $\Autz(\V_{b-k})$, which acts as $\Aut(\p^2,[0:1:0])$. The result then follows from Lemma~\ref{Lemm:DecElLinks}.
\end{proof}

\subsection{Links between Schwarzenberger bundles}

\begin{lemma}\label{Lem:SchwarzOnlyOneCurve}
Let $b\ge 1$, and let $\pi\colon \SS_b\to \p^2$ be the $b$-th Schwarzenberger $\p^1\!$-bundle. If $\,b=1$, no curve of $\SS_b$ is invariant by $\Autz(\SS_b)$. If $b\ge 2$, there is a unique curve invariant by $\Autz(\SS_b)$, which is given on the two charts of Lemma~$\ref{Lemm:TransSchwarz}$ by
\[\begin{array}{l}
\left\{\left([x_0 :x_1] ,([1:2t:t^2]) \right)\in \p^1\times U_0 \mid x_0+tx_1=0\right\},\\[1ex]
\left\{\left([x_0 :x_1] ,([t^2:2t:1]) \right)\in \p^1\times U_1 \mid x_0-tx_1=0\right\}.\end{array}\]
\end{lemma}

\begin{proof}
Let $\rho\colon\Autz(\SS_b)\to \Aut(\p^2)$ be the group homomorphism induced by $\pi$.
If $b=1$, then $\rho$ is surjective (see Lemma~\ref{Lemm:SchwarzJumpLines}\ref{AutoSchwarzb}), so there is no curve of $\SS_b$ which is invariant. Now suppose $b\ge 2$, in which case $\rho$ yields an isomorphism $\Autz(\SS_b)\iso \Aut(\p^2,\Gamma)\simeq \PGL_2$ (again by Lemma~\ref{Lemm:SchwarzJumpLines}\ref{AutoSchwarzb}). Every curve of $\SS_b$ is then contained in the invariant surface $X=\pi^{-1}(\Gamma)\subset \SS_b$. To understand the action of $\Autz(\SS_b)\simeq\PGL_2$ on~$X$, we use the corresponding action on the $\p^1\!$-bundle $\hat\pi\colon \TT_b=\SS_{b}\times_{\p^2} (\p^1\times \p^1)\to \p^1\times \p^1$, obtained by Lemma~\ref{Lemm:LiftSchwarzIsSchwarzHat}. The pull-back of $X$ on $\TT_b$ is the surface $\hat{X}=\hat\pi(\Delta)\subset \TT_b$ (where $\Delta\subset \p^1\times \p^1$ is the diagonal), isomorphic to $X$ via a $\PGL_2$-equivariant isomorphism. By Lemma~\ref{Lemm:LiftSchwarzIsSchwarzHat}, there is a unique curve in $\hat{X}$ invariant by $\Autz(\TT_b)\simeq \PGL_2$, which is sent onto the curve of $X$ given locally as above.
\end{proof}

\begin{lemma}\label{Lem:SchwarzLinks}
  Let $b\ge 2$, and let $\pi\colon \SS_b\to \p^2$ be the $b$-th Schwarzenberger $\p^1\!$-bundle. There is a birational involution $\varphi\colon \SS_b\dasharrow \SS_b$ such that $\varphi\Autz(\SS_b)\varphi^{-1}=\Autz(\SS_b)$. Moreover, every $\Autz(\SS_b)$-equivariant birational map of $\,\p^1\!$-bundle $\SS_b\dasharrow X$ is either an isomorphism of $\,\p^1\!$-bundles or a composition of $\varphi$ with an isomorphism of $\,\p^1\!$-bundles.
\end{lemma}

\begin{proof}
By Lemma~\ref{Lem:SchwarzOnlyOneCurve}, there is a unique curve $D\subset \SS_b$ which is invariant by $\Autz(\SS_b)$ and satisfies $\pi(D)=\Gamma$.  We consider the  birational involutions 
\[\begin{array}{rcccccc}
\varphi_0 \colon & \p^1 \times U_0 &\dasharrow &\p^1 \times U_0 \\
&\big([x_0:x_1] ,[1:u:v] \big) & \mapsto & \big([-ux_0-2vx_1:2x_0+ux_1],[1:u:v]\big),\\[1ex]
\varphi_1 \colon & \p^1 \times U_1 &\dasharrow &\p^1 \times U_1 \\
&\big([x_0:x_1] ,[v:u:1] \big) & \mapsto & \big([ux_0-2vx_1:2x_0-vx_1],[v:u:1]\big)\end{array}\]
which correspond locally to the blow-up of $D$, followed by the contraction of the strict transform of $\pi^{-1}(\Gamma)$, in the two charts (see Lemma~\ref{Lem:SchwarzOnlyOneCurve} for the equation of $D$). We then check that $\varphi_1\theta=\theta\varphi_0$, where $\theta\colon \p^1 \times U_0\dasharrow \p^1 \times U_1$ is the transition function of $\SS_b$ given in Lemma~\ref{Lemm:TransSchwarz}. This follows from the equality 
\[\begin{array}{ll}
\begin{bmatrix}  -s-t & 2 \\ -2st& s+t \end{bmatrix}\begin{bmatrix} s^{b}-t^{b} & st(s^{b-1}-t^{b-1}) \\
 s^{b+1}-t^{b+1}& st(s^{b}-t^{b})\end{bmatrix}= \begin{bmatrix} s^{b}+t^{b} & st(s^{b-1}+t^{b-1}) \\
 s^{b+1}+t^{b+1}& st(s^{b}+t^{b})\end{bmatrix}\\[2ex]
=\begin{bmatrix} s^{b}-t^{b} & st(s^{b-1}-t^{b-1}) \\
 s^{b+1}-t^{b+1}& st(s^{b}-t^{b})\end{bmatrix}\begin{bmatrix}  s+t & 2st \\ -2& -s-t \end{bmatrix}\end{array}\]
 and then yields  a birational map of $\p^1\!$-bundles $\varphi\colon \SS_b\dasharrow \SS_b$, given by the blow-up of $D$, followed by the contraction of the strict transform of $\pi^{-1}(\Gamma)$ onto $D$.

By Lemma~\ref{Lemm:DecElLinks}, every birational map of $\p^1\!$-bundles $\SS_b\to X$ which is not an isomorphism is a composition of $\varphi$ with an isomorphism of $\p^1\!$-bundles.
\end{proof}

\begin{corollary}\label{Cor:RigiditySch}
Let $b\ge 1$. Then, $\Autz(\SS_b)$ is maximal and $\SS_b$ is stiff. It is moreover superstiff if and only if $b=1$.
\end{corollary}

\proof
If $b=1$, the $\p^1\!$-bundle has no invariant curves (see Remark~\ref{PGL3:TP2}) (it is actually a homogeneous variety), and we conclude by Lemma~\ref{Lemm:DecElLinks}. If $b\ge2$, we apply Lemma~\ref{Lem:SchwarzLinks}.
\endproof

\subsection{Rigidity for decomposable bundles over \texorpdfstring{$\p^2$}{P2}}\label{RigidityDecP2}

\begin{lemma}\label{Lem:DecP2NoCurve}
Let $b\ge 0$, and let $\pi\colon \PP_b\to \p^2$ be a decomposable $\p^1\!$-bundle. 
\begin{enumerate}
\item\label{PPbNoInvCurve}
$\PP_b$ does not contain any $\Autz(\SS_b)$-invariant curve.
\item \label{AutPPbSuperstiff}
 $\Autz(\PP_b)$ is maximal and $\PP_b$ is superstiff.
\end{enumerate}
\end{lemma}

\begin{proof}
Assertion~\ref{PPbNoInvCurve} follows from the fact that $\Autz(\PP_b)$ surjects onto $\Aut(\p^2)$ (see Lemma~\ref{lemm:surjectiveDecP2}).

\ref{AutPPbSuperstiff} Let $\varphi\colon \PP_b\dasharrow X$ be an $\Autz(\PP_b)$-square birational map of $\p^1\!$-bundles, over $\eta\colon \p^2\dasharrow X$ (where $X$ is a smooth projective rational surface). We want to show that $\varphi$ is a square isomorphism. The action of $\Autz(\PP_b)$ on $\p^2$ being transitive, the birational map $\eta$ is an isomorphism, so we can assume  $X=\p^2$, and that $\eta$ is the identity. We then apply Lemma~\ref{Lemm:DecElLinks} to get that $\varphi$ is an isomorphism.
\end{proof}

\subsection{Last step}\label{LastStep}
We have all the ingredients needed to prove the main theorems of this paper stated in the introduction. 

\begin{proof}[Proof of Theorems~\ref{Thm:MainA} and~\ref{Thm:MainB}]
These are simply a consequence of Proposition~\ref{Prop:FourCases}, together with Corollaries~\ref{Cor:MaxDecBundlesFa},~\ref{Cor:MaxUmeBundlesFa},~\ref{Cor:RigiditySch} and Lemma~\ref{Lem:DecP2NoCurve}.
\end{proof}

\begin{remark}\label{Rmk:DifferentMaximlities}
Our original motivation was to study the maximal connected algebraic subgroups of the Cremona group $\Bir(\P^3)$. These are studied in the article~\cite{second_paper}. Most of the families appearing in our classification  in fact give maximal  connected algebraic subgroups of the Cremona group, even if some sporadic cases (like $\SS_2$ and $\PP_1$)  disappear as they are conjugate to bigger subgroups, with a birational map which does not preserve any $\p^1\!$-bundle structure.
\end{remark}

\begin{remark}
If we assume the characteristic of the algebraically closed field $\k$ to be positive, the strategy to prove Theorems~\ref{Thm:MainA} and~\ref{Thm:MainB} should be analogous, but some of the results that we use in characteristic zero are no longer valid in positive characteristic; see \textit{e.g.}\ Remarks~\ref{Rem:PGL2positive} and~\ref{Rem:CharP}. As a consequence, it seems that new $\p^1\!$-bundles $X \to S$ analogous to Umemura bundles and Schwarzenberger bundles, and such that $\Autz(X)$ is maximal, could show up in the classification. The positive characteristic case will be studied by the authors in a future work.
\end{remark}


\begin{thebibliography}{Ume82b+++}

\bibitem[ABM12]{ABV}
M.~Aprodu, V.~Br{\^\i}nz{\u a}nescu, and M.~Marchitan, 
\emph{Rank-two vector bundles on {H}irzebruch surfaces}, 
Cent.\ Eur.\ J.~Math.\ \textbf{10} (2012), no.~4, 1321--1330. 

\bibitem[Art86]{Art86}
M.~Artin, 
\emph{Lipman's proof of resolution of singularities for surfaces},
In: \emph{Arithmetic geometry} (Storrs, Conn., 1984), pp.~267--287, Springer, New York, 1986. 

\bibitem[Bea96]{Bea96}
  A.~Beauville, \emph{Complex algebraic surfaces}, Translated from French by R.~Barlow, with assistance from N.\,I.~Shepherd-Barron and M.~Reid, 2nd ed., London Math.\ Soc.\ Stud.\ Texts, vol.~34, Cambridge Univ.\ Press, Cambridge, 1996.


\bibitem[Bla09]{Bla09}
J.~Blanc, \emph{Sous-groupes alg\'ebriques du groupe de {C}remona},  
Transform.\ Groups, \textbf{14} (2009), no.~2, 249--285. 

\bibitem[BFT21]{second_paper}
J.~Blanc, A.~Fanelli, and R.~Terpereau, 
\emph{Connected Algebraic Groups Acting on three-dimensional Mori Fibrations},  Int.\ Math.\ Res.\ Not.\ IMRN (2021), published online on October 20, 2021, \url{https://doi.org/10.1093/imrn/rnab293}. 

\bibitem[Bla56]{Bla56}
A.~Blanchard, 
\emph{Sur les vari\'et\'es analytiques complexes}, 
Ann.\ Sci.\ \'{E}cole Norm.\ Sup.~(3) \textbf{73} (1956), 157--202.

\bibitem[BSU13]{BSU13}
M.~Brion, P.~Samuel, and V.~Uma, 
\emph{Lectures on the structure of algebraic groups and geometric
  applications}, CMI Lecture Series in Math., vol.~1, Hindustan Book Agency, New Delhi; Chennai Math.\ Inst.\ (CMI), Chennai, 2013.


\bibitem[Cor00]{cor00}
A.~Corti,
\emph{Singularities of linear systems and {$3$}-fold birational geometry}, 
In: \emph{Explicit birational geometry of 3-folds}, pp.~259--312, London Math.\ Soc.\ Lecture Note Ser.\ vol.~281, Cambridge Univ.\ Press, Cambridge, 2000.

\bibitem[Dem65]{SGA3Ep24}
M.~Demazure, 
\emph{Automorphismes des groupes r\'eductifs}, In: 
\emph{Sch\'emas en groupes (SGA3) (S\'em.\ G\'eom\'etrie
  Alg\'ebrique, Inst.\ Hautes \'Etudes Sci., 1962/1964)}, Fasc.~7, Expos\'e~XXIV, Inst.\ Hautes \'Etudes Sci., Paris, 1965. 

\bibitem[Dem77]{Dem77}
\bysame, 
\emph{Automorphismes et d\'eformations des vari\'et\'es de {B}orel}, 
Invent.\ Math., \textbf{39} (1977), no.~2, 179--186.

\bibitem[Enr93]{Enr1893}
F.~Enriques,
\emph{Sui gruppi continui di trasformazioni cremoniane nel piano}, 
Rom.\ Acc.\ L.~Rend.\ \textbf{5} (1893), no.~2(1), 468--473.

\bibitem[EF98]{EF1898}
F.~Enriques and G.~Fano, 
\emph{Sui gruppi continui di trasformazioni Cremoniane dello spazio}, 
Annali di Mat.\ \textbf{2} (1898), no.~26, 59--98.

\bibitem[Har77]{Har77}
R.~Hartshorne, 
{\em Algebraic geometry},
Grad.\ Texts in Math., vol.~52,
Springer-Verlag, New York~-- Heidelberg, 1977.

\bibitem[Har80]{Har80}
\bysame,  
\emph{Stable reflexive sheaves},
Math.\ Ann.\ \textbf{254} (1980), no.~2, 121--176.

\bibitem[Hum75]{Hum75}
J.\,E.~Humphreys, 
\emph{Linear algebraic groups},
Grad.\ Texts in Math., vol.~21,
Springer-Verlag, New York~-- Heidelberg, 1975.

\bibitem[IP99]{IP99}
V.\,A.~Iskovskikh and Y.\,G.~Prokhorov, 
\emph{Fano varieties}, 
In:  \emph{Algebraic geometry, {V}}, pp.~1--247, Encyclopedia Math.\ Sci., Springer, Berlin, 1999.

\bibitem[Lan02]{Lang}
S.~Lang, 
\emph{Algebra}, 3rd ed., 
Grad.\ Texts in Math., vol.~211,
Springer-Verlag, New York, 2002.

\bibitem[Lau18]{Lau}
B.~Laurent, 
\emph{Courbes et surfaces presque homog\`enes},
PhD thesis, Universit\'e Grenoble Alpes, 2018. Available from  
  \url{https://tel.archives-ouvertes.fr/tel-01999851}. 

\bibitem[Mel02]{Mel02}
M.~Mella, 
\emph{\#-minimal models of uniruled 3-folds}, 
Math.\ Z., \textbf{242} (2002), no.~4, 687--707.

\bibitem[MO67]{MO67}
H.~Matsumura and F.~Oort, 
\emph{Representability of group functors, and automorphisms of algebraic schemes}, 
Invent.\ Math.\ \textbf{4} (1967), 1--25.

\bibitem[MU83]{MU83}
S.~Mukai and H.~Umemura, 
\emph{Minimal rational threefolds}, 
In:  \emph{Algebraic geometry} ({T}okyo/{K}yoto, 1982), pp.~490--518, 
Lecture Notes in Math., vol.~1016, Springer, Berlin, 1983.

\bibitem[Nak89]{Nak89}
T.~Nakano, 
\emph{On equivariant completions of {$3$}-dimensional homogeneous spaces of
  {${\rm SL}(2,{\bf C})$}}, 
Japan.\ J.~Math.\ (N.S.) \textbf{15} (1989), no.~2, 221--273.

\bibitem[OSS11]{OSS11}
C.~Okonek, M.~Schneider, and H.~Spindler, 
\emph{Vector bundles on complex projective spaces}, 
Mod.\ Birkh\"auser Class., Birkh\"auser/Springer Basel AG, Basel,
  2011. Corrected reprint of the 1988 edition, with an appendix by S.\,I.~Gelfand.

\bibitem[Puk13]{Puk13}
A.~Pukhlikov, 
\emph{Birationally rigid varieties},  Math.\
  Surveys  Monogr., vol.~190, 
Amer.\ Math.\ Soc., Providence, RI, 2013.

\bibitem[Sch61]{Sch61}
R.\,L.\,E.~Schwarzenberger, 
\emph{Vector bundles on the projective plane}, 
Proc.\ London Math.\ Soc.~(3), \textbf{11} (1961), 623--640.

\bibitem[Ume80]{Ume80}
H.~Umemura, 
\emph{Sur les sous-groupes alg\'ebriques primitifs du groupe de {C}remona
  \`a trois variables}, 
Nagoya Math.~J.\  \textbf{79} (1980), 47--67.

\bibitem[Ume82a]{Ume82a}
\bysame, 
\emph{Maximal algebraic subgroups of the Cremona group of three
  variables. {I}mprimitive algebraic subgroups of exceptional type}, 
Nagoya Math.\ J.\ \textbf{87} (1982), 59--78.

\bibitem[Ume82b]{Ume82b}
\bysame, 
\emph{On the maximal connected algebraic subgroups of the {C}remona group.
  {I}}, 
Nagoya Math.\ J.\ \textbf{88} (1982), 213--246.

\bibitem[Ume85]{Ume85}
\bysame, 
\emph{On the maximal connected algebraic subgroups of the {C}remona group.
  {II}}, 
In: \emph{Algebraic groups and related topics} ({K}yoto/{N}agoya,
1983), pp.~349--436, Adv.\ Stud.\ Pure Math., vol.~6, 
  North-Holland, Amsterdam, 1985.

\bibitem[Ume88]{Ume88}
\bysame, 
\emph{Minimal rational threefolds. {II}}, 
Nagoya Math.\ J.\ \textbf{110} (1988), 15--80.

\bibitem[Val00]{Val00}
J.~Vall{\`e}s, 
\emph{Fibr\'es de {S}chwarzenberger et coniques de droites sauteuses}, 
Bull.\ Soc.\ Math.\ France \textbf{128} (2000), no.~3, 433--449.

\bibitem[VdV72]{VdV72}
A.~Van~de Ven, 
\emph{On uniform vector bundles}, 
Math.\ Ann.\ \textbf{195} (1972), 245--248.

\end{thebibliography}
\end{document}